\documentclass[12pt]{amsart}
\usepackage{amssymb}
\usepackage{longtable}
\newcommand{\K}{\mathbb{C}}
\newcommand{\F}{\operatorname{F}}
\newcommand{\spa}{\mathfrak{sp}}
\newcommand{\Sp}{\operatorname{Sp}}
\newcommand{\g}{\mathfrak{g}}
\newcommand{\Spec}{\operatorname{Spec}}
\newcommand{\Z}{\mathbb{Z}}
\newcommand{\A}{\mathcal{A}}
\newcommand{\gr}{\operatorname{gr}}
\newcommand{\Dcal}{\mathcal{D}}
\newcommand{\Str}{\mathcal{O}}
\newcommand{\h}{\mathfrak{h}}
\newcommand{\Aut}{\operatorname{Aut}}
\newcommand{\Der}{\operatorname{Der}}
\newcommand{\Symp}{\operatorname{Symp}}
\newcommand\Ham{\operatorname{H}}
\newcommand\Quant{\mathcal{Q}}

\newcommand\M{\mathcal{M}}
\newcommand\Loc{\operatorname{Loc}}
\newcommand\Ext{\operatorname{Ext}}
\newcommand\G{\operatorname{G}}
\newcommand\Lie{\operatorname{Lie}}
\newcommand\Per{\operatorname{Per}}

\newcommand\GL{\operatorname{GL}}
\newcommand\gl{\mathfrak{gl}}
\newcommand\Span{\operatorname{Span}}
\newcommand\red{/\!\!/\!\!/}
\newcommand\I{\mathcal{I}}
\newcommand\m{\mathfrak{m}}
\newcommand\J{\mathcal{J}}

\newcommand\End{\operatorname{End}}
\newcommand\ad{\operatorname{ad}}
\newcommand\Orb{\mathbb{O}}

\renewcommand\sl{\mathfrak{sl}}
\newcommand\Walg{\mathcal{W}}
\newcommand\Hom{\operatorname{Hom}}
\newcommand\Nil{\mathcal{N}}
\newcommand\p{\mathfrak{p}}
\renewcommand\a{\mathfrak{a}}

\newcommand\quo{/\!/}
\newcommand\vv{\mathbf{v}}
\newcommand\ww{\mathbf{d}}
\newcommand\SL{\operatorname{SL}}
\newcommand\z{\mathfrak{z}}
\newcommand\W{\mathbb{A}}
\newcommand\Pic{\operatorname{Pic}}
\newcommand\Tt{\mathbb{T}}
\newcommand\Ss{\mathbb{S}}
\newcommand\Aa{\mathbb{A}}
\newcommand\Bb{\mathbb{B}}
\newcommand\id{\operatorname{id}}

\newcommand\grad{\operatorname{grad}}
\newcommand\tr{\operatorname{tr}}
\newcommand\X{\mathcal{X}}
\newcommand\Sl{\operatorname{S}}
\newcommand\XS{\operatorname{X}}
\newcommand\D{\mathrm{D}}
\newcommand\Gr{\mathcal{G}}
\newcommand\rk{\operatorname{rk}}
\newcommand\param{\mathfrak{c}}
\newcommand\Halg{\mathbf{H}}
\newcommand\Pro{\mathcal{P}}
\newcommand\diag{\operatorname{diag}}
\newcommand\skewperp{\angle}
\newcommand\Afr{\mathfrak{A}}
\newcommand\ZZ{\mathbb{Z}}
\newtheorem{Thm}{Theorem}[subsection]
\newtheorem{Prop}[Thm]{Proposition}
\newtheorem{Cor}[Thm]{Corollary}
\newtheorem{Lem}[Thm]{Lemma}
\theoremstyle{definition}
\newtheorem{Ex}[Thm]{Example}
\newtheorem{defi}[Thm]{Definition}
\newtheorem{Rem}[Thm]{Remark}

\numberwithin{equation}{section}
\title{Isomorphisms of quantizations via quantization of resolutions}
\author{Ivan Losev}
\thanks{Supported by the NSF grant DMS-0900907}
\thanks{Address: MIT, Dept. of Math., 77 Massachusetts Avenue, Cambridge MA02139, USA}
\thanks{E-mail: ivanlosev@math.mit.edu}
\begin{document}
\begin{abstract}
In this paper we prove the existence of isomorphisms between certain non-commutative algebras that are interesting
from the representation theoretic perspective and arise as quantizations of certain Poisson algebras. We show that
quantizations of Kleinian singularities obtained by three different constructions are isomorphic to each other. The constructions are via symplectic reflection algebras, quantum Hamiltonian reduction, and W-algebras.
Next, we prove that parabolic W-algebras in type A are isomorphic to quantum Hamiltonian reductions
associated to quivers of type A. Finally, we show that the symplectic reflection algebras
for wreath-products of the symmetric group and a Kleinian group are isomorphic to certain quantum
Hamiltonian reductions. Our results involving W-algebras are new, while for those dealing with symplectic
reflection algebras we just find  new proofs. A key ingredient in our proofs is the study of quantizations
of symplectic resolutions of appropriate Poisson varieties.
\end{abstract}
\maketitle
\tableofcontents
\section{Introduction}
In this paper we prove the existence of isomorphisms between quantizations of certain graded Poisson algebras.
We start by giving all necessary general definitions.

A base field will be the field $\K$ of complex numbers.
Let $A$ be a commutative associative unital algebra over $\K$ equipped with a Poisson bracket
$\{\cdot,\cdot\}$. We suppose that $A$ is graded, $A=\bigoplus_{i\in \Z_{\geqslant 0}} A_i$,
and that the Poisson bracket has degree $-d$ with $d>0$: $\{A_i,A_j\}\subset A_{i+j-d}$
(in most examples, $d=2$).

Now let $\A$ be an associative unital algebra equipped with an increasing exhaustive filtration $\F_i\A, i\geqslant 0$.
Suppose that $[\F_i \A, \F_j \A]\subset \F_{i+j-d}\A$. Then the associated graded algebra
$\gr \A$ is commutative. Moreover, it comes equipped with a natural Poisson bracket of degree $-d$
induced by taking the top degree term in the Lie bracket of elements of $\A$.
We say that $\A$ is a {\it quantization} of  $A$ if the graded Poisson algebras $\gr\A$ and $A$ are isomorphic.

Perhaps, the easiest example is as follows. Let $V$ be a vector space equipped with a symplectic
(=skew-symmetric and non-degenerate) form $\omega$. Then $SV$ becomes a Poisson algebra with bracket uniquely
determined from $\{u,v\}=\omega(u,v), u,v\in V$. This algebra is graded in the standard way, the Poisson bracket
has degree $-2$. It admits a (unique, in fact) quantization called the {\it Weyl algebra} $\W(V)$ of $V$.
This is  the quotient of the tensor algebra $T(V)$ by the relations $u\otimes v-v\otimes u-\omega(u,v)=0$.
The algebra $\W(V)$ is filtered by the order of a monomial.

In general, a graded Poisson algebra $A$ admits many non-isomorphic quantizations and it is difficult to
describe them. Many constructions (such as a Hamiltonian reduction, see Subsection \ref{SUBSECTION_class_red}
for the definition) produce quantizations depending on a parameter (for a quantum Hamiltonian reduction a parameter
will be a character of a Lie algebra used to reduce). Basically, this paper concentrates on the following problem:
suppose we have two families of quantizations of $A$ parametrized by two (isomorphic) parameter spaces. Given a parameter for the first family determine a parameter for the second one producing an isomorphic quantization.

Of course, this problem is too vague and too general to approach. We will deal with some cases
that are interesting from the representation theoretic perspective. Also in all of the cases we
are going to consider that algebra $A$ will have some nice algebro-geometric properties.
 Namely, we will assume that $A$ is finitely generated and integral, and that the Poisson variety $X_0:=\Spec(A)$
admits a symplectic resolution of singularities, see Subsection \ref{SUBSECTION_sympl_general} for a definition.
The existence of a symplectic resolution is crucial for us, because we will approach quantizations
of $A$ via the quantization of  the structure sheaf or more general vector bundles on the resolution.

%


For different classes of varieties $X_0$ we have various classes of quantizations. Let us describe them briefly.

The first class of varieties comes from Lie theory.
A {\it Slodowy slice} $\Sl(=\Sl(\Orb))$ in a semisimple Lie algebra $\g$ is a transversal slice to a nilpotent orbit
$\Orb$. The algebra $\K[\Sl]$ comes equipped with a natural grading and a natural Poisson bracket of degree $-2$.
Choose a parabolic subgroup $P\subset G$, where $G$ is a connected algebraic group
corresponding to $\g$. Then $X_0:=\Sl\cap G\p^\perp$ is a Poisson subvariety
of $\Sl$ (a parabolic Slodowy slice). The algebra $\K[\Sl]$ admits a natural quantization -- a finite W-algebra $\Walg$. We will recall two definitions of $\Walg$ in Subsection \ref{SUBSECTION_Walg_def}. So one can hope
to construct quantizations of $X_0$ as quotients of $\Walg$. This can be done under some restrictions on $\g,\Orb$
and $P$, the corresponding quotients ({\it parabolic
W-algebras}) are defined in Subsection \ref{W_parab}.  In this paper we are concerned with two cases:

1) $\g$ is of type $A,D$ or $E$, and $\Orb$ is a subprincipal nilpotent orbit.

2) $\g$ is of type $A$, $\Orb$ is any nilpotent orbit, and $P$ is a certain parabolic
subgroup (of course, we want $\Sl\cap G\p^\perp$ to be non-empty).

Let us proceed to the second class of varieties.  Now suppose that $X_0$ is obtained by Hamiltonian reduction of a symplectic vector space by a reductive group, the definition is  recalled in Subsection \ref{SUBSECTION_class_red}. One can hope to get a quantization of $\K[X_0]$ by replacing a classical
Hamiltonian reduction  with a quantum one, Subsection \ref{SUBSECTION_quant_red_alg}, and this hope is fulfilled
under some technical assumptions on the action. One special case of reductive group actions
on symplectic vector spaces comes from quivers. The corresponding classical reductions
are Nakajima's quiver varieties, see, for example, \cite{Nakajima}. These varieties together with
the corresponding quantizations  are recalled in Subsection
\ref{SUBSECTION_quiver_var}. Each quiver variety
is constructed from a quiver together with two dimension vectors.  In this paper we are concerned with two classes
of quivers and of dimension vectors:

1) Finite Dynkin quivers of type $A$ and  dimension vectors subject to certain inequalities.

2) Affine Dynkin quivers and  dimension vectors as in \cite{EGGO}.

Finally, the third class of varieties we consider is as follows. Let $V$ be a symplectic
 vector space, and  $\Gr$ be a finite subgroup of $\Sp(V)$.    Consider the symplectic quotient singularity $X_0=V^*/\Gr$. Here one can produce quantizations of $\K[X_0]$ by studying deformations of the "orbifold"
algebra $SV\#\Gr$, that is the semidirect product of the symmetric algebra $SV$ and
the group algebra $\K\Gr$. These deformations are known as symplectic reflection algebras (SRA)
and appeared first in this context in \cite{CBH} in the special case when $V=\K^2$
and in \cite{EG} in the general case. For a deformation of $\K[X_0]$ one takes a
so called {\it spherical subalgebra} $e\Halg e$, where $e$ is the trivial idempotent in
$\Gr$. For details see Subsection \ref{SUBSECTION_SRA_def}. We are interested in the case
when $V=\K^{2n}$ and $\Gr$ is the wreath-product $\Gamma_n:=S_n\ltimes \Gamma^n$, where
$\Gamma$ is a Kleinian group, i.e., a finite subgroup of $\SL_2(\K)$.

In some cases one Poisson variety can be obtained using two different constructions.
For example, the Kleinian singularity $\K^2/\Gamma$ can be obtained both as a quiver
variety  and as the intersection of the Slodowy slice with the null-cone in an appropriate
Lie algebra $\g$.
In this case all three families of quantizations are the same, see  Theorems \ref{Thm:Klein_W},
\ref{Thm:iso_Kleinian2} for the precise statements including the explicit correspondence
between the parameters.

Parabolic Slodowy varieties for $\g=\sl_N$ can be realized as quiver varieties for the Dynkin quiver of type
$A$, this was first discovered by Nakajima, \cite{Nakajima}. The corresponding quantizations
are also the same, see Theorem \ref{Thm:type_A}.

Finally, the quotient singularity $\K^{2n}/\Gamma_n$ for $n>1$ can also be constructed as a quiver variety
(but no Slodowy type construction is known, in general).  Again, the two families of quantizations are the same,
see Theorem \ref{Thm_SRA_iso}. We remark that the results relating the SRA quantizations with quantum
Hamiltonian reductions were known before, see Subsection \ref{SUBSECTION_SRA_main} for references.
However, our proofs here are new.

%
\subsection{Content of the paper}
Our arguments are based on the study of partially non-commutative deformations of
the symplectic resolutions of the Poisson varieties of interest. Such deformations
 were studied in \cite{KV} (purely commutative deformations) and in \cite{BK1}
(quantizations=non-commutative deformations). The results of these papers (with appropriate
modifications we need) are explained in Section \ref{SECTION_deformation}.
One of very pleasant features of resolutions is that their deformations
are very easy to control, they are basically parameterized by the second cohomology
via a certain {\it period map}.
It is this feature that was a reason for us to consider deformations of the resolutions.
In Subsection \ref{SUBSECTION_def_classic} we explain results of \cite{KV} including the existence
of a universal commutative formal deformation. In Subsection \ref{SUBSECTION_Basic_def}
we explain some general definitions and constructions including various compatibilities
with group actions, in particular, notions of {\it graded} and {\it even}
quantizations.  In Subsection \ref{SUBSECTION_non_comm_period} we recall the main construction
of \cite{BK1}, the {\it non-commutative period map}, that classifies quantizations of symplectic
varieties  under  certain vanishing assumptions on a variety.
We also study the compatibility of
this map with the group actions introduced in Subsection \ref{SUBSECTION_Basic_def}.
The compatibility results seem to be pretty standard but we could not find a reference.

All Poisson varieties we consider and most of their quantizations are obtained by
Hamiltonian reduction. In Section \ref{SECTION_red} we recall general definitions and
results regarding both classical and quantum Hamiltonian reduction. This section
again basically does not contain new results. We recall the definition
of a classical Hamiltonian reduction both for Poisson algebras and for Poisson varieties
in Subsection \ref{SUBSECTION_class_red}. In Subsection \ref{SUBSECTION_DH} we establish
an algebro-geometric version of the Duistermaat-Heckman theorem, \cite{DH}. For us this theorem
is a recipe to compute the period map for Hamiltonian reductions. The reason why we need this result
is as follows. If a symplectic variety is obtained by Hamiltonian reduction, then one can produce
its (formal) deformations also using the Hamiltonian reduction. The Duistermaat-Heckman theorem
can be viewed as a tool to identify two commutative
deformations related to different constructions of a given variety by Hamiltonian reduction. In the next two subsections, \ref{SUBSECTION_quant_red_alg},\ref{SUBSECTION_quant_red_sheaf}, we recall generalities on quantum Hamiltonian reduction.

Our goal in Section \ref{SECTION_sympl_resol} is to recall several examples of symplectic resolutions
of singularities as well as isomorphisms between them. There are no new results in this section either.
In Subsection \ref{SUBSECTION_sympl_general}
we recall the general definition and some general vanishing properties for them.
In Subsection \ref{SUBSECTION_quiver_var}
we provide necessary information on Nakajima quiver varieties, including the definitions
of both affine and non-affine quiver varieties. Also we recall sufficient conditions to obtain symplectic
resolutions using the quiver variety construction, as well as sufficient conditions for a quantum
Hamiltonian reduction to provide a quantization of an affine quiver variety. Then we recall the Slodowy
slices and their ramifications, Subsection \ref{SUBSECTION_Slodowy}. After this, in Subsection \ref{SUBSECTION_resolution}, we  recall different facts about the quotients $\K^{2n}/\Gamma_n$ and their resolutions including the construction via quiver varieties and the existence of some special vector bundles
("weakly Procesi bundles"). In Subsection \ref{SUBSECTION_Klein_classical} we consider the $n=1$ case in more
detail recalling the constructions of the Kleinian singularities and of their minimal (=symplectic)
resolutions via the Slodowy varieties. One of important results here
is the comparison between two families of line bundles on the resolution: one coming from the Slodowy
construction and other from the quiver varieties construction. This comparison is one of ingredients
in constructing a correspondence between the quantization parameters in Theorem \ref{Thm:Klein_W}.
Finally, in Subsection \ref{SUBSECTION_A_quiver_vs_Slod} we recall
the isomorphisms between parabolic Slodowy varieties in type $A$ and quiver varieties of type $A$
due to Maffei, \cite{Maffei}, and establish some easy properties of these isomorphisms.

A common feature of our main results is that one of the families of quantizations we consider always
comes from quiver varieties, while the other is obtained by a different (W-algebra or symplectic
reflection algebra) construction. In Section \ref{SECTION_Walg} we establish results involving W-algebras.
In Subsection \ref{SUBSECTION_Walg_def} we recall the definitions of W-algebras due to the author,
\cite{Wquant} and Premet, \cite{Premet1}. Then in Subsection \ref{W_parab} we introduce certain ramifications
of W-algebras, the {\it parabolic} W-algebras, and relate them to quantizations of appropriate parabolic Slodowy varieties. Next, in Subsection \ref{SUBSECTION_W_main}, we state two main results relating parabolic W-algebras
to quantum Hamiltonian reductions of quiver varieties: Theorem \ref{Thm:Klein_W} (Kleinian case)
and Theorem \ref{Thm:type_A} (type $A$ case).

The strategy of the proofs of these theorem is the same. First, we have two constructions
of commutative formal deformations of the resolutions, both obtained via Hamiltonian reduction.
We use the Duistermaat-Heckman theorem and the results on the relation between appropriate
line bundles to establish an isomorphism between the two commutative deformations.
Then  the algebras, for which we need to establish an isomorphism,
are, roughly speaking, the algebras of global sections of  quantizations of the two
 commutative deformations. The commutative deformations are again  obtained by two different quantum
Hamiltonian reductions. Isomorphisms of interest
follow from the claim that the corresponding quantizations are isomorphic. Thanks to
results of Subsection \ref{SUBSECTION_non_comm_period}, it is enough to show that both quantizations
are {\it canonical} in the sense of Bezrukavnikov and Kaledin or, equivalently, even.
So we need to understand when a quantization obtained by a quantum Hamiltonian reduction is even.
This is done in Subsection \ref{SUBSECTION_even_red}. There we obtain a general result answering this
question, Theorem \ref{Thm:even_quant}, which seems to be of independent interest.
After this completing the proofs of Theorems \ref{Thm:Klein_W}, \ref{Thm:type_A} is pretty
easy, see Subsection \ref{SUBSECTION_W_proofs}.

The SRA side is studied in Section \ref{SECTION_SRA}. The definition of the symplectic reflection
algebras is recalled in Subsection \ref{SUBSECTION_SRA_def}. The main result, Theorem \ref{Thm_SRA_iso},
comparing the family of quantizations of $\K^{2n}/\Gamma_n$ coming from an SRA with that obtained
by quantum Hamiltonian reduction is stated in Subsection \ref{SUBSECTION_SRA_main}. This theorem
is proved in the rest of Section \ref{SECTION_SRA}.  The scheme of this proof
will be explained in the end of Subsection \ref{SUBSECTION_SRA_main}.

\subsection{Conventions and notation}\label{SUBSECTION_conventions} Let us list some conventions we use in the paper.

{\it Quivers}. Recall that by a quiver one means an oriented graph possibly with multiple arrows or loops.
More precisely, a quiver $Q$ consists of two sets, a set $Q_0$ of vertices and a set
$Q_1$ of arrows, and two maps $t$ (tail) and $h$ (head) from $Q_1$ to $Q_0$.

Associated to a quiver is the vector space $\K^{Q_0}$ with the "scalar product" $\alpha\cdot\beta=\sum_{i\in Q_0}
\alpha_i\beta_i$. Let $\epsilon_i, i\in Q_0,$ denote the tautological basis in $\K^{Q_0}$.

{\it Quotients.} In this paper we deal with several kinds of quotients. Geometric quotients for an action of
an algebraic group $G$ on a scheme $X$
(e.g., quotients for finite group actions or quotients by free group actions) are denoted
by $X/G$. The categorical quotient of  an affine variety $X$ by an action of a reductive group $G$
is denoted by $X\quo G$. The GIT quotient of $X$ corresponding to a stability condition
$\theta$ will be denoted by $X\quo^\theta G$. The set of $\theta$-semistable points in
$X$ will be denoted by $X^{\theta,ss}$.

Finally, "$\red$" means a (classical or quantum) Hamiltonian
reduction (either of an algebra or of a variety).

{\it Sheaves.} We usually denote sheaves as $\mathcal{A}_X$, where $X$ is a (formal) scheme,
where the sheaf lives. For a sheaf $\mathcal{A}_X$ by $\mathcal{A}(X)$ we denote its global
sections.

\begin{longtable}{p{2.5cm} p{13.5cm}}
$\widehat{\otimes}$&completed tensor product of complete topological vector spaces/modules.\\
$(X)$& the two-sided ideal in an associative algebra generated by a subset $X$.\\
$\W_h(V)$& the homogenized Weyl algebra of a symplectic vector space $V$.\\
$\W_{2n,h}$&$:=\W_h(\K^{2n})$.\\
$\Aut(\mathcal{A})$& the automorphism group of an algebra (or a sheaf of algebras) $\mathcal{A}$.\\
$\Dcal_{h,X}$& the sheaf of "homogenized" differential operators on a smooth scheme $X$.\\
$\Der(A)$& the Lie algebra of derivations of an algebra $A$.\\
$G_x$& the stabilizer of $x$ in $G$.\\
$\gr \A$& the associated graded vector space of a filtered
vector space $\A$.\\
$H^i_{DR}(X)$&$i$-th De Rham cohomology of a smooth (formal) scheme $X$.\\
$\K[X]$& the algebra of regular functions on a (formal) scheme $X$.\\
$\Str_X$& the structure sheaf of a (formal) scheme $X$.\\
$R\Gamma$& the group algebra of a group $\Gamma$ with coefficients in a ring $R$.\\
$SV$& the symmetric algebra of a vector space $V$.\\
$\mathcal{T}_{X/S}$& the relative tangent bundle of a (formal) scheme $X/S$.\\
$T^*X$& the cotangent bundle of a smooth scheme $X$.\\
$U_h(\g)$& homogenized universal enveloping algebra of a Lie algebra $\g$, i.e.,
the quotient of $T(\g)[h]$ by the relations $\xi\otimes \eta-\eta\otimes \xi- h[\xi,\eta]=0$.\\
$U^\perp$& the annihilator of a subspace $U\subset V$ in $V^*$.\\
$X^G$& $G$-invariants in a $G$-set $X$.\\
$X^{\wedge_Y}$& the completion of a scheme $X$ along a subscheme $Y$.\\
$\iota_\xi$& the contraction with a vector field $\xi$.\\
$\Omega^i_{X/S}$& the bundle of relative $i$-forms on a (formal) scheme
$X/S$.
\end{longtable}

{\bf Acknowledgements.} This research was initiated by a question that
 A. Premet asked me in  the beginning of 2008.
Numerous discussions with P. Etingof were of great importance for
this project. Also I would like to thanks D. Maulik and T. Schedler for discussing various aspects
of the paper.

\section{Deformations of symplectic schemes}\label{SECTION_deformation}
\subsection{Commutative deformations}\label{SUBSECTION_def_classic}
In this subsection we will discuss the existence of a universal (commutative)  deformation of a symplectic
variety $X_0$ over $\K$. Basically, all results here are taken from \cite{KV}.
The deformations are  controlled by a certain {\it period map} which has a very explicit description.
Throughout the section we suppose that $H^1(X_0,\Str_{X_0})=H^2(X_0,\Str_{X_0})=\{0\}$, although
most of the results we provide hold in greater generality.
Let $\Omega_0$ denote the symplectic form on $X_0$.

Let $S$  be the $n$-dimensional formal polydisc, i.e., the completion $(\K^n)^{\wedge_0}$ of the
affine space $\K^n$ at 0. Let $X$  be a formal scheme over $S$ with zero fiber $X_0$. More precisely, we suppose
that $X$ is a formal scheme equipped with a morphism  $\rho:X\rightarrow S$  such that $\rho^{-1}(0)=X_0$, and  $X$ coincides with its completion along $X_0$. In this case we will say that $X$ is a {\it formal deformation} of
$X_0$ over $S$. We also assume that the deformation $X$ is symplectic, i.e., $X/S$ is equipped with
a  symplectic form $\Omega\in \Omega^2(X/S)$, whose restriction to $X_0$ coincides with $\Omega_0$.
Here "symplectic", as usual, stands for ``closed and non-degenerate'', where the latter means that
the natural map between the relative tangent and cotangent bundles induced by $\Omega$ is an isomorphism.

The universal deformation of $X_0$ will be a formal scheme $X$ over the formal scheme $\operatorname{Per}$
that, by definition, is the formal neighborhood of the cohomology class $[\Omega_0]$ in $H^2_{DR}(X_0)$.


The $\K[S]$-module $H^2_{DR}(X/S)$ is naturally isomorphic to $H^2_{DR}(X_0)\widehat{\otimes} \K[S]$
(via the Gauss-Manin connection to be recalled in the end of the subsection). The cohomology class $[\Omega]$ is then an element of $H^2_{DR}(X_0)\widehat{\otimes} \K[S]$ and as such defines a linear map $H^2_{DR}(X_0)^*\rightarrow \K[S]$. This linear map produces a morphism $S\rightarrow \operatorname{Per}$
of formal schemes. So to a formal symplectic deformation $(X/S,\Omega)$  of  $(X_0,\Omega_0)$
one assigns a morphism $p_{X/S}:S\rightarrow \operatorname{Per}$.

\begin{Prop}\label{Prop:comm_def}
Recall that $X_0$ is a symplectic variety with $H^1(X_0,\Str_{X_0})=H^2(X_0,\Str_{X_0})=\{0\}$.
The map $X/S\mapsto p_{X/S}$ is a bijection between
\begin{itemize} \item the set of isomorphism classes of symplectic formal
deformations $X/S$ of  $X_0$
\item and the set of formal scheme morphisms $S\rightarrow \Per$. \end{itemize}
The inverse map assigns to $p:S\rightarrow \Per$ the pull-back $p^*(X_{univ}/\Per)$ of some {\rm universal}
deformation $X_{univ}/\Per$.
\end{Prop}

This statement is a variant of Theorem 3.6 from \cite{KV} (that theorem deals with deformations
of $X_0$ over local Artinian schemes so our statement is the "inverse limit" of theirs).

Below we will need a "graded" version of Proposition \ref{Prop:comm_def}.
Suppose that $\K^\times$ acts on $X_0$ such that $t.\Omega_0=t^2\Omega_0$.
  In particular, $\Omega_0$
is exact. Equip $S$ (=$(\K^n)^{\wedge_0}$)
with the $\K^\times$-action induced from the action $t.v=tv, t\in \K^\times, v\in \K^n$.
We say that a symplectic formal deformation $X/S$ is {\it graded} if $X$ is equipped with a $\K^\times$-action
such that the morphism $X\rightarrow S$ is equivariant and the symplectic form $\Omega$ on $X/S$
satisfies $t.\Omega=t^2\Omega$. Then $p_{X/S}$  is restricted from a linear
map $\K^n\rightarrow H^2_{DR}(X_0)$ also denoted by $p_{X/S}$.

Performing an easy modification of the argument in \cite{KV} one gets the following result.
\begin{Prop}\label{Prop:comm_def_graded}
We preserve the assumptions of Proposition \ref{Prop:comm_def} and of the previous paragraph.
The universal deformation $X_{univ}/\Per$ can be made graded in such a way that
the map $X/S\rightarrow p_{X/S}$  is a bijection between
\begin{itemize}
\item the set of isomorphism classes of graded symplectic  formal
deformations $X/S$ of  $X_0$.
\item and the set of linear maps $\K^n\rightarrow H^2_{DR}(X_0)$.
\end{itemize}
To a $\K^\times$-equivariant morphism $p:S\rightarrow \Per$ the inverse map assigns  the pull-back $p^*(X_{univ}/\Per)$ of $X_{univ}/\Per$.
\end{Prop}

Now let us recall the Gauss-Manin connection on $H^2_{DR}(X/S)$, where
$S$ still denotes $(\K^n)^\wedge_0$. Let $[\alpha]$ be an element of $H^2_{DR}(X/S)$ and $\xi\in \K^n$.
We need to define the covariant derivative $\partial_\xi [\alpha]$.

Let $X=\bigcup_i X^i$ be an open covering  by affine formal subschemes.  Pick sections $\alpha^i\in \Omega^2(X^i), \alpha^{ij}\in \Omega^1(X^{ij}), \alpha^{ijk}\in \K[X^{ijk}]$, such that the images of these forms in $\bigwedge^\bullet\mathcal{T}^*(X/S)$ form a \v{C}ech-De Rham cocycle  representing $[\alpha]$. Here, as usual, $X^{ij}:=X^i\cap X^j, X^{ijk}:=X^i\cap X^j\cap X^k$.

Fix liftings $\xi^i$ of $\xi\in \K^n$ to $X^i$. Then set \begin{equation}\label{eq:GM_conn1}\begin{split}&\beta^i:=\iota_{\xi^i} d\alpha^i,\\ &\beta^{ij}:=\iota_{\xi^i}(d\alpha^{ij}-\alpha^i+\alpha^j),\\
&\beta^{ijk}:=\iota_{\xi^i}(d\alpha^{ijk}-\alpha^{ij}-\alpha^{jk}-\alpha^{ki})\end{split}\end{equation}

Here $\iota_\bullet$ stands for the contraction with a vector field.
Further, let $\underline{\beta}^i,\underline{\beta}^{ij},\underline{\beta}^{ijk}$ denote the
image of $\beta^i,\beta^{ij},\beta^{ijk}$ in $\bigwedge^\bullet\mathcal{T}^*(X/S)$. From the claim that
the image of $(\alpha^i,\alpha^{ij},\alpha^{ijk})$ in $\bigwedge^\bullet\mathcal{T}^*(X/S)$ is closed it is
easy to deduce that $\underline{\beta}^{ij},\underline{\beta}^{ijk}$ are skew-symmetric.   So
$(\underline{\beta}^i,\underline{\beta}^{ij},\underline{\beta}^{ijk})$ is a cochain in the \v{C}ech-De Rham complex of $X/S$. Similarly, we see that $\underline{\beta}^i,\underline{\beta}^{ij},\underline{\beta}^{ijk}$ do not depend on the choice of the local liftings $\xi^i$.

Now let us check that $(\underline{\beta}^i,\underline{\beta}^{ij},\underline{\beta}^{ijk})$ is a  cocycle.
This, by definition, means that $d\beta^{i}, d\beta^{ij}-\beta^i+\beta^j, d\beta^{ijk}-(\beta^{ij}+\beta^{jk}+\beta^{ki})$ vanish on the vector fields that are tangent to the fibers
of $X\rightarrow S$, while $\beta^{ijk}-\beta^{ijl}+\beta^{ikl}-\beta^{jkl}=0$. The last equality follows easily from the observation
that $\beta^{ijk}(=\underline{\beta}^{ijk})$ does not depend on the choice of the liftings. Let us check
that $d\beta^i$ vanishes on the tangent vector fields. The proofs of the other two claims are similar but more involved
computationally.

Each $X^i$ can be identified with $X_i^0\times S$,
where $X^i_0$ is an open affine subvariety of $X_0$. Fix such an identification.
We may assume that $\xi^i$ is tangent to $S$ in $X^i=X_0^i\times S$. Further, we can write
$\alpha^i$ as $\sum_{p}f_p\otimes\alpha^i_p  +\overline{\alpha}^i$, where $f_p$'s form a topological basis in $\K[S]$,
 $\alpha^i_p$ are some 2-forms on $X^i_0$, and $\overline{\alpha}^i$ is a 2-form vanishing
 on all vector fields tangent to $X^i_0$. Then $\beta^i=\sum_{p}\partial_{\xi^i}f_p \otimes \alpha^i_p+\iota_{\xi^i} d\overline{\alpha}^i=\sum_p\partial_{\xi^i} f_p\alpha^i_p- d\iota_{\xi^i}\overline{\alpha}_i-\mathcal{L}_{\xi^i}\overline{\alpha}_i$, where $\mathcal{L}_{\xi^i}$
 denotes the Lie derivative. We remark that $\mathcal{L}_{\xi^i}\overline{\alpha}_i$ vanishes on the vector fields
 tangent to $X^i_0$. Therefore the image of $d\beta^i$ in $\Omega^3(X^i/S)$ equals $\sum_p \partial_{\xi} f_p\otimes d\alpha^i_p$. The image of $d\alpha^i$ in $\Omega^3(X^i/S)$ equals $\sum_p f_p\otimes d\alpha^i_p$. Therefore $d\alpha^i_p=0$ and so $d\beta^i=0$ in $\Omega^3(X^i/S)$, and we are done.

Now we need to check that the class of $(\underline{\beta}^i,\underline{\beta}^{ij},\underline{\beta}^{ijk})$
does  not depend on the choice of $(\alpha^i,\alpha^{ij},\alpha^{ijk})$. This will follow if we check that
$(\underline{\beta}^i,\underline{\beta}^{ij},\underline{\beta}^{ijk})$ is exact provided
$\alpha^{ijk}=0$ and $\alpha^i,\alpha^{ij}$ vanish on the tangent vectors. This is checked analogously
to the previous paragraph.

So for $\xi.[\alpha]$ we take the class of $(\underline{\beta}^i,\underline{\beta}^{ij},\underline{\beta}^{ijk})$.
Similarly to the above, one can check that $\xi.\eta.[\alpha]-\eta.\xi.[\alpha]=[\xi,\eta].[\alpha]$. So
$[\alpha]\mapsto \xi.[\alpha]$ defines a flat connection on the $\K[S]$-module $H^2_{DR}(X/S)$. This is the Gauss-Manin
connection. It is a standard fact that a flat connection defines an identification $H^2_{DR}(X/S)\cong H^2_{DR}(X_0)\widehat{\otimes}\K[S]$.

\subsection{Generalities  on deformation quantization}\label{SUBSECTION_Basic_def}
Let $S$ be a  scheme of finite type over $\K$ or a completion of such a scheme.

Let $X$ be a smooth scheme of finite type over $S$, $\rho:X\rightarrow S$ be a projection.
We assume that $X$ is symplectic with symplectic form $\Omega\in \Omega^2(X/S)$.
The form $\Omega$ induces a $\rho^{-1}(\Str_{S})$-linear Poisson bracket on $\Str_{X}$.

Let $h$ be a formal variable. Let $\Dcal$ be a sheaf of  associative $\rho^{-1}(\Str_S)[[h]]$-algebras flat over $\K[[h]]$, complete in the $h$-adic topology, and equipped with an isomorphism  $\theta:\Dcal/h \Dcal\xrightarrow{\sim} \Str_X$ (of sheaves of $\rho^{-1}(\Str_S)$-algebras). There is  a Poisson structure on $\Dcal/ h\Dcal$: for local sections $a,b$ of $\Str_X$ we pick their local liftings $\tilde{a},\tilde{b}$ to sections of $\Dcal$ and define the bracket $\{a,b\}$ as the class of $\frac{1}{h}[\tilde{a},\tilde{b}]$. We say that $\Dcal$ (or, more precisely, the  pair $(\Dcal,\theta)$) is a {\it quantization} of $\Str_X$ (or of $X$) if the isomorphism $\theta:\Dcal/ h\Dcal\rightarrow \Str_X$ intertwines the Poisson brackets.

We say that two quantizations $(\Dcal_1,\theta_1),(\Dcal_2,\theta_2)$ of $X$ are isomorphic if there is a $\rho^{-1}(\Str_S)[[h]]$-linear isomorphism $\varphi:\Dcal_1\rightarrow \Dcal_2$ of sheaves of algebras such that the induced isomorphism $\Dcal_1/h\Dcal_1\rightarrow \Dcal_2/ h \Dcal_2$ intertwines $\theta_1,\theta_2$. The set of isomorphism classes of quantizations of $X$ will be denoted by $\Quant(X/S)$.

Below we will need to consider two special types of quantizations: {\it graded} and {\it even}
quantizations.

The former is defined when we have a $\K^\times$-action on $X$ with some special properties.
Namely, suppose that $X$ and $S$ are acted on by $\K^\times$ in such a way that $\rho:X\rightarrow S$
is $\K^\times$-equivariant and $\Omega$ has degree 2 with respect to the $\K^\times$-action:
the form $t.\Omega$ obtained from $\Omega$ by the push-forward with respect to the automorphism induced by
$t\in \K^\times$ equals $t^2\Omega$.
We say that a quantization $\Dcal$ of $X$  is {\it graded} if $\Dcal$ is equipped with a
$\K^\times$-action by algebra automorphisms such that $t.h=t^2h$ and the isomorphism
$\Dcal/ h\Dcal\cong \Str_X$ is $\K^\times$-equivariant.

Moreover, there is a natural action of $\K^\times$ on $\Quant(X/S)$. Namely, pick a quantization
$\Dcal$ of $X/S$. For $t\in \K^\times$ define a sheaf $^t\Dcal$ on $X$ as $t_*(\Dcal)$, where $t_*$
stands for the sheaf-theoretic push-forward with respect to the automorphism of $X$ induced by $t$.
Clearly, $^t\Dcal$ is a sheaf of $\K$-algebras on $X$. Turn it into a sheaf of $\rho^{-1}(\Str_S)[[\hbar]]$-algebras
by setting $sd:= (t.s)d, hd:=(t^2h)d$ (in the right hand side the products are taken
in $\Dcal$) for local sections $s$ of $\Str_S$, $d$ of $^t\Dcal$.
It is straightforward to check that $^t\Dcal$ is again a quantization of $X$. This defines a
$\K^\times$-action on $\Quant(X/S)$.

Clearly, an isomorphism class of a graded quantization is a $\K^\times$-stable point in $\Quant(X/S)$.
Conversely, let $^t\Dcal\cong \Dcal$ for all $t\in \K^\times$. Pick an isomorphism $\varphi_t:\Dcal\rightarrow \,^t\Dcal$. Define a map  $a:\K^\times\times \K^\times\rightarrow \Aut(\Dcal)$ by $\varphi_{t_1t_2}=
t_{1*}(\varphi_{t_2})\varphi_{t_1}a(t_1,t_2)$. Since the group $\Aut(\Dcal)$ is non-commutative,
the map $a$ is, in general, not a 2-cocycle. However, every element of $\Aut(\Dcal)$ has the form $\exp( h d)$,
where $d$ is a $\rho^{-1}(S)[[h]]$-linear derivation of $\Dcal$. Indeed, for $\varphi\in \Aut(\Dcal)$
we can put $d=\frac{1}{h}\ln(\varphi)$.

So, as a group, $\Aut(\Dcal)$ is just
$h\Der(\Dcal)$, where the multiplication is given by the Campbell-Hausdorff series. In particular,
this multiplication is commutative modulo $h$. Therefore modulo $h$ the function $a$ is a 2-cocycle.
By the Hilbert theorem 90, this cocycle is a coboundary. After modifying $\varphi_t$
in an appropriate way  we get $\varphi_{t_1t_2}=t_{1*}(\varphi_{t_2})\varphi_{t_1}$ modulo $h$.
Now repeat the same procedure  modulo $h^2$. So we see that we can choose isomorphisms $\varphi_t:\Dcal\rightarrow \,^t\Dcal$ in such a way that $\varphi_{t_1t_2}=t_{1*}(\varphi_{t_2})\varphi_{t_1}$ for all $t_1,t_2\in \K^\times$.
The existence of such isomorphisms means that $\Dcal$ is graded.

A similar argument implies that two graded quantizations are isomorphic if and only if they are
$\K^\times$-equivariantly isomorphic.

Let us present a very standard example of a quantization.
Let $S$ be a single point, $\underline{X}$ a smooth variety and $X:=T^*\underline{X}$
be the cotangent bundle of $\underline{X}$ equipped with a standard symplectic form.
Consider the sheaf $\Dcal_{h,\underline{X}}$ of ``homogeneous" differential operators.
This is a sheaf of $\K[h]$-algebras on $\underline{X}$  generated by $\Str_{\underline{X}}$ and the
tangent sheaf $T_{\underline{X}}$ modulo the following relations
\begin{equation}\label{eq:diff_relations1}
\begin{split}
&f*g=fg,\\
&f*v=fv,\\
&v*f=fv+hv.f,\\
&u*v-v*u=h[u,v].
\end{split}
\end{equation}
Here $f,g$ are local sections of $\Str_X$, while $u,v$ are local sections of $\mathcal{T}_X$.
In the left hand side $*$ means a product in $\Dcal_{h,\underline{X}}$. We remark that $\Dcal_{h,\underline{X}}/(h-1)$
is the usual sheaf of linear differential operators on $\underline{X}$. Also $\Dcal_{h,\underline{X}}/(h)$ is naturally identified with $p_*(\Str_{X})$, where $p:X\twoheadrightarrow \underline{X}$ is a natural projection.

We remark that $\Dcal_{h,\underline{X}}$ is not a quantization of $X$ in the above sense because
this is a sheaf on $\underline{X}$  and not on $X$ and it is not complete in the $h$-adic topology.
To fix this we can replace $\Dcal_{h,\underline{X}}$
with its $h$-adic completion and localize it to $X$ (compare with \cite{BK1}, Remark 1.6).
Abusing the notation, we still denote this sheaf on $X$ by
$\Dcal_{h,\underline{X}}$. Consider a  $\K^\times$-action on $\Dcal_{h,\underline{X}}$
that is uniquely determined by $t.f=f, t.v=t^2v, t.h=t^2 h$ for any $t\in \K^\times$.
To recover the initial sheaf on $\underline{X}$ one takes $\K^\times$-finite sections
in $p_*(\Dcal_{h,\underline{X}})$.

More generally, we can consider the homogeneous analogs of twisted differential operators.
By definition, a sheaf of homogeneous twisted differential operators is a pair of
\begin{itemize}\item a sheaf $\Dcal$ of $\K[h]$-algebras   on $\underline{X}$ equipped with a $\K^\times$-action by algebra automorphisms and \item a $\K^\times$-equivariant $\K[h]$-algebra embedding $\Str_{\underline{X}}[h]\rightarrow \Dcal$ (where $\K^\times$ acts trivially on $\Str_{\underline{X}}$
and $t.h=th$)\end{itemize} subject to the following condition.  There are
\begin{itemize}
\item
an open  covering $\underline{X}:=\bigcup_i \underline{X}^i$ and
\item $\K^\times$-equivariant isomorphisms $\iota^i:\Dcal|_{\underline{X}^i}\rightarrow \Dcal_{h,\underline{X}^i}$ intertwining the embeddings of $\Str_{X^i_0}[h]$
\end{itemize} that produce a global isomorphism $\Dcal/(h)\xrightarrow{\sim}p_*(\Str_X)$. As with usual twisted differential operators, the sheaves of homogeneous differential operators are classified up to an isomorphism by $H^1(X, \Omega^1_{cl})$, where $\Omega^1_{cl}$ denotes the sheaf of closed  1-forms on $\underline{X}$ (an isomorphism is supposed to intertwine the $\Str_{\underline{X}}[h]$-embeddings).

For example, let $\mathcal{L}$ be a line bundle. Let $\underline{X}'$ denote the complement to
$\underline{X}$ in the total space of $\mathcal{L}$.  This is a principal $\K^\times$-bundle on $\underline{X}$,
let $\tau:\underline{X}'\twoheadrightarrow \underline{X}$
denote the canonical projection.
The $\K^\times$-action on $\underline{X}'$ gives rise to the Euler vector field $\mathbf{eu}$ on $\underline{X}'$
so that the eigensheaf of $\mathbf{eu}$ in $\tau_*(\mathcal{L})$ with eigenvalue $1$ is nothing else
but $\mathcal{L}$. For $\alpha\in \K$ define the sheaf $\Dcal^{\alpha\mathcal{L}}_{h,\underline{X}}$ as the quantum Hamiltonian
reduction $$\tau_*(\Dcal_{h,\underline{X}'})^{\K^\times}/\tau_*(\Dcal_{h,\underline{X}'})^{\K^\times}(\mathbf{eu}-\alpha h).$$
For example, when $\mathcal{L}$ is the canonical bundle $\Omega^{top}_{\underline{X}}$ of $\underline{X}$, the  sheaf $\Dcal^{\alpha\Omega^{top}_{\underline{X}}}_{h,\underline{X}}$ is generated by $\Str_{\underline{X}},\mathcal{T}_{\underline{X}}$ subject to the 1st and 4th relations  (\ref{eq:diff_relations1}) and to
\begin{equation}\label{eq:diff_relations2}
\begin{split}
&f*v=fv-\alpha h v.f,\\
&v*f=fv+(1-\alpha) h v.f.
\end{split}
\end{equation}

Below we write $D^\alpha_{h,\underline{X}}$ instead of $D^{\alpha\Omega^{top}_{\underline{X}}}_{h,\underline{X}}$ and view it
as a sheaf defined by generators and relations as above.


Proceed to the definition of  even quantizations.
We say that a graded quantization  $\Dcal$ is even, if there is an involutory  antiautomorphism $\sigma:\Dcal\rightarrow \Dcal$ (to be referred to as a "parity antiautomorphism") that is trivial modulo $h$ and maps $h$ to $-h$.
If $\Dcal$ was $\K^\times$-equivariant we additionally require that $\sigma$ is $\K^\times$-equivariant.

We can define an action of $\Z/2\Z$ on $\Quant(X/S)$. Namely, let $\sigma$ denote the non-trivial element
in $\Z/2\Z$. For $\Dcal\in \Quant(X/S)$ set $^\sigma \Dcal=\Dcal^{op}$ and equip $\Dcal^{op}$ with
a $\rho^{-1}(\Str_S)[[h]]$-algebra structure by preserving the same $\rho^{-1}(\Str_S)$-algebra structure but
making $h$ act by $-h$. Similarly to the above, a quantization is even if and only if it is a fixed
point for the $\Z/ 2\Z$-action on $\Quant(X/S)$.

Let us present an (again, pretty standard) example of an even quantization. Let $\underline{X},X$ be such as above. Consider the sheaf $\Dcal_h^{1/2}(\underline{X})$. Define its anti-automorphism $\sigma$ by $\sigma(f)=f,
\sigma(v)=v, \sigma(h)=-h$ for local sections $f$ of $\Str_{\underline{X}}$ and $v$ of $\mathcal{T}_{\underline{X}}$. It is clear that
$\Dcal^{1/2}_h(\underline{X})$ is even (with parity antiautomorphism $\sigma$).

We will also need some more straightforward compatibility of star-products with group actions.
Namely, let $G$ be an algebraic group acting on $X$ such that $\rho:X\rightarrow S$ is
$G$-invariant. We say that a quantization $\Dcal$ is $G$-equivariant if $\Dcal$ is equipped with
an action of $G$ by algebra automorphisms such that $h$ is $G$-invariant and
the isomorphism $\Dcal/h\Dcal\cong \Str_X$ is $G$-equivariant.

\subsection{Non-commutative period map}\label{SUBSECTION_non_comm_period}
Let $X/S,\Omega$ be such as in the previous subsection.
In this subsection  we will explain the approach of \cite{BK1} to the problem of
classifying quantizations of $X/S$. Following \cite{BK1} we will produce a certain
natural map $\Per:\Quant(X/S)\rightarrow h^{-1} H^2_{DR}(X/S)[[h]]\subset H^2_{DR}(X/S)[h^{-1},h]]$
(the non-commutative period map).
Under some cohomology vanishing conditions on $X$ this map is a bijection.
   As Bezrukavnikov and Kaledin point out in their paper,
their approach is not essentially new, but the language they use turns out to be very
convenient for our purposes. The only (somewhat) new feature in our exposition is the presence of a $\K^\times\times \Z/2\Z$-action.

Let us recall the main result of \cite{BK1} providing a classification
of quantizations. We say that $X$ is {\it admissible} if the natural homomorphism
$H^i_{DR}(X/S)\rightarrow H^i(X,\Str_X)$ is surjective for $i=1,2$. Let $K^2$ denote the kernel
of $H^2_{DR}(X/S)\rightarrow H^2(X,\Str_X)$.

\begin{Prop}\label{Prop:quant_classif}
For any $\Dcal\in \Quant(X/S)$ we have $\Per(\Dcal)\in h^{-1}[\Omega]+H^2_{DR}(X/S)[[h]]$.
Further, fix a splitting $P:H^2_{DR}(X/S)\twoheadrightarrow K^2$ of the inclusion $K^2\hookrightarrow H^2_{DR}(X/S)$.
Then $\Dcal\mapsto P(\Per(\Dcal)-[\Omega])$ defines a bijection $\Quant(X/S)\xrightarrow{\sim} K^2[[h]]$.
\end{Prop}

Now let us discuss the compatibility of $\Per$ with group actions considered in the previous subsection.
Let $\K^\times$ act on $X$ and $S$ as above.
We remark that $\Omega$ is exact provided $S$ is a point.
The following proposition establishes a  relation between $\Per$ and the $\K^\times\times \Z/2\Z$-action on $\Quant(X/S)$.

\begin{Prop}\label{Prop:Per_group}
\begin{enumerate}
\item If $\Dcal\in \Quant(X/S)$ is graded, then $\Per(\Dcal)\in H^2_{DR}(X/S)\subset h^{-1}H^2_{DR}(X/S)[[h]]$.
\item If $\Dcal$ is, in addition, even, then $\Per(\Dcal)=0$.
\end{enumerate}
\end{Prop}

The proof of Proposition \ref{Prop:Per_group} follows fairly easily from the constructions of \cite{BK1}
and will be given after we recall all necessary definitions and constructions below in this subsection.

Following \cite{BK1}, we call a quantization $\Dcal$ of $X/S$ {\it canonical}
if $\Per(\Dcal)=[\Omega]$.

\begin{Cor}\label{Cor:quant_classif}
Suppose $H^1(X,\Str_X)=H^2(X,\Str_X)=0$.
Let $X/S$ be equipped with a $\K^\times$-action as above. Then $\Per$ identifies the set of isomorphism classes of
graded quantizations of $X/S$  with $H^2_{DR}(X/S)\subset h^{-1}H^2_{DR}(X/S)[[h]]$.
Further, an even graded quantization of $X/S$ is canonical.
\end{Cor}

An essential notion in the approach of \cite{BK1} is that of a {\it Harish-Chandra torsor}. For reader's convenience, we will sketch basic definitions and results here. For a more detailed
 exposition, the reader is referred to \cite{BK1}.

Let $G$ be a (pro)algebraic group and $\h$ a (pro)finite dimensional Lie algebra equipped
with a $G$-action and with a $G$-equivariant embedding $\g\hookrightarrow \h$ that are compatible in the sense
that the differential of the $G$-action coincides with the adjoint action of $\g$ on $\h$. A pair
$(G,\h)$ is called a Harish-Chandra pair.  It is straightforward
to define the notions  of a module over a Harish-Chandra pair and of a homomorphism of Harish-Chandra pairs.

Examples we need in this paper (and which were considered in \cite{BK1}) include the following.


1) Consider the Poisson bracket on $A:=\K[[x_1,\ldots,x_n,y_1,\ldots,y_n]]$ given by $\{x_i,x_j\}=\{y_i,y_j\}=0, \{x_i,y_j\}=\delta_{ij}$.
Consider the subgroup $\Symp A$ of $\Aut A$ consisting of all Poisson automorphisms  and the subalgebra
$\Ham\subset \Der A$ consisting of all Hamiltonian (=annihilating the Poisson bracket) derivations.
Then $(\Symp A, \Ham)$ is a Harish-Chandra pair.

2) Let $D:=\W_{2n,h}^{\wedge_0}$ be the completed Weyl algebra $\K[[x_1,\ldots,x_n,y_1,\ldots,y_n,h]]$ with the multiplication given by the Moyal-Weyl formula $f*g=m(\exp(\frac{Ph}{2})f\otimes g)$, where $m$ stands for
the standard commutative multiplication map $D\otimes D\rightarrow D$ and $P:D\otimes D\rightarrow D\otimes D$
is given by $$f\otimes g\mapsto \sum_{i=1}^n \frac{\partial f}{\partial x_i}\otimes\frac{\partial g}{\partial y_i}-
\frac{\partial f}{\partial y_i}\otimes\frac{\partial g}{\partial x_i}.$$
A Harish-Chandra pair under consideration is $(\Aut D, \Der D)$, where both automorphisms and derivations
are supposed to be $\K[[h]]$-linear.

The next notion we need is that of a {\it Harish-Chandra torsor}. Let $G$ be a pro-algebraic group.
By a torsor over  $X$ we mean a $G$-scheme  $M$ over $X$ such that the structure morphism $\pi:M\rightarrow X$
is $G$-invariant and the  action map $G\times M\rightarrow M$ together with a projection $G\times M\rightarrow M$ gives rise to an isomorphism $G\times M\xrightarrow{\sim}M\times_X M$. We have the short exact sequence
(the Atiyah extension) of $\Str_X$-modules
$$0\rightarrow \g_M\rightarrow \mathcal{E}_M\rightarrow \mathcal{T}_{X/S}\rightarrow 0,$$
where $\g_M:=\pi_*(\Str_M\widehat{\otimes} \g)$
and $\mathcal{E}_M:=\pi_*(\mathcal{T}_{M/S})$. By an $\h$-valued connection on $M$ one means a $G$-equivariant bundle
map $\theta_M:\mathcal{E}_M\rightarrow \h_M:=\pi_*(\Str_M\widehat{\otimes} \h)$
such that the composition $\g_M\rightarrow \mathcal{E}_M\rightarrow \h_M$ coincides with the embedding
$\g_M\rightarrow \h_M$ induced by the embedding $\g\hookrightarrow \h$. An $\h$-valued connection
$\theta_M$ is said to be {\it flat} if  the $\h$-valued 1-form $\Omega:=\pi^*(\theta_M)$ on $M$
satisfies the Maurer-Cartan equation $2 d\Omega+\Omega\wedge\Omega=0$. The pair $(M,\theta_M)$ of a $G$-torsor
$M$ and an $\h$-valued flat connection $\theta_M$ is said to be a Harish-Chandra $\langle G,\h\rangle$-torsor.
The Harish-Chandra $\langle G,\h\rangle$-torsors on $X$ naturally form a groupoid (=category with invertible morphisms), let $H^1(X,\langle G,\h\rangle)$ denote the set of isomorphism classes of Harish-Chandra torsors.

We will need some functoriality properties for Harish-Chandra torsors. Now let $\langle G_1,\h_1\rangle,$ $\langle G_2,\h_2\rangle$ be two Harish-Chandra torsors and let
$\varphi:\langle G_1,\h_1\rangle\rightarrow \langle G_2,\h_2\rangle$ be a homomorphism.
Pick a Harish-Chandra $\langle G_1,\h_1\rangle$-torsor $M_1$ and define its push-forward $\varphi_*(M)$
as follows. As a torsor $\varphi_*(M_1)$ is the quotient of $G_2\times M_1$ by the diagonal action
of $G_1$ ($g_1\in G_1$ acts on $G_2$ via $g_1.g_2=g_2\varphi(g_1)^{-1}$). The bundle $\mathcal{E}_{\varphi_*(M_1)}$
is the quotient of $\g_{2M_1}\oplus \mathcal{E}_{M_1}$ modulo the image of $\g_{1M_1}$ under $(-\varphi_{M_1},\iota_{M_1})$, where $\varphi_{M_1}:\g_{1M_1}\rightarrow \g_{2M_2}$ is the map induced by $\varphi$
and $\iota_{M_1}:\g_{1M_1}\rightarrow \mathcal{E}_{M_2}$ is the inclusion above. Define a map
$\g_{2M_1}\oplus \mathcal{E}_{M_1}\rightarrow \h_{2M_1}$  as the direct sum of the inclusion $g_{2M_2}\hookrightarrow \h_{2M_2}$ and of $\varphi_{M_1}\circ \theta_{M_1}$. This map vanishes on the image of $\g_{1M_1}$
and hence defines a map $\mathcal{E}_{M_2}\rightarrow \h_{2M_2}$. It is straightforward to check that
this map is a flat connection.

We will need two appearances   of this construction. First of all, let a group
$\Gamma$ act by automorphisms of a Harish-Chandra pair $\langle G,\h\rangle$ and also
by automorphisms on $X$. 
Then the previous construction gives rise to a $\Gamma$-action on $H^1(X,\langle G,\h\rangle)$.

Second, let $\varphi:\langle G_1,\h_1\rangle\rightarrow \langle G,\h\rangle$ be an epimorphism
and $M$ be a Harish-Chandra $\langle G,\h\rangle$-torsor. By  a {\it lifting} of $M$ to $\langle G_1,\h_1\rangle$
we mean a Harish-Chandra $\langle G_1,\h_1\rangle$-torsor $M_1$ equipped with an isomorphism
$\pi_*(M_1)\xrightarrow{\sim} M$. Let $H^1_M(X,\langle G_1,\h_1\rangle)$ denote the set
of isomorphism classes of liftings of $M$ to $\langle G_1,\h_1\rangle$.

The constructions considered in the previous two paragraphs combine together as follows.
Suppose that $\Gamma$ acts by automorphisms on both $\langle G,\h\rangle$
and $\langle G_1,\h_1\rangle$ such that the epimorphism $\varphi$ is $\Gamma$-equivariant.
Further, suppose that $M$ is $\Gamma$-stable. Then we get a natural action of
$\Gamma$ on $H^1_M(X,\langle G_1,\h_1\rangle)$.

Let us proceed to some examples of Harish-Chandra torsors and their liftings we need.
Set $n:=\frac{1}{2}\dim X/S$.

Recall the Harish-Chandra pair $\langle\Symp A, \Ham\rangle$. Then $X$ has a canonical Harish-Chandra
torsor $\M_{symp}$ of symplectic coordinate systems. Following \cite{BK1}, Lemma 3.2,
we define it as the (pro-finite) scheme of all morphisms $\varphi:(\K^{2n})^{\wedge_0}\rightarrow X$
that are compatible with the symplectic structure, i.e., $\varphi^*(\Omega)$ coincides
with the symplectic form $\sum_{i=1}^n dx_i\wedge dy_i$.

Assume now that $\K^\times$ acts on $X$
satisfying the conventions of the previous section. The action of $\K^\times$ on $A$
given by $t.x_i=tx_i, t.y_i=ty_i$ gives rise to a $\K^\times$-action on $\langle \Symp A, \Ham\rangle$
and hence on $H^1(X, \langle \Symp A, \Ham\rangle)$. Also we remark that under the $\K^\times$-actions
both on $X$ and $(\K^{2n})^{\wedge_0}$ an element $t\in \K^\times$ multiplies the symplectic forms by $t^2$.
Thus $\M_{symp}\in H^1(X,\langle \Symp A, \Ham\rangle)$  is $\K^\times$-stable.

Now recall the Harish-Chandra pair $\langle\Aut D, \Der D\rangle$. The group $\K^\times$ acts on
$D$ by automorphisms: $t.x_i=tx_i, t.y_i=ty_i, t.h=t^2h$. The group $\Z/2\Z$ acts on $D$ by antiautomorphisms:
$\sigma.x_i=x_i, \sigma.y_i=y_i, \sigma.h=-h$ (here $\theta$ stands for the non-unit element of
$\Z/2\Z$). This gives rise to an action of $\K^\times\times \Z/2\Z$ on $\langle \Aut D, \Der D\rangle$
and hence to the action of $\K^\times\times \Z/2\Z$ on $H^1_{\M_{symp}}(X, \langle \Aut D, \Der D\rangle)$.

The reason why we are interested in the set $H^1_{\M_{symp}}(X, \langle \Aut D, \Der D\rangle)$
is that it is in a natural bijection with $\Quant(X/S)$, see \cite{BK1}, Lemma 3.3. To define
the bijection we need  the {\it localization} construction.

Given a module $V$ over a Harish-Chandra pair $\langle G,\h\rangle$ and a Harish-Chandra torsor
$M$ over $X$ one can define a flat bundle $\Loc(M,V)$ (the localization of $V$ with respect to
$M$) as follows.  As a bundle, $\Loc(M,V)$ is the associated bundle of the principal bundle $M$ with fiber
$V$. By the construction of $\Loc(M,V)$, both
$\mathcal{E}_M,\h_M$ act on $\Loc(M,V)$. A flat connection $\nabla$ on $\Loc(M,V)$ is defined by (see \cite{BK1}, (2.2)) $$\nabla_\xi (a)= \widetilde{\xi} a -\theta_{M}(\widetilde{\xi} a),$$
where $a$ is a local section of $\Loc(M,V)$, $\xi$ is a local vector field on $X$,
and $\widetilde{\xi}$ is a local section of $\mathcal{E}_M$ lifting $\xi$. We remark that the right hand side
does not depend on the choice of $\widetilde{\xi}$.

\begin{Lem}\label{Lem:tors_vs_quant}
The map $M\mapsto \Loc(M,D)^{\nabla}$, where $\Loc(M,D)^\nabla$ stands for the sheaf
of $\nabla$-flat sections, defines a $\K^\times\times \Z/2\Z$-equivariant bijection
between $H^1_{\M_{symp}}(X, \langle \Aut D, \Der D\rangle)$ and $\Quant(X/S)$.
\end{Lem}

The claim that the map is a bijection is \cite{BK1}, Lemma 3.3. The equivariance part is verified
directly from the definitions of the $\K^\times\times \Z/2\Z$-actions.

 To define the non-commutative period map  we need some more discussion on the localization functor
 and extensions of Harish-Chandra torsors.

The localization functor $V\mapsto \Loc(M,V)$ is exact and hence it gives rise
to a natural map $\Loc(M,\bullet): H^i(\langle G,\h\rangle,V)\rightarrow H^i_{DR}(X/S, \Loc(M,V))$.
Here on the left hand side $H^i(\langle G,\h\rangle,V)$ means $\operatorname{Ext}^i(\K,V)$, where
$\Ext^i$ is taken in the category of $\langle G,\h\rangle$-modules, while in the right hand side
$H^i_{DR}$ means De Rham hyper-cohomology with respect to the flat connection on $\Loc(M,V)$.

Let $\langle G,\h\rangle$ be a Harish-Chandra pair and $V$ be a  $\langle G,\h\rangle$-module.
Let $\langle G_1,\h_1\rangle$ be an extension of $\langle G,\h\rangle$ by the Harish-Chandra
pair $(V,V)$. Further, let $M$ be a Harish-Chandra torsor over $\langle G,\h\rangle$.
We need an obstruction for extending $M$ to  a Harish-Chandra torsor over $\langle G_1,\h_1\rangle$.

Following \cite{BK1} we will state the corresponding result under some restriction on $M$. Namely,
we say that $M$ is {\it transitive} if the connection map $\theta_M: \mathcal{E}_M\rightarrow \h_M$
is an isomorphism. In particular, $\M_{symp}$ is transitive (compare with the discussion in the beginning
of Section 3 in \cite{BK1}).

\begin{Prop}[\cite{BK1}, Proposition 2.7]\label{Prop:obstruction}
In the above notation, suppose $M$ is transitive.
\begin{enumerate}
\item There exists a canonical cohomology class $c\in H^2(\langle G,\h\rangle,V)$ with the following property:
$H^1_M(\langle G_1,\h_1\rangle, V)\neq \varnothing$ if and only if $\Loc(M,c)\in H^2_{DR}(X/S, \Loc(M,V))$
vanishes.
\item If $\Loc(M,c)=0$, then $H^1_M(X,\langle G_1,\h_1\rangle)$ has a natural structure of an affine
space with the underlying vector space $H^1_{DR}(X/S, \Loc(M,V))$.
\end{enumerate}
\end{Prop}

We will need compatibility of constructions of the previous proposition with group actions.
Namely, let $\Gamma$ be a group that acts:
\begin{itemize}
\item on $X$ and on $S$ by scheme automorphisms such that the morphism $\pi:X\rightarrow S$
is $\Gamma$-equivariant.
\item on $\langle G_1,\h_1\rangle$ and $\langle G,\h\rangle$ by Harish-Chandra pairs automorphisms
such that the projection $\langle G_1,\h_1\rangle\twoheadrightarrow \langle G,\h\rangle$
is $\Gamma$-equivariant.
\end{itemize}

In particular, $\Gamma$ acts on $V$ by automorphisms of a $\langle G,\h\rangle$-module. Hence
$\Gamma$ acts on $H^2(\langle G,\h\rangle,V)$. Being canonical (see the proof of Proposition 2.7
in \cite{BK1}), the class $c$ is $\Gamma$-equivariant. The actions of an element $\gamma\in \Gamma$
on  $H^1(X,\langle G,\h\rangle)$ and on $V$  define
a push-forward isomorphism $\gamma_*: \Loc(M,V)\rightarrow \Loc(^\gamma M, V)$ of sheaves
that is compatible with the automorphism $\gamma_*$ of $\Str_X$ and intertwines the flat connections.
This follows directly
from the definition of $\Loc(\bullet,\bullet)$ given above. So we have the induced
linear map $\gamma_*: H^i(X/S, \Loc(M,V))\rightarrow H^i(X/S, \Loc(^\gamma M, V))$. Again, from the construction
of $\Loc(\bullet,\bullet)$ and the observation that $c$ is $\Gamma$-invariant
one deduces that $\gamma_*(\Loc(M,c))=\Loc(^\gamma M,c)$.

Now suppose that $M$ is $\Gamma$-stable and $\Loc(M,c)=0$. Then $\Gamma$ acts on both
the affine space $H^1_M(X,\langle G_1,\h_1\rangle)$ and the underlying vector space $H^1_{DR}(X/S, \Loc(M,V))$ and these actions are compatible.

Proceed to the definition of a non-commutative period map (see \cite{BK1}, Section 4). For this we will need
one more Harish-Chandra pair related to the Weyl algebra $D$. Consider the subspace
$h^{-1}D\subset D[h^{-1}]$. This subspace is closed with respect to
the Lie bracket on $D[h^{-1}]$. There is a Lie algebra homomorphism $\eta:h^{-1}D\rightarrow \Der D$
given by $a\mapsto [a,\cdot]$. This is a standard fact that this map is surjective,
its kernel coincides with $h^{-1}\K[h]$.

There is also an extension $\G$ of $\Aut D$ by the abelian
group $h^{-1}\K[h]$ constructed as follows. The group $\Aut D$ is the semidirect product of $\Sp_{2d}$
(acting linearly on the span of $x_1,\ldots,x_n,y_1,\ldots,y_n$) and of the normal subgroup
$\Aut_+ D$ that acts by the identity modulo the ideal generated by $h,x_ix_j, y_iy_j, x_iy_j$.
The Lie algebra $\Lie(\Aut D)$ of $\Aut D$ coincides with the algebra $\Der_0 D$ of derivations preserving the
maximal ideal of $D$. The  preimage $(h^{-1}D)_0$ of $\Der_0 D$ in $h^{-1}D$ is spanned by $h^{-1}\K[h]$
and all monomials $h^i x_1^{i_1}\ldots x_m^{i_m}y_1^{j_1}\ldots y_m^{j_m}$ where either $i\geqslant 0$
or $\sum_{k=1}^m i_k+j_k>1$. The Lie algebra $\spa_{2d}$ embeds naturally to $(h^{-1}D)_0$ and
$(h^{-1}D)_0=\spa_{2d}\ltimes \eta^{-1}(D_+)$. The algebra $\eta^{-1}(D_+)$ is pro-nilpotent and
so integrates to a pro-unipotent group $\G_+$. Set $\G:=\Sp_{2n}\rightthreetimes \G_+$. We have
a natural projection $\G\twoheadrightarrow \Aut D$, whose kernel is the abelian group $h^{-1}\K[[h]]$.
The group $\G$ acts on $h^{-1}D$ via the epimorphism $\G\twoheadrightarrow \Aut D$.  So $(\G,h^{-1}D)$ becomes a Harish-Chandra pair  and there is a natural
epimorphism $\langle \G,h^{-1}D\rangle\twoheadrightarrow \langle \Aut D, \Der D\rangle$ with kernel
$(h^{-1}\K[[h]],h^{-1}\K[[h]])$.

\begin{defi}\label{defi:Per}
Let $c\in H^2(\langle \Aut D,\Der D\rangle, h^{-1}\K[[h]])$ denote the canonical class of the extension
$0\rightarrow \langle h^{-1}\K[[h]],h^{-1}\K[[h]]\rangle\rightarrow  \langle \G, h^{-1}D\rangle
\rightarrow \langle\Aut D,\Der D\rangle\rightarrow 0$, see Proposition \ref{Prop:obstruction}
Let $M\in H^1_{\M_{symp}}(X,\langle \Aut D, \Der D\rangle)$. To $M$  assign an element $\Per(M)\in H^2_{DR}(X/S,\Loc(M,h^{-1}\K[[h]]))=
h^{-1}H^2_{DR}(X/S)[[h]]$ by $\Per(M):=\Loc(M,c)$. The map $$\Per: \Quant(X/S)=H^1_{\M_{symp}}(X,\langle \Aut D,\Der D\rangle)\rightarrow h^{-1}H^2_{DR}(X/S)[[h]]$$ is called the {\it non-commutative period map}.
\end{defi}

Recall that by Lemma \ref{Lem:tors_vs_quant}, the set $H^1_{\M_{sympl}}(X, \langle \Aut D, \Der D\rangle)$
is in bijection with $\Quant(X/S)$. For $\Dcal\in \Quant(X/S)$ we write $\Per(\Dcal)$ for the image of
the corresponding torsor under $\Per$.

\begin{proof}[Proof of Proposition \ref{Prop:Per_group}]
Let us introduce a $\K^\times\times \Z/2\Z$-action on the Harish-Chandra pair $(\G,h^{-1}D)$.
Take a $\K^\times$-action on $h^{-1}D$ restricted from $D[h^{-1}]$. An action of $\Z/2\Z$ we need
is defined as follow. A non-trivial element $\sigma$ acts on $h^{-1}D$ as a unique continuous anti-automorphism
 mapping $x_i$ to $x_i$, $y_i$ to $y_i$ and $h$ to $-h$.
Equip $\G$ with a unique $\K^\times\times \Z/2\Z$-action compatible with the action on $h^{-1}D$.
It is checked directly that the epimorphism $\langle \G, h^{-1}D\rangle\twoheadrightarrow \langle \Aut D,\Der D\rangle$
is $\K^\times\times \Z/2\Z$-equivariant. The induced action on $h^{-1}\K[[h]]$ is given by:
$t.h^i= t^i h^i, \sigma.h^i=(-1)^{i+1}h^i$. Both claims of the proposition follow from the definition
of $\Per$ and the above discussion of  the compatibility of $\Loc(\bullet,\bullet)$
with group actions.
\end{proof}

\begin{Rem}\label{Rem_formal}
All results quoted above transfer to the formal scheme setting (e.g., to formal deformations of
symplectic varieties) directly without
any noticeable modifications.
\end{Rem}

\section{Hamiltonian reduction}\label{SECTION_red}
\subsection{Classical reduction}\label{SUBSECTION_class_red}
An important construction of symplectic varieties  is that of a Hamiltonian reduction. In this subsection we will recall it. All results gathered in this subsection are very standard. The proofs are the same as in the
$C^\infty$-setting, see, for example, \cite{GS}.

First of all, let $A$ be a Poisson algebra, $\g$ be a Lie algebra equipped with a Lie algebra homomorphism
$\varphi:\g\rightarrow A$. Extend $\varphi$ to a Poisson algebra
homomorphism $S\g\rightarrow A$. Pick  $\chi\in \z$. The ideal $A\varphi(\g_\chi)$, where
$\g_\chi:=\{\xi-\langle \chi,\xi\rangle, \xi\in\g\}\subset S\g$, is stable with respect to the
action of $\g$ on $A$ given by $\xi\mapsto [\varphi(\xi),\cdot]$.
Define the algebra $A\red_\chi \g:=(A/A \varphi(\g_\chi))^{\g}$. By definition, this is a classical Hamiltonian
reduction of $A$ (with respect to $\varphi:\g\rightarrow A$).

There is one important special case of this construction.
Suppose that a connected algebraic group $G$ with Lie algebra $\g$
acts on $A$ in such a way that the derivation $[\varphi(\xi),\cdot]$
is the image of $\xi\in\g$ under the differential of the action. Then $G$ acts on $A/A\varphi(\g_\chi)$
and the $G$-invariants are the same as the $\g$-invariants. Let $\z$ be $\K\otimes_{\Z}\mathfrak{X}(G)$,
where $\mathfrak{X}(G)=\Hom_{\Z}(G,\K^\times)$ is the group of characters of $G$.
In this case we denote the reduction
by $A\red_\chi G$. Also for a subspace $U\subset (\g/[\g,\g])^*$ we can define the reduction
$A\red_U \g:=(A/A\varphi(U^\perp))^\g$. Here $U^\perp$ stands for the annihilator of $U$ in $\g$.
When $U=\z$ we write $A\red\g$ instead of $A\red_U \g$. We remark that $\varphi$
induces an algebra homomorphism $\K[U]\rightarrow A\red_U \g$, whose image lies in
the Poisson center. Clearly, $A\red_U\g:=\K[U]\otimes_{\K[(\g/[\g,\g])^*]}A\red \g$.

Proceed to the definition of reduction for Poisson varieties.
Let $\X$ be a  Poisson algebraic variety.
Let $G$ be a connected algebraic group acting on $\X$ by Poisson automorphisms. Suppose that the action
admits a  moment map $\mu:\X\rightarrow\g^*$, a $G$-equivariant map such that the following condition holds:
\begin{itemize}
\item $\{\mu^*(\xi),\cdot\}=\xi_\X$, where $\xi_\X$ denotes the derivation of $\Str_\X$
(the velocity vector field) corresponding to $\xi$.
\end{itemize}

We suppose that the $G$-action is free, that the quotient $\X/G$ exists and that the quotient
morphism $\pi_G:\X\rightarrow \X/G$ is affine.

Consider the inverse image $\mu^{-1}(\chi)\subset \X$. This is a smooth complete intersection.
Set $\X\red_\chi G:= \mu^{-1}(\chi)/G\subset \X/G$. Thanks to the discussion above, $\X\red_\chi G$
has a natural structure of a Poisson variety. Similarly, we can define the Poisson varieties
$\X\red_U G, \X\red G$.

Now suppose that $\X$ is a smooth variety and that the Poisson bivector is non-degenerate, so
that $\X$ is symplectic. Let $\omega$ denote the symplectic form on $\X$.
There is a unique 2-form  $\Omega$ on $\X\red G$  such that
$\mu^*(\Omega)$ coincides with the restriction of $\omega$ to $\mu^{-1}(\z)$. The
natural morphism $\pi: \X\red G\rightarrow \z$ induced from $\mu$
is smooth of relative dimension $\dim \X-2\dim G$.
The form $\Omega$ belongs to $\Omega^2(\X\red G, \z)$ and is a symplectic
form on the $\z$-scheme $\X\red G$. The reduction $\X\red_\eta G$ is a
symplectic variety. We remark that
$\X\red_U G$ equals $U\times_{\z}\X\red G$ and so is a symplectic scheme over $U$.

Suppose that $\K^\times$ acts on $\X$ such that $t.\omega=t^2\omega$ and $\mu(t.x)=t^{-2}\mu(x)$ for all
$x\in \X$. Then the $\K^\times$-action descends to $\X\red G$. We have $t.\Omega=t^2\Omega$. Equip $\z$ with a $\K^\times$-action
by $t.\alpha=t^{-2}\alpha$. Then the morphism $\pi: \X\red G\rightarrow \z$ becomes $\K^\times$-equivariant. In particular, we have  natural $\K^\times$-actions on $\X\red_0 G, \X\red_U G$.

Finally, let $V$ be a $G$-module. Consider the $G$-equivariant vector bundle $\X\times V\rightarrow \X$, where
the $G$-action on $\X\times V$ is supposed to be diagonal. Since the action of $G$ on $\X$ is free,
this $G$-equivariant bundle descends to the bundle on $\X/G$ to be denoted by $\mathcal{V}$.
Restricting the latter to $\X\red G$ or to $\X\red_0 G$ we get  bundles on these varieties.


\subsection{The Duistermaat-Heckman theorem}\label{SUBSECTION_DH}
If a symplectic variety is obtained by Hamiltonian reduction, then there is an easy way to produce its formal deformation.
The period map (see Subsection \ref{SUBSECTION_def_classic})
for this deformation can be computed with the help of (an algebraic variant
of) the Duistermaat-Heckman theorem.

Let $\X,\omega,G,\z$ be  as in Subsection \ref{SUBSECTION_class_red}. Suppose that $\K^\times$ acts on $\X$ as explained in the end of that subsection.

Let us produce a graded symplectic formal deformation of $X_0:=\X\red_0 G$. Namely, take the symplectic scheme $X:=\X\red G$
over $\z$ and consider the completion of $X$ along $X_0$ to be denoted by $\X\widehat{\red} G$.
This is a graded formal deformation of interest.
According to Subsection \ref{SUBSECTION_def_classic}, this deformation gives rise to a linear map $p:\z\rightarrow H^2_{DR}(X_0)$. Our goal is to describe this map.


Replacing $G$ with $G/(G,G)$ and $\X$ with $\X\red (G,G)$, we may assume that $G$ is a torus. We have the principal $G$-bundle $\mu^{-1}(0)\rightarrow \X\red_0 G$.
Let $c\in H^2_{DR}(X_0)\otimes \g=\Hom(\g^*, H^2_{DR}(X_0))$ denote its Chern class (we will recall the definition suitable for our purposes below).  

The following statement is an algebro-geometric version of the Duistermaat-Heckman theorem.

\begin{Prop}\label{Prop_DH}
The map $\g^*\rightarrow H^2_{DR}(X_0)$ induced by the deformation $\X\widehat{\red}G$ of $X_0$
coincides with $c$.
\end{Prop}

Let us recall one of the possible definitions of $c$. In the $C^\infty$-setting the Chern class of a
line bundle can be defined as the cohomology class of the curvature form of a connection on this bundle. A
similar definitions can be given in the algebraic situation.

Namely, let $Y$ be a smooth algebraic variety and $\widetilde{Y}\rightarrow Y$ be a principal $G$-bundle,
where $G$ is still a torus.
Then $\widetilde{Y}\rightarrow Y$ is locally trivial in the Zariski topology. In particular, one can find
an open affine covering $Y=\bigcup_i Y^i$ such that the restriction $\widetilde{Y}^i\rightarrow Y$ of
$\widetilde{Y}\rightarrow Y$ admits a connection. Let $\alpha^i\in \Omega^1(Y^i)\otimes \g$ be the  connection form.
Then both $d\alpha^i$ and $\alpha^i-\alpha^j$ descend to $Y$ and $(d\alpha^i,\alpha^i-\alpha^j)$ form
a \v{C}ech-De Rham 2-cocycle on $Y$. It is straightforward to check that
the cohomology class of this cocycle does not depend on the choices we made. By definition,
this cohomology class is the Chern class of the bundle $\widetilde{Y}\rightarrow Y$.

\begin{proof}[Proof of Proposition \ref{Prop_DH}]
Our proof is a slight modification of the original proof by Duistermaat and Heckman.

%
Let $\underline{\Omega}$ stand for the symplectic form on $X\widehat{\red} G/\g^*$.
We need to check that $\xi. [\underline{\Omega}]=\langle c,\xi\rangle$ for any $\xi\in\g^*$.

Let $\widehat{\X}$ denote the formal neighborhood $\mu^{-1}(0)$ in $\X$. Consider a covering $\widehat{\X}=\bigcup_i \widehat{\X}^i$ by open affine  $G$-stable formal subschemes. Fix an identification
$\widehat{\X}^i\cong Y^i\times (\g^*)^\wedge_0$, where $Y^i$ is an open affine subvariety in
$\mu^{-1}(0)$. Moreover, shrinking $Y^i$'s if necessary, we may assume that
the restriction of the principal bundle $\mu^{-1}(0)\rightarrow X_0$ to $Y^i/G$ is trivial. So
$Y^i:=Y^i_0\times G\times (\g^*)^\wedge_0$. The formal scheme $G\times (\g^*)^\wedge_0$ can be thought
as the completion $(T^*G)^\wedge_G$ of the cotangent bundle $T^*G$ along the base $G$.

Following the definition of the Gauss-Manin connection recalled in Subsection \ref{SUBSECTION_def_classic},
to compute $\xi.[\underline{\Omega}]$ we need to lift $\underline{\Omega}$ to a \v{C}ech-De Rham cochain
on $\X\widehat{\red}G$. Let $\pi$ denote the quotient morphism $\widehat{\X}\rightarrow \X\widehat{\red}G$
and $p_i$ be the projection $\widehat{\X}^i\rightarrow (T^*G)^\wedge_G$ induced by the decomposition
$\widehat{\X}^i= Y^i_0\times (T^*G)^\wedge_G$ introduced in the previous paragraph. Further, let $\gamma_i$
stand for the canonical 1-form on $T^*G$  (so that $d\gamma_i$ is the natural symplectic form on $T^*G$).

Consider the 2-form $\omega- p_i^*(d\gamma_i)$ on $\widehat{\X}^i$. It is easy to see that this form is $G$-invariant and vanishes on the vector fields tangent to the fibers of $\pi$. So there is a unique 2-form $\omega_i$ on
on $Y^i=Y_0^i\times (\g^*)^\wedge_0$ such that $\pi^*\omega_i=\omega-p_i^*(d\gamma_i)$.
Also it is easy to see that the restriction of $\omega_i$ to the vector fields tangent to the fibers
of the projection $Y^i\rightarrow (\g^*)^\wedge_0$ coincides with $\underline{\Omega}$.
So the cochain $(\omega_i,0,0)$ represents $\underline{\Omega}$.

Clearly, $d\omega_i=0$, so the \v{C}ech-De Rham differential of $(\omega_i,0,0)_i$ is nothing else but $(0,\omega_i-\omega_j,0)_{ij}$. Now $\pi^*(\omega_i-\omega_j)=dp_i^*(\gamma_i)-dp_j^*(\gamma_j)$. The form $p_i^*(\gamma_i)$
is $G$-equivariant. Moreover, for $\xi\in \g^*,\eta\in \g$ we have $dp_i^*(\gamma_i)(\xi,\eta_X)=d\gamma_i(\xi,\eta_G)=\langle\xi,\eta\rangle$.
Analogously to the original proof in \cite{DH}, we see that the map $\g^*\rightarrow \Omega^1(Y^i)$
given by $\xi\mapsto \iota_\xi dp_i^*(\gamma_i)$ is a connection form on $Y^i\times (\g^*)^\wedge_0$.
Now the claim of the proposition follows from the definition of the Chern class recalled above.
\end{proof}

\subsection{Quantum reduction: algebra level}\label{SUBSECTION_quant_red_alg}
Let $\A$ be an associative algebra, $\g$ be a Lie algebra equipped with a Lie algebra homomorphism
$\Phi:\g\rightarrow \A$. Similarly to Subsection \ref{SUBSECTION_class_red}, we can define the spaces
$\A\red_\chi \g, \A\red_U \g,\A\red \g$ and $\A\red_\chi G, \A\red_U G, \A\red G$ (for a connected
group $G$ with Lie algebra $\g$). We remark that all these spaces have natural algebra structures.
For example, $\A\red_U \g, \A\red_U G$ are algebras over $\K[U]$. The algebras above are called
{\it quantum Hamiltonian reductions} of $\A$.

We are mostly interested in the following special case. Suppose that the algebra $\A$ comes
equipped with an exhaustive increasing filtration  $\F_i \A, i\in \Z$. Suppose that $[\F_i \A, \F_j \A]\subset
\F_{i+j-2}\A$ for all $i,j$ and that $\operatorname{im}\Phi\subset \F_2\A$.
Then $A:=\gr \A$ is a graded commutative algebra and has a natural Poisson
bracket of degree $-2$. Then the quantum reductions $\A\red_\chi \g$ etc. inherit a filtration from
$\A$. On the other hand, set $\varphi:=\gr\Phi: \g\rightarrow \A$. The algebras $A\red_0 \g, A\red_U\g$
are graded with the Poisson brackets of degree $-2$.

We say that {\it quantization commutes with reduction}  for $\chi$ if the following two conditions hold:
\begin{enumerate}
\item $\gr \A\Phi(\g_\chi)=A\varphi(\g)$.
\item Any $\g$-invariant in $A/A\varphi(\g)$ lifts to a $\g$-invariant in $\A/\A\Phi(\g_\chi)$.
\end{enumerate}
One can give a similar definition for $U$. Clearly, if quantization commutes with reduction,
then $\gr \A\red_\chi \g=A\red_0\g, \gr\A\red_U\g=A\red_U\g$.

Here are some conditions that guarantee that quantization commutes with reduction.

\begin{Lem}\label{Lem:quant_comm_red}
Let $G$ be an algebraic group that acts on $\A$ rationally by filtration
preserving automorphisms. Suppose that $A=\gr \A$ is finitely generated.
Let $U\subset\z$ be a subspace. Suppose that the elements $\varphi(\xi_1),\ldots, \varphi(\xi_k)$
form a regular sequence in $A$, where $\xi_1,\ldots,\xi_k$ is a basis in $U^\perp\subset \g$.
Then (1) holds.
Finally, suppose that one of the following conditions holds:
\begin{itemize}
\item[(A)] $G$ is reductive.
\item[(B)] The $G$-action on $\mu^{-1}(U)$ (where $\mu:\Spec(A)\rightarrow\g^*$
is the moment map) is free, and $A\red_U G$ is finitely generated.
\end{itemize}
Then (2) holds.
\end{Lem}
\begin{proof}
(1) is pretty
standard, see, for instance, proof of Lemma 3.6.1 in \cite{Wquant}. (A) easily implies (2).
If (B) holds, then (2) also is pretty standard, compare with
\cite{Miura}, proof of Proposition 3.4.1.
\end{proof}

We will need a ramification of the above construction. Namely, let $\A_h$ be a
flat $\K[h]$-algebra such that $A:=\A_h/(h)$ is commutative. Suppose that $\K^\times$ acts on $\A_h$ by automorphisms
such that $t.h=t^2 h$ for all $t\in \K^\times$. Let $\A_h$ be equipped with a
$\K[h]$-linear map $\Phi_h:\g\rightarrow \A_h$ with $t.\Phi_h(\xi)=t^2\Phi_h(\xi)$ for all
$\xi\in\g$. On $\A_h$  we have a Lie bracket $[\cdot,\cdot]_h$ given by
$[a,b]_h=\frac{1}{h}[a,b]$. We suppose, in addition, that $\Phi_h$ is a Lie algebra
homomorphism. For $\chi\in(\g/[\g,\g])^*$ form the shift
$\g_{\chi h}:=\{\xi-\langle \chi,\xi\rangle h\}\subset \g\oplus \K h$ and
set $\A_h\red_{\chi h}\g:=(\A_h/\A_h \Phi_h(\g_{\chi h}))^\g$. The reduction
$\A_h\red_U \g$ is defined in the same way as above.

The quantization commutes with reduction condition is stated analogously to the filtered
situation (taking the associated graded should be replaced with taking the quotient
by $h$).

The relation between the $\K[h]$-algebra setting and the filtered algebra setting is as follows.
Suppose the $\K^\times$-action comes from a grading on $\A_h$. Then $\A:=\A_h/(h-1)$ is a filtered
algebra with $\gr \A=\A_h/(h)$. The homomorphism $\Phi_h:\g\rightarrow \A_h$
induces a homomorphism $\Phi:\g\rightarrow \A$.
It is straightforward to see that $\left[\A_h\red_{\chi h}\g\right]/(h-1)=\A\red_\chi \g$ and
$\A_h\red_{U}\g/(h-1)=\A\red_U\g$.

We remark that there is a way to replace a filtered algebra with a graded algebra over the polynomial
ring: the Rees algebra construction. However, in this construction the independent variable is
of degree 1, not 2 that we need. Of course, $\A_h$ is closely related to the Rees algebra of $\A$
but we will not need this relation in the sequel.

\subsection{Quantum reduction: sheaf level}\label{SUBSECTION_quant_red_sheaf}
Let $\X,G$ be as in Subsection \ref{SUBSECTION_class_red}. We suppose
that $\K^\times$ acts on $\X$ as in the end of that subsection.
Let $\Dcal$ be a quantization of $\X$. For convenience we suppose that $\Dcal$ is graded.

Suppose further that $\Dcal$ is $G$-equivariant, see Subsection \ref{SUBSECTION_Basic_def}.
The  $G$-action defines a Lie algebra homomorphism $\g\rightarrow \Der_{\K[[h]]}(\Dcal),\xi\mapsto\xi_{\Dcal}$.
Finally, we suppose that the $G$-action on $\Dcal$ admits a quantum comoment
map, i.e., a $G$-equivariant linear map $\Phi:\g\rightarrow \Gamma(X,\Dcal)$ such that
\begin{itemize}
\item
$\frac{1}{h}[\Phi(\xi),\cdot]=\xi_{\Dcal}$,
\item the image of $\Phi(\xi)$ in
$\Str_X$ coincides with $\mu^*(\xi)$,
\item $\Phi(\xi)$ has degree 2 with respect to the
$\K^\times$-action.
\end{itemize}

%

We can sheafify the constructions of the previous paragraph. From Lemma \ref{Lem:quant_comm_red}
we see that the sheaf $\Dcal\red G$ on the $\z$-scheme $X\red G$ is a graded quantization.
Similarly, we can define the graded quantizations $\Dcal\red_{\chi h}G,\Dcal\red_U G$.

\begin{Ex}\label{Ex_diff_op}
Let $G$ be an algebraic group, $H$ be its algebraic subgroup. Then $\Dcal_h(G/H)$ is naturally identified
with $\Dcal_h(G)\red_0 H$. This is well known for the usual differential operators. The proof
in the homogenized setting is similar. 
\end{Ex}

Finally, we will need the completion $\Dcal\widehat{\red}G$ of $\Dcal\red G$ along $X\red_0 G$. By definition,
this completion is the inverse limit $\varprojlim \left[\Dcal\red G\right]/ \mathcal{J}^k$, where $\mathcal{J}$ is the kernel
of the composition $\Dcal\red G\twoheadrightarrow \Str_{X\red G}\twoheadrightarrow \Str_{X\red_0 G}$. One can view
$\Dcal\widehat{\red} G$ as a quantization of the  scheme $X\widehat{\red} G/ \z^{\wedge_0}$.

\section{Symplectic resolutions}\label{SECTION_sympl_resol}
\subsection{Generalities}\label{SUBSECTION_sympl_general}
We start by recalling the  definition of a symplectic resolution.

Let $X_0$ be a Poisson  (possibly singular) normal affine variety. Further, suppose that $X_0$ is a scheme over a smooth affine variety $S$ and the subalgebra $\K[S]\subset \K[X_0]$ is central. Suppose that the restriction of the Poisson
bivector to $X^{reg}_0$ is a non-degenerate section of $\bigwedge^2 \mathcal{T}_{X_0^{reg}/S}$, let $\Omega_0$
be the corresponding 2-form in $\Omega^2(X_0^{reg}/S)$. A smooth $S$-scheme $X$ equipped with a morphism
$\pi:X\rightarrow X_0$ over $S$ and with a fiberwise symplectic form $\Omega\in \Omega^2(X/S)$ is said to be
a {\it symplectic resolution} of $X_0$ if
\begin{enumerate}
\item $\pi$ is projective and is birational fiberwise.
\item $\pi^*(\Omega_0)=\Omega$.
\end{enumerate}

Of course, when $S$ is a point, we get the usual definition of a symplectic resolution.

We will need some general properties of symplectic resolutions. The following lemma is very
standard.

\begin{Lem}\label{Lem:K_X}
Let $X,X_0,S$ be as above. Then $H^i(X,\Str_X)=\{0\}$ for $i>0$.
\end{Lem}
\begin{proof}
Let $\rho:X\rightarrow S$ denote the structure morphism and set $X_s:=\rho^{-1}(s), s\in S$.
Each $X_s$ is a symplectic resolution of a normal affine variety. So thanks to the
Grauert-Riemenschneider theorem, $H^i(X_s, \Str_{X_s})=\{0\}$. It follows that the fiber of
$R^i_* \rho(\Str_X)$ at $s$ coincides with   $H^i(X_s, \Str_{X_s})$ and hence is zero.
Since $S$ is affine, we are done.
\end{proof}

The main application of this lemma for us is as follows. Let $\Dcal$ be a quantization of
$X/S$. Then using Lemma \ref{Lem:K_X} it is easy to show that $H^i(X, \Dcal)=0$ for  $i>0$, while $\Gamma(X,\Dcal)/(h)=\K[X]$.

\subsection{Quiver varieties}\label{SUBSECTION_quiver_var}
The basic notation related to quivers was introduced in Subsection \ref{SUBSECTION_conventions}.

To a quiver $Q$ one  assigns its double quiver $\D Q=(\D Q_0,\D Q_1)$ with
$\D Q_0:=Q_0$ and $\D Q_1:=Q_1\sqcup Q_1^{op}$, where $Q_1^{op}$ is identified
with $Q_1$, an element of $Q_1^{op}$ corresponding to $a\in Q_1$ is usually denoted by $a^*$.
The maps $t,h:Q_1\rightarrow Q_0$ in the double quiver are as before, while $t(a^*)=h(a),h(a^*)=t(a), a\in Q_1^{op}$.
In a sentence, $\D Q$ is obtained from $Q$ by adding a reverse arrow for any arrow in $Q$.

Fix non-negative integers $v_i,d_i, i\in Q_0$. Set $\vv:=(v_i)_{i\in Q_0}, \ww:=(d_i)_{i\in Q_0}$. Following Nakajima consider the spaces
\begin{align*}
R(Q,\vv,\ww):=&\bigoplus_{a\in Q_1}\Hom(V_{t(a)},V_{h(a)})\oplus \bigoplus_{i\in Q_0}\Hom(D_i,V_i),\\
R(\D Q,\vv,\ww):=&\bigoplus_{a\in Q_1}(\Hom(V_{t(a)},V_{h(a)})\oplus \Hom(V_{h(a)},V_{t(a)}))\oplus \bigoplus_{i\in Q_0}(\Hom(D_i,V_i)\oplus \Hom(V_i,D_i))\\&=R(Q,\vv,\ww)\oplus R(Q,\vv,\ww)^*,
\end{align*}
where $V_i,D_i, i\in Q_0,$ are vector spaces of dimensions $v_i,d_i$, respectively. Throughout
the paper we skip $\ww$ from the notation if $\ww=0$ and write $R(Q,\vv)$ instead of $R(Q,\vv,0)$, etc.

Being identified with $R(Q,\vv,\ww)\oplus R(Q,\vv,\ww)^*$, the space $R(\D Q,\vv,\ww)$ is symplectic, let $\omega$ stand for the symplectic form. The group $\GL(\vv):=\prod_{i\in Q_0}\GL(v_i)$ acts naturally on
$R(\D Q,\vv,\ww)$. The map $\mu:R(\D Q,\vv,\ww)\rightarrow \gl(\vv)^*\cong \gl(\vv)$ sending
$(A_a,B_a,\Gamma_i,\Delta_i), A_a\in \Hom(V_{t(a)},V_{h(a)}), B_a\in
\Hom(V_{h(a)},V_{t(a)}), \Gamma_i\in \Hom(D_i,V_i),\Delta_i\in \Hom(V_i,D_i)$ to
$$\left(\sum_{a, i=h(a)}A_aB_{a^*}-\sum_{a,t(a)=i}B_{a^*}A_a+ \Gamma_i\Delta_i\right)_{i\in Q_0}$$ is a moment map
for the $\GL(\vv)$-action. Set $\z:=(\gl(\vv)/[\gl(\vv),\gl(\vv)])^*$.
We identify $\z$ with $\K^{Q_0}$ by setting $\langle\chi, (\xi_i)_{i\in Q_0}\rangle\mapsto
\sum_{i\in Q_0}\chi_i\tr(\xi_i)$.

Below for $U\subset \z$ we set $\Lambda_U(\D Q,\vv,\ww):=\mu^{-1}(U)$.
We write $\Lambda(\D Q,\vv,\ww)$ for $\Lambda_\z(\D Q, \vv,\ww)$
and $\Lambda_\chi(\D Q,\vv,\ww)$ for the fiber of $\chi\in \z$ of the natural map
$\Lambda(\D Q,\vv,\ww)\rightarrow \z$.

So one can form the quotient $$\M(\D Q,\vv,\ww):=R(\D Q,\vv,\ww)\red \GL(\vv)=\Lambda(\D Q,\vv,\ww)\quo \GL(\vv).$$ This is a Poisson algebraic variety to be referred to as a (universal) affine quiver variety. Also
we can form the reductions $\M_\chi(\D Q,\vv,\ww),\M_U(\D Q,\vv,\ww)$.

These reductions were studied by Crawley-Boevey, see, in particular, \cite{CB1},\cite{CB2}.
In fact, he worked only with the case when $\ww=0$. However, the general case can be reduced
to this one by using the following construction. Consider the quiver $Q^\ww=(Q^\ww_0,Q^\ww_1)$,
where $Q^{\ww}_0:=Q_0\sqcup\{s\}$, where $s$ is a new vertex, and $Q^{\ww}_1$ is the disjoint union of $Q_1$ and the set
$Q_1':=\{a^{ij}, i\in Q_0, j=1,\ldots, d_i\}$ with $t(a^{ij})=i, h(a^{ij})=s$. Then
$R(Q, \vv,\ww)=R(Q^{\ww},\vv^\ww)$, where $\vv^\ww:=(\vv,\epsilon_s)$,
$\epsilon_s$ being the coordinate vector at the vertex $s$. The group $\GL(\vv)$ is included
naturally into $\GL(\vv^\ww)$. Moreover, $\GL(\vv^\ww)=\K^\times\times \GL(\vv)$, where
$\K^\times=\{(x\cdot 1, x\in \K^\times\}$ with $1$ standing  for the unit in $\GL(\vv^\ww)$.
The subgroup $\K^\times$ acts trivially on $R(\D Q^{\ww},\vv^\ww)$ and so the action of $\GL(\vv^\ww)$
reduces to that of $\GL(\vv)$. It follows that $\Lambda_U(\D Q^{\ww},\vv^\ww)=\Lambda_U(\D Q,\vv,\ww), \M_U(\D Q^\ww,\vv^\ww)=\M_U(\D Q,\vv,\ww)$ etc.

Following Crawley-Boevey, let us describe the algebro-geometric properties of the varieties
$\M_\chi(\D Q,\vv), \M_{U}(\D Q,\vv)$.
Define a quadratic function $p:\K^{Q_0}\rightarrow \K, \alpha=\sum_{i\in Q_0}\alpha_i \epsilon_i\mapsto 1-\sum_{i\in Q_0}\alpha_i^2+ \sum_{a\in Q_1}\alpha_{t(a)}\alpha_{h(a)}$.

\begin{Prop}\label{Prop:CB}
\begin{enumerate}
\item Suppose that the following condition holds: for any decomposition of $\vv$ into the sum
$\vv=\vv^1+\ldots+\vv^k,k>1,$ of positive roots $\vv^1,\ldots,\vv^k$ of $Q$
the inequality $p(\vv)\geqslant \sum_{i=1}^k p(\vv^i)$ holds. Then the moment map $\mu$
is flat, the schemes $\Lambda_\chi(\D Q,\vv),\Lambda_U(\D Q,\vv)$ are non-empty  complete intersections
provided $\chi\cdot \vv=0$. Moreover, if $p(\vv)>\sum_{i=1}^k p(\vv^i)$ for any decomposition as above, then the schemes $\Lambda_\chi(\D Q,\vv),\Lambda_U(\D Q,\vv)$ are reduced and irreducible, and each scheme $\Lambda_\chi(\D Q,\vv)$ (with $\chi\cdot \vv=0$) contains a closed
$\GL(\vv)$-orbit with the stabilizer $\K^\times=\{x\cdot 1, x\in \K^\times\}$.
\item The varieties $\M_\chi(\D Q,\vv), \M_U(\D Q,\vv)$ are normal.
\end{enumerate}
\end{Prop}
\begin{proof}
The statements for $\Lambda_\chi,\M_\chi$ were proved by Crawley-Boevey: (1) in \cite{CB1},
Theorems 1.1, 1.2, and (2) in \cite{CB2}. The claim in (1) that $\Lambda_U$ is reduced
follows from the reducedness of the $\Lambda_\chi$'s. The irreducibility for $\Lambda_U$
follows from the irreducibilty and the reducedness for $\Lambda_0$, thanks to the contracting action
of $\K^\times$ on $\Lambda_U$. A similar argument proves the normality.
\end{proof}

Now we turn to non-affine quiver varieties.
Namely, consider a character $\theta$ of $\GL(\vv)$. Recall that a point
$x\in R(\D Q,\vv,\ww)$ is said to be {\it $\theta$-semistable} if there is
a homogeneous polynomial $f\in \K[R(\D Q,\vv,\ww)]$ with $f(x)\neq 0$
that is $\GL(\vv)$-semiinvariant
of weight that is a positive multiple of $\theta$.
Geometric Invariant Theory implies that there is a categorical quotient $\M^\theta(\D Q,\vv,\ww)$ for the $\GL(\vv)$-action on $\Lambda(\D Q,\vv,\ww)^{\theta,ss}$ such that the quotient
morphism $\Lambda(\D Q,\vv,\ww)^{\theta,ss}\rightarrow \M^\theta(\D Q,\vv,\ww)$
is affine. The latter implies that $\M^\theta(\D Q,\vv,\ww)$ has a natural Poisson structure.
This structure is symplectic (over $\z$) whenever the action of $\GL(\vv)$ on $\Lambda(\D Q,\vv,\ww)^{\theta,ss}$
is free.   Introduce the Poisson schemes $\M^\theta_\chi(\D Q,\vv,\ww),\M^\theta_U( \D Q,\vv,\ww)$
in a similar way.

It is a standard fact from Geometric Invariant Theory that there are natural projective morphisms
\begin{equation}\label{eq:morphisms}
\M^\theta_\chi(\D Q,\vv,\ww)\rightarrow \M_\chi(\D Q,\vv,\ww),
\M^\theta_U(\D Q,\vv,\ww) \rightarrow \M_U(\D Q,\vv,\ww).\end{equation}

Nakajima in \cite{Nakajima} found some conditions on $\theta$ guaranteeing  that the $\GL(\vv)$-actions on the varieties  $\Lambda_\chi(\D Q,\vv,\ww)^{\theta,ss}$ are free and the projective morphisms
(\ref{eq:morphisms}) are birational (and so are symplectic resolutions). His results are summarized below.

To state the proposition we will need to recall some definitions. Let $R_+$ denote the system of positive
roots of the quiver $Q$. Recall that $\z$ is identified with $\K^{Q_0}$.  We say that $\theta\in \K^{Q_0}$
is {\it generic} if $\theta\cdot \alpha\neq 0$ for any $\alpha\in R_+$.

The following statement follows from  \cite{Nakajima}, Theorems 2.8,3.2,4.1.

\begin{Prop}\label{Prop:Nakajima1}
Suppose $\theta$ is generic. Then  the $\GL(\vv)$-action on $\Lambda(\D Q,\vv,\ww)^{\theta,ss}$ is free,
and the morphisms (\ref{eq:morphisms})   are birational provided $\Lambda_{0}(\D Q,\vv,\ww)^{\theta,ss}\neq \varnothing$.
\end{Prop}

For example, suppose that
\begin{equation}\label{eq:character}\theta((X_i)_{i\in Q_0})=\prod_{i\in Q_0}\det(X_i)^{-1}\end{equation} The character $\theta$ is generic. An element
$(A_a,B_{a^*}, \Gamma_i,\Delta_i)$ is semistable if there are no nonzero subspaces $V_i'\subset \ker \Delta_i$
such that the collection $(V_i')_{i\in Q_0}$ is stable under all $A_a,B_{a^*}$.
It is clear that the $\GL(\vv)$-action on  $\Lambda(\D Q,\vv,\ww)^{\theta,ss}$
is free.

%

In fact, Nakajima's results also allow to determine the number of irreducible
 components of $\Lambda_{0}(\D Q,\vv,\ww)^{\theta,ss}$.
Namely, assume that the quiver $Q$ is either of finite type or affine. Let $\g(Q)$ be the corresponding
(finite dimensional semisimple or affine) Kac-Moody algebra. Then we can view $\ww$ as an element of the weight
lattice of $\g(Q)$, the corresponding weight is $\mathrm{d}:=\sum_{i\in Q_0}d_i \omega_i$, where $\omega_i,i\in Q_0,$
are fundamental weights. Also to $\vv$ we assign an element $\mathrm{v}:=\sum_{i\in Q_0}v_i \epsilon_i$
in the root lattice.

The next proposition follows from Theorem 10.16 in \cite{Nakajima}.

\begin{Prop}\label{Prop:quiver_non_empty}
Let $\theta$ be generic.
Consider the irreducible highest weight module $L(\mathrm{d})$ of $\g(Q)$. The number of irreducible components in $\Lambda_{0}(\D Q,\vv,\ww)^{\theta,ss}$ coincides with the dimension
of the weight space of weight $\mathrm{d}-\mathrm{v}$ in $L(\mathrm{d})$.
\end{Prop}

Finally, let us discuss the quantum analogs of quiver varieties.


Let $\W_h(=\W_{h}(V^*))$ denote the homogeneous Weyl algebra of $V^*$, i.e.,
the quotient of the tensor algebra $T(V^*)[h]$
by the relations $u\otimes v-v\otimes u-h\omega(u,v), u,v\in V^*$, where $\omega$ denotes the symplectic form on
$V^*$. This algebra inherits a grading from $T(V^*)[h]$ with
$\W_h/h\W_h=\K[V]$.

Set $V:=R(\D Q,\vv,\ww)$. Clearly, the group $G:=\GL(\vv)$ acts on $\W_h$ by graded algebra automorphisms.
Further, the comoment map $\mu^*: \g\rightarrow \K[V]$ naturally lifts
to a quantum comoment map $\Phi: \g\rightarrow \W_h$. Namely, we have a unique (up to a scalar factor) $\operatorname{Sp}(V)$-equivariant
homomorphism $\mathfrak{sp}(V)\rightarrow \W_h$. The map $\Phi$ is obtained by restricting an
this  homomorphism to $\g\subset \mathfrak{sp}(V)$. So we can define the quantum Hamiltonian
reductions $\W_{\chi h}(\D Q, \vv,\ww)_h:=\W_h\red_{\chi h}G, \W_U(\D Q,\vv,\ww)_h:=\W_h\red_U G$
as in Subsection \ref{SUBSECTION_quant_red_alg}.

We can also consider the sheaf version of this construction.
Namely, consider the deformation quantization $\W_{V^*,h}$ of
$V$ obtained by localizing (the $h$-adic completion of) the algebra $\W_h$.
This quantization is $G$-equivariant and graded.
Now let us compare the algebra $\W_U(\D Q,\vv,\ww)_h$ with the algebra of global
sections of $\W_{V^*,h}\red_U^\theta G:= \W_{V^*,h}|_{V^{\theta,ss}}\red_U G$.

\begin{Lem}\label{Lem:global_sections}
Suppose that
\begin{itemize}
\item[(i)] quantization commutes with reduction for $\W_U(\D Q,\vv,\ww)_h$,
\item[(ii)] the action of $G$ on $\Lambda_U(\D Q,\vv,\ww)^{\theta,ss}$ is free,
\item[(iii)] and the morphism $\M^\theta_U(\D Q,\vv,\ww)\rightarrow \M_U(\D Q,\vv,\ww)$
is a symplectic resolution of singularities.\end{itemize}
Then the natural morphism
\begin{equation}\label{eq:natural_hom11}\W_h^G\rightarrow
\Gamma(\M_U^\theta(\D Q,\vv,\ww),\W_{V^*,h}\red_U^\theta G)\end{equation}
gives rise to an isomorphism between $\W_U(\D Q,\vv,\ww)_h$ and the subalgebra
of $\K^\times$-finite elements in $\Gamma(\M_U^\theta(\D Q,\vv,\ww),\W_{V^*,h}\red_U^\theta G)$.
\end{Lem}
\begin{proof}
Let us construct a homomorphism $\W_U(\D Q, \vv,\ww)_h^{\wedge_h}\rightarrow \Gamma(\M_U^\theta(\D Q,\vv,\ww),\W_{V^*,h}\red_U^\theta G)$, where the superscript $^{\wedge_h}$ means the $h$-adic completion.
Since $G$ is reductive, and its action on $\W_h(V^*)^{\wedge_h}$ is pro-finite, we see that
the natural homomorphism $(\W_h(V^*)^{\wedge_h})^G\rightarrow \W_h(V^*)^{\wedge_h}\red_U G=\W_U(\D Q,\vv,\ww)_h^{\wedge_h}$ is surjective and its kernel coincides with $(\W_h(V)^{\wedge_h})^G\cap \W_h(V)^{\wedge_h}\Phi_h(U^\perp)=
[\W_h(V)^{\wedge_h}\Phi_h(U^\perp)]^G$. However, it is easy to see that this kernel
is contained in the kernel of  (\ref{eq:natural_hom11}). So we have an obviously $\K^\times$-equivariant
and $\K[U][h]$-linear homomorphism $\W_U(\D Q,\vv,\ww)_h^{\wedge_h}\rightarrow \Gamma(V\red_U^\theta G,\W_{V^*,h}\red_U^\theta G)$. We are going to show that this homomorphism is an isomorphism.

We have the following commutative diagram, where the vertical arrows are quotients by $h$.

\begin{picture}(90,30)
\put(4,2){$\K[\M(\D Q,\vv,\ww)]$}
\put(2,22){$\W_U(\D Q, \vv,\ww)_h^{\wedge_h}$}
\put(56,2){$\K[\M^\theta(\D Q,\vv,\ww)]$}
\put(50,22){$\Gamma(V\red_U^\theta G, \W_{V^*,h}\red_U^\theta G)$}
\put(12,20){\vector(0,-1){12}}
\put(60,20){\vector(0,-1){12}}
\put(34,4){\vector(1,0){21}}
\put(30,23){\vector(1,0){19}}
\end{picture}

The bottom horizontal arrow is an isomorphism  since the morphism $\M^\theta_U(\D Q,\vv,\ww)\rightarrow \M_U(\D Q,\vv,\ww)$
because of (iii). The left vertical arrow is surjective because
of (i). Thanks to (ii), $\W_{V^*,h}\red_U^\theta G$ is a quantization of $V\red_U^\theta G$.
The right vertical arrow is the quotient by $h$ because of (iii) and the remarks in the end of
Subsection \ref{SUBSECTION_sympl_general}.  So the top horizontal arrow is surjective modulo $h$ and hence is genuinely surjective.
Since the sheaf $\W_{V^*,h}\red_U^\theta G$ is $\K[[h]]$-flat, we see that  $\Gamma(\M^\theta(\D Q,\vv,\ww),
\W_{V^*,h}\red_U^\theta G)$ is a flat $\K[[h]]$-algebra. Using this and the observation that the bottom arrow
is an isomorphism we see that the top horizontal arrow is injective.

To complete the proof we notice that, since the grading on $\W_U(\D Q,\vv,\ww)_h$ is positive,
the algebra $\W_U(\D Q,\vv,\ww)_h$ coincides with the subalgebra of $\K^\times$-finite elements
in $\W_U(\D Q,\vv,\ww)_h^{\wedge_h}$.
\end{proof}

We note that instead of $\M^\theta_U(\D Q,\vv,\ww)$ in the previous lemma we could consider
the formal neighborhood $V\widehat{\red}^\theta_U G$ of $\M^\theta_0(\D Q,\vv,\ww)$
in $\M^\theta(\D Q,\vv,\ww)$. The claim is still that $\W(\D Q,\vv,\ww)_h$ is the algebra
of $\K^\times$-finite elements in the algebra of global sections of $\W_{h,V^*}\widehat{\red}^\theta G$.

\subsection{Slodowy varieties}\label{SUBSECTION_Slodowy}
In this subsection we will define {\it Slodowy varieties} (tracing back to \cite{Slodowy}) that are certain smooth symplectic varieties that are related to  {\it Slodowy slices} in reductive Lie algebras.

First of all, let us recall Slodowy slices. Let $G$ be a reductive algebraic group, $\g$ its
Lie algebra, $e\in\g$ a nilpotent element, and $\Orb:=Ge$.
The Slodowy slice $\Sl(=\Sl(e)=\Sl(\Orb))$ associated to $(\g,e)$ is a transverse slice to the adjoint orbit $\Orb$ of $e\in \g$. It is constructed as follows. Pick a semisimple element $h\in\g$ and a nilpotent element
$f\in \g$ forming an $\sl_2$-triple with $e$. Then set $\Sl:=e+\ker\ad(f)$.
In the sequel we will identify $\g$ with $\g^*$ using a symmetric non-degenerate invariant
form $(\cdot,\cdot)$ and will consider
$\Orb,\Sl$ as subvarieties in $\g^*$.  Also set $\chi:=(e,\cdot)$.

The algebra $\K[\Sl]$ has some nice grading (often called the {\it Kazhdan grading}). Namely, let $\gamma:\K^\times\rightarrow G$ be the one-parameter group associated to $h$ (so that $\gamma(t).\xi=t^i\xi$
for $\xi\in\g$ with $[h,\xi]=i\xi$). Consider the $\K^\times$-action on $\g^*$ given by $t.\alpha=t^{-2}\gamma(t)\alpha, t\in \K^\times, \alpha\in \g^*$. It is easy to see that $\lim_{t\rightarrow \infty} t.s=\chi$ for all $s\in \Sl$.
In other words, the grading on $\K[\Sl]$ is positive.

Also thanks to the Kazhdan action, we see that $\Sl$ intersects an orbit $\Orb'\subset\g^*$ if and only if
$\Orb\subset \overline{\K^\times \Orb'}$.

In \cite{slice}, \cite{Wquant} we considered a certain symplectic $G$-variety $\XS$, called
the {\it equivariant Slodowy slice}, whose quotient is naturally identified with $\Sl$. As a variety,  $\XS:=G\times \Sl$. The group $G$ acts on $\XS$ by the left translations: $g.(g_1,s)=(gg_1,s)$. A symplectic form on $X$ is defined as follows. Identify $T^*G$ with $G\times\g^*$ via the trivialization by means of the left-invariant 1-forms.
Then $\XS=G\times \Sl$ is included into $G\times \g^*=T^*G$. It turns out that the restriction of the canonical
symplectic form from $T^*G$ to $\XS$ is non-degenerate. Denote this restriction by $\omega$.
Define the action of  $\K^\times$  on $T^*G$  by $t.(g,\alpha)=(g\gamma(t)^{-1}, t^{-2}\gamma(t)\alpha)$.
The subvariety $\XS\subset T^*G$ is $\K^\times$-stable.

Pick a parabolic subgroup $P\subset G$. By the {\it Slodowy variety} $\Sl(e,P)$ corresponding to $e$ and $P$
we mean the reduction  $\XS\red_0 P$. We will also need the formal deformation $\widetilde{\Sl}(e,P):=\XS\widehat{\red}P$.

It is clear that $\Sl(e,P)$ can be naturally embedded into $T^*G\red_0 P=T^*(G/P)$ (here we consider
the quotient with respect to the action of $P$ from the left). From here it is easy to see $\Sl(e,B)$ is the resolution of singularities of the intersection $\Sl\cap\Nil$ of $S$ with the nil-cone $\Nil\subset \g^*$, compare with \cite{Ginzburg_HC}. For a general parabolic subgroup $P\subset G$, there are  projective morphisms $\Sl(e,P)\rightarrow \Sl\cap \Nil, \XS\red P\rightarrow \Sl$ restricted from $T^*(G/P)\rightarrow \Nil, G*_{P_0}\p^\perp\rightarrow \g^*$, where $P_0$ is the solvable radical of $P$.

Below we will be mostly interested in two cases.

{\bf Case 1.} $\g$ is simple of types $A,D,E$, and  $e$ is a subprincipal nilpotent element in $\g$. The latter means that the orbit
$\Orb$ is of codimension 2 in $\Nil$.

{\bf Case 2.} $G=\SL_n$. In this case the moment map $T^*(G/P)\rightarrow \g^*$ is generically injective and its
image coincides with $\overline{G\p^\perp}$. The latter subvariety is the closure of an appropriate nilpotent
orbit in $\g^*$ (the Richardson orbit of $\p$). So we see that $T^*(G/P)$ is  a symplectic resolution of $\overline{G\p^\perp}$,
while $\Sl(e,P)$ is a symplectic resolution of $\Sl\cap \overline{G\p^\perp}$.

In fact, there is an alternative construction of $\Sl,\XS, \Sl(e,P)$ in terms of a Hamiltonian
reduction.

Consider the grading $\g:=\bigoplus_{i\in \Z}\g(i)$ given by the eigenvalues of $\ad(h)$.
Recall an element $\chi\in \g^*$.
The restriction of the skew-symmetric form $(\xi,\eta)\mapsto \langle\chi,[\xi,\eta]\rangle$
to $\g(-1)$ is non-degenerate. Following  \cite{Kawanaka},\cite{Moeglin},\cite{Premet1},\cite{GG}, pick a lagrangian subspace $l\subset \g(-1)$ and set
$\m:=l\oplus \bigoplus_{i\leqslant -2}\g(i)$.

As Gan and Ginzburg checked in \cite{GG}, $\Sl$ is naturally identified with the reduction
$\g^*\red_\chi M$, where $M$ is the unipotent subgroup of $G$ corresponding to $\m$.
More precisely, let $\rho:\g^*\twoheadrightarrow \m^*$ be the natural projection.
Then the natural map $M\times S\rightarrow \g^*, (m,s)\mapsto ms,$ is an isomorphism
of $M\times S$ and $\rho^{-1}(\chi|_{\m})$. From here it also follows that
$\XS$ is naturally (in particular, $G$- and $\K^\times$-equivariantly
and symplectomorphically) identified with $T^*G\red_\chi M$. Hence $\Sl(e,P)=(T^*G\red_0 P)\red_\chi M=
T^*(G/P)\red_\chi M$.

\subsection{Resolutions of quotient singularities}\label{SUBSECTION_resolution}
Set $L:=\K^2$ and consider a non-zero $\SL_2(\K)$-invariant
form $\omega_0$ on $L$. Equip $L^n$ with the form $\omega=\omega_0^{\oplus n}$ and set $\Gamma_n:=\Gamma^n \rtimes S_n$.
The group $\Gamma_n$ acts naturally on $L^n$ (the symmetric group just permutes the copies
of $L$) and this action preserves $\omega$.

Set $X_0:=L^n/\Gamma_n$. There is a $\K^\times$-action on $X_0$ induced by the action on $\K^{2n}$
given by $(t,v)\mapsto t^{-1}v$.

Consider the set $N_0,\ldots,N_r$ of irreducible $\Gamma$-modules, where $N_0$ is the trivial module. Consider the {\it McKay  quiver}, whose vertices are $0,\ldots,r$ and the number of arrows from $i$ to $j$ equals the dimension of $\Hom_\Gamma(\K^2\otimes N_i,N_j)$. This quiver is known to be the double of an affine Dynkin quiver $Q$.
Moreover, $0$ can be taken for the extending vertex of $Q$.

Let $\delta$ denote the indecomposable positive imaginary
root. Set $\vv:=n\delta, \ww=\epsilon_0$.  Take a generic character $\theta$ of $\GL(\vv)$,
see the discussion before Proposition \ref{Prop:Nakajima1}. Then  there is a $\K^\times$-equivariant isomorphism
$X_0\cong \M_0(\D Q,\vv,\ww)$ of schemes (one uses  \cite{GG_quiver}
to show that the right scheme is reduced and then argues as in the proof of Theorem 11.16 from \cite{EG}).

Let us show that this isomorphism can be made Poisson maybe after rescaling. This stems
from the following proposition proved in \cite{EG}, Lemma 2.23.

\begin{Prop}\label{Prop:symplect_quotient_sing}
There is unique (up to rescaling) Poisson bracket of degree $-2$ on $\K[X_0]$.
Furthermore, there are no (not necessarily Poisson) brackets of degree $i$
with $i<-2$ on $\K[X_0]$.
\end{Prop}

We fix a $\K^\times$-equivariant Poisson isomorphism $X_0\cong \M_0(\D Q,\vv,\ww)$.

According to \cite{GG_quiver}, Theorem 1.4.1, the variety $\Lambda_0(\D Q,\vv,\ww)$ (and hence any $\Lambda_\chi(\D Q,\vv,\ww),$ $\Lambda_U(\D Q,\vv,\ww)$) is a non-empty reduced complete intersection
and the action of $\GL(\vv)$ on each component is generically free.
By Proposition \ref{Prop:Nakajima1}, we have a $\K^\times$-equivariant symplectic resolution
$X:=M^\theta_0(\D Q,\vv,\ww)\rightarrow X_0$.

Consider the graded symplectic deformation $\widehat{X}:=\M^\theta(\D Q,\vv,\ww)$ of $X$.

In the sequel we will need a certain vector bundle on $X$ to be referred to as a {\it weakly Procesi}
bundle whose existence was proved by  Bezrukavnikov and Kaledin in \cite{BK2}. Namely, there is a $\K^\times$-equivariant vector bundle $\Pro$ on $X$ with the following properties:
\begin{itemize}
\item[($P$1)] There is a graded $\K[X]=\K[L^n]^{\Gamma_n}$-algebra isomorphism $\End_{\Str_X}(\Pro)\cong \K[L^n]\#\Gamma_n$.
\item[($P$2)] $\Ext^i_{\Str_X}(\Pro,\Pro)=0$ for $i>0$.
\end{itemize}
In particular, ($P$1) implies that $\Gamma_n$ acts on $\Pro$ fiberwise and each fiber is isomorphic
to $\K\Gamma_n$ as a $\Gamma_n$-module.

Thanks to ($P$2) we can uniquely extend $\Pro$ to a $\K^\times$-equivariant vector bundle $\widehat{\Pro}$
on $\widehat{X}$. This vector bundle  automatically satisfies the following three conditions.
\begin{itemize}
\item[($\widehat{P}$0)] $\End_{\Str_{\widehat{X}}}(\widehat{\Pro})$ is flat over $\K[\z]$.
\item[($\widehat{P}$1)] $\End_{\Str_{\widehat{X}}}(\widehat{\Pro})/(\z)=\K[L^n]\#\Gamma_n$.
\item[($\widehat{P}$2)] $\Ext^i_{\Str_{\widehat{X}}}(\widehat{\Pro},\widehat{\Pro})=0$.
\end{itemize}
The group $\Gamma_n$ still acts on $\widehat{\Pro}$ fiberwise and each fiber is isomorphic to
$\K\Gamma_n$ as a  $\Gamma_n$-module.

\subsection{Kleinian case}\label{SUBSECTION_Klein_classical}
An important special case of the quotient singularity considered in the previous
section is that of the Kleinian singularities, i.e., $n=1$ and $\Gamma_n=\Gamma$.
The reader is referred to \cite{Nakajima_book} for generalities on the Kleinian
singularities and their resolutions.

Set $X_0=\K^2/\Gamma$ and let $\pi:X\rightarrow X_0$ be the minimal resolution of $X_0$. Then $X$ is a symplectic variety with symplectic form, say, $\Omega$.

Let $D_1,\ldots,D_r$ be the irreducible components of the exceptional fiber $\pi^{-1}(0)$.
It is well-known, see, for example, \cite{GSV}, that $D_1,\ldots,D_r$ can be identified with simple roots $\alpha_1,\ldots,\alpha_r$ of a  simple root system of type $A,D,E$. Moreover, the intersection pairing between $D_i,D_j$ equals $-a_{ij}$,
where $a_{ij}=\langle\alpha_i^\vee,\alpha_j\rangle$ is the corresponding entry of the Cartan matrix.

As we have seen in the previous subsection, one can construct $X_0$ and $X$ as  quiver varieties.
The quiver $Q$ is the affine quiver of the Dynkin diagram of $\alpha_1,\ldots,\alpha_r$.

Alternatively, $X$ can be realized as a Slodowy variety.
Namely, let $\g$ be the simple Lie algebra with system $\alpha_1,\ldots,\alpha_r$ of simple roots, $G$
the corresponding simply connected group. Let $e\in \g$ be a subprincipal nilpotent element and construct
the Slodowy slice $\Sl$ and the equivariant Slodowy slice $\XS$ from $e$. Let $B$ denote the Borel
subgroup of $G$ corresponding to the choice of $\alpha_1,\ldots,\alpha_r$.

According to Brieskorn, the intersection
$S\cap \Nil$ is isomorphic to $\K^2/\Gamma$. The isomorphism can be made $\K^\times$-equivariant. This follows,
for instance, from the construction explained in \cite{Slodowy}. By Proposition \ref{Prop:symplect_quotient_sing},
we may assume, in addition, that an isomorphism is Poisson.

Consider the resolution $\pi:\Sl(e,B)\rightarrow \Sl\cap\Nil=X_0$.
Its exceptional fiber again consists of $r$ divisors,
say $D_i',i=1,\ldots,r,$ that are in one-to-one correspondence with the set of simple roots of $\g$.
More precisely, consider the line bundle $\mathcal{L}_i$ on $G/B$ corresponding to the fundamental weight
$\omega_i$ (given by $\langle\omega_i,\alpha_j^\vee\rangle=\delta_{ij}$).
Lift the bundles $\mathcal{L}_i$ to $T^*(G/B)$ and then restrict them $\Sl(e,B)\subset T^*(G/B)$.
It is known that the restriction is the line bundle corresponding a unique component $D_i'$ of the exceptional
divisor. Moreover, the intersection pairing between $D_i',D_j'$ is again $-a_{ij}$.

Now both $\Sl(e,B), \M^\theta_0(\D Q,\vv,\ww)$ are symplectic (=minimal) resolutions of $X_0=\Sl\cap\Nil=\M_0(\D Q,\vv,\ww)$. There is only one minimal resolution of $X_0$.
So there is a isomorphism $\varphi:\Sl(e,B)\rightarrow \M_0^\theta(\D Q,\vv)$ of schemes over $X_0$.
This isomorphism is automatically a $\K^\times$-equivariant symplectomorphism.
We may assume, in addition, that $\varphi(D_i')=D_i$. If not, we can enumerate $D_i$'s
differently.

We will need to understand the behavior of some natural line bundles on the varieties $\Sl(e,B),
\M^\theta_0(\D Q,\vv,\ww)$ under the isomorphism $\varphi$.
On the Slodowy side we have line bundles $\mathcal{L}_i$ corresponding to the fundamental weights.
On the quiver variety side we also have $r+1$ line bundles constructed as follows.
Let $\mathcal{L}_i', i=0,\ldots,r,$ be the line bundle on $\M^\theta(\D Q,\vv,\ww)$
induced by the 1-dimensional $\GL(\vv)$-module $\bigwedge^{\delta_i} N_i$, compare with the last paragraph of Subsection \ref{SUBSECTION_class_red}.

\begin{Prop}\label{Prop:line_bundle_transform}
Let $C=(a_{ij})_{i,j=1}^r$ be the Cartan matrix of $Q$. Then $\mathcal{L}'_i=\prod_{j=1}^r \varphi_*(\mathcal{L}_i)^{a_{ij}}$ for $i=1,\ldots,r$.
\end{Prop}
The proposition simply means that if we identify the free group spanned by $\mathcal{L}_i$ (in fact,
this group coincides with $\operatorname{Pic}(X)$) with the weight lattice of $Q$ by sending
$\mathcal{L}_i$ to $\omega_i$, then the bundles $\mathcal{L}_i'$ get identified with the simple roots.
\begin{proof}
Consider the bundle $\mathcal{N}_i$ on $X\cong R(\D Q,\vv,\ww)^{\chi,ss}\red_0 \GL(\vv)$
associated to $N_i$. Gonzales-Sprinberg and Verdier in \cite{GSV} computed the first Chern classes $c_1(\mathcal{N}_i)$
of $\mathcal{N}_i$ (of course, $c_1(\mathcal{N}_i)=c_1(\mathcal{L}_i')$). The required
result is the direct corollary of their computation (and the well-known fact that $H^2_{DR}(X)=\K\otimes_Z \operatorname{Pic}(X)$).
\end{proof}

Finally, we will need an explicit construction of the bundle $\widehat{\Pro}$ on $\widehat{X}$.
Namely, consider the $\GL(\vv)$-module $P:=\bigoplus_{i=0}^l N_i^{\oplus\delta_i}$. Let
$\widehat{\Pro}$ denote the corresponding bundle on $\widetilde{X}$.

Consider the restriction $\Pro$ of $\widehat{\Pro}$ to $X$. It is known, see, for example,
Section 1.5 in \cite{KaVa},
that $\Pro$ satisfies ($P$1),($P$2). Being a $\K^\times$-equivariant extension
of $\Pro$, the bundle $\widehat{\Pro}$ satisfies ($\widehat{P}$0),($\widehat{P}$1),($\widehat{P}$2).

\subsection{Quiver varieties in type $A$ vs Slodowy varieties}\label{SUBSECTION_A_quiver_vs_Slod}
In \cite{Maffei} Maffei  established isomorphisms between quiver varieties and Slodowy varities in type $A$.
In this subsection we are going to recall his construction and deduce some of its easy corollaries. We remark
that results closely related to Maffei's were also obtained in \cite{MV}.

First, let us fix some notation. Let $N$ be a positive integer and $\g=\sl_N$. Fix $n>0$ and
 $r_1,\ldots,r_n\in \Z_{\geqslant 0}$ with $\sum_{i=1}^n r_i=N$. The numbers $r_1,r_2,\ldots,r_{n}$ define
a parabolic subgroup $P$ in $G:=\SL_N$, namely, for $P$ we take the stabilizer of a partial flag $\mathcal{F}=(0\subset F_1\subset F_2\subset\ldots\subset F_n=\K^N)$
with  $\dim F_j=\sum_{i=1}^j r_i$.

Also pick  $\ww=(d_1,\ldots,d_{n-1})$ with $\sum_{i=1}^{n-1}id_i=N$
and let $e\in \g$ be a nilpotent element whose Jordan type is $(1^{d_1},2^{d_2},\ldots,(n-1)^{d_{n-1}})$.
From these data we can construct the Slodowy
slice $\Sl\subset\g$, 
and the
Slodowy variety $\Sl(e,P)$.

Define $\vv:=(v_1,\ldots,v_{n-1})$ by   \begin{equation}\label{eq:vv_typeA}
v_i:=\sum_{j=i+1}^n r_i-\sum_{j=i+1}^{n-1}(j-i)d_i.\end{equation}
Below we assume that all $v_i$'s are positive. Finally, let $Q$ be the quiver
of type $A_{n-1}$ and a character $\theta$ of $\GL(\vv)$ be as in (\ref{eq:character}).

 Maffei proved in \cite{Maffei}, Theorem 8, that the algebraic varieties $\M_0^\theta(\D Q,\vv,\ww)$
and $\Sl(e,P)$ are isomorphic. His construction is pretty technical but we will need to recall it to establish some additional properties of his isomorphism, for example, that it is a symplectomorphism.

Let us proceed to recalling the construction of a required isomorphism.
First of all, there is a special case when an isomorphism is easy: when $e=0$
or, equivalently, $d_1=N, d_2=\ldots=d_{n-1}=0$, see
\cite{Nakajima}, Theorem 7.2, or \cite{Maffei}, Lemma 15.
In this case, $\Sl(e,P)$ is nothing else but the cotangent bundle $T^*(G/P)$.
A point in $T^*(G/P)$ can be thought as a pair $(x,\mathcal{F})$, where $\mathcal{F}$ is a partial flag  as above  and
$x\in \g$ is such that $x(F_i)\subset F_{i-1}$.
An isomorphism $\widetilde{\varphi}:\M^\theta_0(\D Q,\vv,\ww)\rightarrow T^*(G/P)$
is given by $$\GL(\vv).[(A_i),(B_i),\Gamma_1,\Delta_1]\mapsto (\Delta_1\Gamma_1, 0\subset \ker \Gamma_1\subset
\ker A_1\Gamma_1\subset\ldots\subset \ker A_{n-2}\ldots A_1\Gamma_1\subset \K^N).$$

Now proceed to the general case. Following \cite{Maffei}, we will first introduce some "transversal subvariety"
in $T^*(G/P)$.

For this we need some notation.
Set
\begin{align*}&\widetilde{d}_1:=N, \widetilde{d}_i:=0, i>0,\\
&\widetilde{v}_i:=v_i+\sum_{j=i+1}^{n-1}d_j, i=1,\ldots,n-1, \\ &\widetilde{\ww}:=(\widetilde{d}_1,\ldots,\widetilde{d}_{n-1}),
\widetilde{\vv}:=(\widetilde{v}_1,\ldots,\widetilde{v}_{n-1}).\end{align*}
Further, set \begin{equation}\label{eq:direct_sum}\widetilde{D}_1:=\sum_{1\leqslant k\leqslant j<n}D_i^{(k)},\, D_i':=\bigoplus_{1\leqslant k\leqslant j-i\leqslant n-i-1}D_j^{(k)},\,
\widetilde{V}_i:=V_i\oplus D_i',\end{equation} where $D_i^{(k)}$ means a copy of $D_i$.
For the notational convenience we write $\widetilde{V}_0=D_0':=\widetilde{D}_1$.

Let $(\widetilde{A}_i,\widetilde{B}_i,\widetilde{\Gamma}_1,\widetilde{\Delta}_1)_{i=1,\ldots,n-2}$ be an element
of $R(\D Q,\widetilde{\vv},\widetilde{\ww})$. Put $\widetilde{A}_0:=\widetilde{\Gamma}_1,\widetilde{B}_0:=\widetilde{\Delta}_1$. As in \cite{Maffei}
we will write the elements $\widetilde{A}_i,\widetilde{B}_i$ in the block form as follows:
\begin{equation}\label{eq:block_decomp}
\begin{split}
&\pi_{D_j^{(h)}}\widetilde{A}_i|_{D_{j'}^{(h')}}=\Tt_{i,j,h}^{j',h'},\quad \pi_{D_j^{(h)}}\widetilde{B}_i|_{D_{j'}^{(h')}}=\Ss_{i,j,h}^{j',h'},\\
&\pi_{D_j^{(h)}}\widetilde{A}|_{V_i}=\Tt_{i,j,h}^V,\quad \pi_{D_j^{(h)}}\widetilde{B}|_{V_i}=\Ss_{i,j,h}^V,\\
&\pi_{V_{i+1}}\widetilde{A}|_{D_{j'}^{(h')}}=\Tt_{i,V}^{j',h'},\quad \pi_{V_i}\widetilde{B}_i|_{D_{j'}^{(h')}}=\Ss_{i,V}^{j',h'},\\
&\pi_{V_{i+1}}\widetilde{A}_i|_{V_i}=\Aa_i,\quad \pi_{V_i}\widetilde{B}_i|_{V_{i+1}}=\Bb_i,
\end{split}
\end{equation}
where $\pi_\bullet$ stand for the projections to the summands in (\ref{eq:direct_sum}).

Embed $\GL(\vv)$ into $\GL(\widetilde{\vv})$ by making $\GL(\vv)$ act trivially on $D_i'$
and as before on $V_i$. Let $\widetilde{\theta}$ denote the character of $\GL(\widetilde{\vv})$
defined analogously to $\theta$. Clearly, $\theta$ coincides with the restriction of $\widetilde{\theta}$
to $\GL(\vv)$.

Further, choose $\sl_2$-triples $(e_i,[e_i,f_i],f_i)$ in $\gl(D_i'), i=0,\ldots,n-1$ as follows:
\begin{equation}\label{eq:sl2}
\begin{split}
&e_i|_{D_j^{(1)}}:=0, \quad e_i|_{D_j^{(h)}}:=\id_{D_j}:D_j^{(h)}\rightarrow D_j^{(h-1)},\\
&f_i|_{D_j^{j-i}}=0, \quad f_i|_{D_j^{(h)}}:=h(j-i-h)\id_{D_j}:D_j^{(h)}\rightarrow D_j^{(h+1)}.
\end{split}
\end{equation}

In particular, the nilpotent element $e_0\in \g$ corresponds to the partition $\ww$. Under the isomorphism
$T^*(G/P)\cong M(\D Q,\widetilde{\vv},\widetilde{\ww})$ described above the variety $\Sl(e,P)$ is identified with
$$\{(\widetilde{A}_i,\widetilde{B}_i)_{i=0}^{n-2}\in \Lambda(\D Q,\widetilde{\vv},\widetilde{\ww})^{ss,\widetilde{\theta}}: [\widetilde{B}_0\widetilde{A}_0-e_0,f_0]=0\}/\GL(\widetilde{\vv}).$$

Now we are ready to define {\it transversal} elements. For this we need to assign degrees (denoted by $\grad$) to the blocks $\Tt^{\bullet}_\bullet, \Ss^\bullet_\bullet$ as follows.
\begin{equation}\label{def_gr_deg}
\begin{split}
&\grad(\Tt_{i,j,h}^{j',h'}):=\min(h-h'+1, h-h'+1+j'-j),\\
&\grad(\Ss_{i,j,h}^{j',h'}):=\min(h-h', h-h'+j'-j),\\
\end{split}
\end{equation}

An element $((\widetilde{A}_i),(\widetilde{B}_i))_{i=0}^{n-2}\in \Lambda_0(\D Q,\widetilde{\vv},\widetilde{\ww})$ is said to be {\it transversal} if it satisfies the following relations for $i=0,1,\ldots,n-2$:
\begin{equation}\label{eq:transv_vanish}
\begin{array}{ll}
\Tt_{i,j,h}^{j',h'}=0 & \text{ if } \grad(\Tt_{i,j,h}^{j',h'})<0,\\
&\text{ or if }\grad(\Tt_{i,j,h}^{j',h'})=0\text{ and } (j',h')\neq (j,h+1),\\
\Tt_{i,j,h}^{j',h'}=\id_{D_j}&\text{ if }\grad(\Tt_{i,j,h}^{j',h'})=0\text{ and } (j',h')=(j,h+1),\\
\Tt_{i,j,h}^V=0,&\\
\Tt_{i,V}^{j',h'}=0&\text{ if }h'\neq 1,\\
\Ss_{i,j,h}^{j',h'}=0 &\text{ if }\grad(\Ss_{i,j,h}^{j',h'})<0,\\
&\text{ or if }\grad(\Ss_{i,j,h}^{j',h'})=0\text{ and } (j',h')\neq (j,h),\\
\Ss_{i,j,h}^{j',h'}=\id_{D_j}&\text{ if }\grad(\Tt_{i,j,h}^{j',h'})=0\text{ and } (j',h')=(j,h),\\
\Ss_{i,j,h}^V=0&\text{ if }h\neq j-i,\\
\Ss_{i,V}^{j',h'}=0.&\\
\end{array}
\end{equation}
\begin{equation}\label{eq:transv_rel}
[\pi_{D_i'}\widetilde{B}_i\widetilde{A}_i|_{D_i'}-e_i,f_i]=0.
\end{equation}

The subvariety of $\Lambda(\widetilde{D},\widetilde{V})$ consisting of the transversal elements
will be denoted by $\mathfrak{T}$.

To establish an isomorphism $\M_0^\theta(\D Q,\vv,\ww)\xrightarrow{\sim} \Sl(e,P)$
Maffei constructs a morphism $\Lambda_0(\D Q,\vv,\ww)\rightarrow \mathfrak{T}$.


The main technical step in Maffei's construction is the following lemma that is a union
of Lemmas 17-19 from \cite{Maffei}.

\begin{Lem}\label{Lem:Maffei_main}
Let $x=((A_i),(B_i),(\Gamma_j),(\Delta_j))\in \Lambda_0(\D Q,\vv,\ww)$. Then
the following claims hold:
\begin{enumerate}
\item
there is a unique
element $\widetilde{x}=((\widetilde{A}_i),(\widetilde{B}_i))\in \mathfrak{T}$ satisfying the following equalities:
\begin{equation}\label{eq:Maffei_eq1}
\begin{split}
&\mathbb{A}_i=A_i,\\
&\mathbb{B}_i=B_i,\\
&\Tt^{i+1,1}_{i,V}=\Gamma_{i+1},\\
&\Ss^V_{i,i+1,1}=\Delta_{i+1}.
\end{split}
\end{equation}
for all $i=0,\ldots,n-2$ (where we set $A_0:=\Gamma_1,B_0:=\Delta_1$).
\item The map $\Phi:\Lambda_0(\D Q,\vv,\ww)\rightarrow \mathfrak{T}, x\mapsto \widetilde{x},$ is a $\GL(\vv)$-equivariant isomorphism.
\item $\Phi(\Lambda_0(\D Q,\vv,\ww)^{\theta,ss})=\mathfrak{T}\cap R(\D\widetilde{Q},\widetilde{\vv},\widetilde{\ww})^{\widetilde{\theta},ss}$.
\end{enumerate}
\end{Lem}

So we can define the morphism $\varphi: \M^\theta_0(\D Q, \vv,\ww)\rightarrow \M^{\widetilde{\theta}}_0(\D Q, \widetilde{\vv},\widetilde{\ww})=T^*(G/P)$ sending an
orbit $\GL(\vv)x$ to $\GL(\widetilde{\vv})\Phi(x)$.

The following proposition is the main result of \cite{Maffei}.

\begin{Prop}\label{Prop:Maffei_main}
The morphism $\varphi$ is an isomorphism of $\M^\theta_0(\D Q, \vv,\ww)$
onto $\Sl(e,P)$.
\end{Prop}

Now we will use the Maffei construction to establish some properties of the morphism $\varphi$: namely,
that this morphism is $\K^\times$-equivariant, is a symplectomorphism and is compatible, in an appropriate
sense, with natural line bundles. This is done in the next three lemmas.

\begin{Lem}\label{Lem:Iso_A_equivariance}
The morphism $\varphi$ is $\K^\times$-equivariant.
\end{Lem}
\begin{proof}
Let us define a certain $\K^\times$-action on $R(\D Q, \widetilde{\vv},\widetilde{\ww})$.
For this consider the element $\mathbf{h}=([e_0,f_0],[e_1,f_1],\ldots, [e_{n-2},f_{n-2}])$.
The element $[e_i,f_i]$ acts by $j-i+1-2h$ on $D_j^{(h)}\subset D'_i$. Let $\gamma: \K^\times\rightarrow
\GL(\widetilde{\vv})$ denote the one-parameter subgroup corresponding to $\mathbf{h}$.
Consider the action of $\K^\times$ on $R(\D Q, \widetilde{\vv},\widetilde{\ww})$ given
by $t. \widetilde{x}=t^{-1} \gamma(t) x$. The following claims are checked directly:
\begin{itemize}
\item[(i)] The $\K^\times$-action preserves the affine subspace given by (\ref{eq:transv_vanish})
and also the subvariety of solutions of (\ref{eq:transv_rel}). So $\mathfrak{T}$ is
$\K^\times$-stable.
\item[(ii)] The induced $\K^\times$-action on $T^*(G/P)=\M^{\widetilde{\theta}}_0(\D Q,\widetilde{\vv},\widetilde{\ww})$ is the Kazhdan action.
\item[(iii)] The blocks $\mathbb{A}_i,\mathbb{B}_i, \mathbb{T}_{i,V}^{i+1,1},\mathbb{S}_{i,i+1,1}^V$
are multiplied by $t^{-1}$.
\end{itemize}
(i),(iii) and assertion (1) of Lemma \ref{Lem:Maffei_main} imply that the morphism
$\Phi: \Lambda_0(\D Q,\vv,\ww)\rightarrow \mathfrak{T}$ is $\K^\times$-invariant.
Now (ii) and the construction of $\varphi$ complete the proof of the present lemma.
\end{proof}

\begin{Lem}\label{Lem:Iso_A_symplect}
The isomorphism $\varphi$ is a symplectomorphism.
\end{Lem}
\begin{proof}
Let $\omega,\widetilde{\omega}$ denote the symplectic forms on the spaces $R(\D Q, \vv,\ww)$
and $R(\D Q, \widetilde{\vv},\widetilde{\ww})$. Explicitly, for $v^\alpha=((A_i^\alpha),(B_i^\alpha),(\Gamma_j^\alpha),(\Delta_j^\alpha))\in R(\D Q, \vv,\ww), \alpha=1,2$ we have  $\omega(v^1,v^2)=\beta(v^1,v^2)-\beta(v^2,v^1)$, where
$\beta(v^1,v^2)=\sum_{i=1}^{n-2}\tr(B_i^{1}A_i^{2})+\sum_{i=1}^{n-1}\tr(\Delta_i^{1}\Gamma_i^2)$. Analogously, $\widetilde{\omega}(\widetilde{v}^1,\widetilde{v}^2)=\widetilde{\beta}(\widetilde{v}^1,\widetilde{v}^2)-
\widetilde{\beta}(\widetilde{v}^2,\widetilde{v}^1)$, where $\widetilde{\beta}$ is defined similarly to
$\beta$.

Let us show that \begin{equation}\label{eq:form_pullback}\Phi^*(\widetilde{\beta}|_{\mathfrak{T}})=\beta|_{\Lambda_0(\D Q, \vv,\ww)}.\end{equation} First of all, pick $x\in \Lambda_0(\D Q, \vv,\ww)$
and  $v\in T_x \Lambda_0(\D Q,\vv,\ww)$. Write the element $d_x\Phi(v)$ in the block form $(\Tt_{\bullet}^\bullet, \Ss_\bullet^\bullet, \Aa_\bullet,\Bb_\bullet)$ as above.
Then $\Tt_{i,j,h}^{j',h'}=0$ if $h<h'$ and $\Ss_{i,j,h}^{j',h'}=0$ if $h\leqslant h'$. Moreover,
$\Tt_{i,j,h}^V=0, S_{i,V}^{j',h'}=0$ for all $i,j,h,j',h'$, and $\Tt_{i,V}^{j',h'}=0$ if $h'\neq 1$,
$\Ss_{i,j,h}^V=0$ if $h\neq j-i$. It follows that
for $v^1,v^2\in T_x\Lambda_0(\D Q,\vv,\ww)$ we have
$$\widetilde{\beta}(d_x\Phi(v^1), d_x \Phi(v^2))=\sum_{i=1}^{n-2} \tr(\Bb^1_i \Aa^2_i)+\sum_{i=0}^{n-2}\tr((\Ss^{V}_{i,i+1,1})^1(\Tt_{i,V}^{i+1,1})^2).$$
From assertion (1) of Lemma \ref{Lem:Maffei_main} we deduce that $\widetilde{\beta}(d_x\Phi(v^1),d_x\Phi(v^2))=\beta(v^1,v^2)$. This is
equivalent to (\ref{eq:form_pullback}).

It follows that
\begin{equation}\label{eq:form_pullback2}\Phi^*(\widetilde{\omega}|_{\mathfrak{T}})=\omega|_{\Lambda_0(\D Q, \vv,\ww)}\end{equation}
In particular, we see that for $x\in \mathfrak{T}\cap R(\D Q, \widetilde{\vv},\widetilde{\ww})^{\widetilde{\theta},ss}$
the kernel of the restriction of $\widetilde{\omega}$ to $T_x \mathfrak{T}$ coincides with $T_x \GL(\vv)x$.
So the pull-back of the symplectic form from $\Sl(e,P)= (\mathfrak{T}\cap R(\D Q, \widetilde{\vv},\widetilde{\ww})^{\widetilde{\theta},ss})/\GL(\vv)$ to $\mathfrak{T}\cap R(\D Q, \widetilde{\vv},\widetilde{\ww})^{\widetilde{\theta},ss}$ coincides with the restriction of
$\widetilde{\omega}$. Using the definition of the symplectic form on a reduction, we see
that $\varphi$ is a symplectomorphism.
\end{proof}

Below we will need to understand the behavior of some natural line bundles under the isomorphism $\varphi$.
Let $L_i, i=1,\ldots,n-1,$ be the 1-dimensional $\GL(\vv)$-module, where $\GL(\vv)$ acts by
$(X_1,\ldots, X_{n-1})\mapsto \det(X_i)$. Let $\mathcal{L}_i$ denote the corresponding line bundle on
$\M^\theta_0(\D Q, \vv,\ww)$.

Now let us define certain line bundles on $\Sl(e,P)$. Let $\mathcal{F}=(0=F_0\subset F_1\subset F_2\subset \ldots \subset F_{n-1}\subset F_n= \K^N)$ be the flag stabilized by $P$. Consider the $P$-modules $\widetilde{L}_i:= \bigwedge^{top}F_{i}$. Let $\widetilde{\mathcal{L}}_i$ denote the corresponding bundles on $T^*(G/P),\Sl(e,P)$.

\begin{Lem}\label{Lem:A_line_bundles}
$\varphi^*(\widetilde{\mathcal{L}}_i)\cong \mathcal{L}_i$.
\end{Lem}
\begin{proof}
From the construction of the isomorphism $\M^{\widetilde{\theta}}_0(\D Q, \widetilde{\vv},\widetilde{\ww})\cong T^*(G/P)$ produced above, we see that one can interpret the line bundles $\widetilde{\mathcal{L}}_i$ in a different
way. Namely, $\widetilde{\mathcal{L}}_i$ coincides with the line bundle on $R(\D Q, \widetilde{\vv},\widetilde{\ww})\red^{\widetilde{\chi}}_0 \GL(\widetilde{\vv})$ induced by the
1-dimensional $\GL(\widetilde{\vv})$-module $\widetilde{L}_i$ defined analogously to $L_i$.

Now the isomorphism of the lemma  follows from the fact that the restriction of the character $(\widetilde{X}_1,\ldots, \widetilde{X}_{n-1})\mapsto \det(\widetilde{X}_i)$ of $\GL(\widetilde{\vv})$ to $\GL(\vv)$ coincides
with the character $(X_1,\ldots,X_{n-1})\mapsto \det(X_i)$.
\end{proof}

\section{W-algebras}\label{SECTION_Walg}
\subsection{Definitions}\label{SUBSECTION_Walg_def}
Let $G$ be a reductive algebraic group, $\g$ be the Lie algebra of $G$. Pick a nilpotent element
$e\in \g$ and choose $f,[e,f]$ forming an $\sl_2$-triple with $e$. Recall
the Slodowy slice $\Sl$ and the equivariant Slodowy slice $\XS=G\times \Sl$.

A (finite) W-algebra is a quantization of the graded Poisson algebra $\K[S]$.
In full generality, it was first defined by Premet in \cite{Premet1}.
In this subsection we will recall the definitions of a W-algebra following  \cite{Wquant}
and \cite{GG}. For details the reader is referred to the review \cite{ICM}.
We remark, however, that here we will need homogenized versions of W-algebras,
i.e., our algebras will be graded algebras over $\K[h]$.

The variety $\XS$ is affine and hence admissible in the sense of Subsection \ref{SUBSECTION_non_comm_period}.
So we can consider the canonical quantization  $\widetilde{\Walg}_h$ of $\XS$. Consider the
algebra $\Gamma(X,\widetilde{\Walg}_h)^G$. This algebra is complete in the $h$-adic topology,
and $\Gamma(X,\widetilde{\Walg}_h)^G/h \Gamma(X,\widetilde{\Walg}_h)^G=\K[\XS]$. Let
$\Walg_h$ denote the subalgebra of $\K^\times$-finite vectors in $\Gamma(X,\widetilde{\Walg}_h)^G$.
Since the $\K^\times$-action on $\Sl$ is contracting, we see that $\Walg_h/h \Walg_h= \K[S]$.

We remark that the quantization $\widetilde{\Walg}_h$ of $\XS$ admits a quantum comoment map
$\g\rightarrow \Gamma(\XS,\widetilde{\Walg}_h)$, see \cite{HC}. This  gives rise to a $G\times \K^\times$-equivariant algebra homomorphism $U_h(\g)\rightarrow \Gamma(\XS,\widetilde{\Walg}_h)$. Restricting the latter to the $G$-invariants  we get a monomorphism $U_h(\g)^G\hookrightarrow \Walg_h$. Since the $G$-action on $X$
is free, it is easy to see that $U_h(\g)^G$ coincides with the center of $\Walg_h$. An alternative proof is given in
\cite{HC}, Subsection 2.2.

In the sequel we will need an extension of $\Walg_h$. Namely, pick a Cartan subalgebra $\h\subset\g$ and let $W$
denote the Weyl group of $(\g,\h)$ and $\Delta\subset \h^*$ be the root system. Pick a Borel subgroup $B\subset G$. This choice defines a system $\Pi$ of simple roots in $\Delta$. Let, as usual, $\rho$ stand for  half the sum of all positive roots. Then one defines the $\cdot$-action of $W$ on $\h^*$ by $w\cdot \lambda= w(\lambda+\rho)-\rho$. Consider the induced action of $W$ on $S\h=\K[\h^*]$. Recall the Harish-Chandra isomorphism $U(\g)^G\cong S\h^W$.
We will use its homogenized version: we identify $U_h(\g)^G$ with $\K[\h^*,h]^W$, where the action of
$W$ on the latter algebra  is given by $w.f(\lambda)=f(w^{-1}(\lambda+\rho h)-\rho h)$.

Below we will need to consider the algebra $\Walg_{h,\h}:=\Walg_h\otimes_{U_h(\g)^G}\K[\h^*,h]$.
For $\lambda\in \h^*$ let $\Walg_\lambda$ denote the quotient of $\Walg_\h$ by the ideal in $\K[\h^*]$ corresponding to $\lambda$ and $h=1$.
It is easy to see that the natural homomorphism $\Walg_h\rightarrow \Walg_\lambda$ is an epimorphism.

Now let us explain the approach to W-algebras of Premet, \cite{Premet1}, in the version of Gan and Ginzburg,
\cite{GG}. Recall  the grading $\g=\bigoplus_{i\in \Z}\g(i)$, the subalgebra $\m\subset \g$,
the element $\chi\in \g^*$ from Subsection \ref{SUBSECTION_Slodowy}.

Consider
the quantum Hamiltonian reduction $U_h(\g)\red_\chi M$ equipped with the so
called {\it Kazhdan grading}. The latter is defined as follows: for $\xi\in \g(i)$
the Kazhdan degree of $\xi$ is, by definition, $i+2$ (and the degree of $h$ is 2, as usual).
The algebra
$U_h(\g)\red_\chi M$ inherits the grading from $U_h(\g)$.

In \cite{Wquant}, the author checked that the filtered algebra $\Walg_h/(h-1)\Walg_h$ and
$U(\g)\red_\chi M$ are isomorphic. This does not automatically imply that the graded algebras
$\Walg_h$ and $U_h(\g)\red_\chi M$ are isomorphic. However, the existence of the latter isomorphism
can be easily deduced from \cite{Wquant}, Remark 3.1.4.

%
%
%

\subsection{Parabolic W-algebras and quantizations of Slodowy varieties}\label{W_parab}
Let us define the parabolic analogs of $\Walg_{h,\h}$.
Namely, let $P$ be a parabolic subgroup of $G$ and let $\a$ stand for the quotient
of $\p$ by its solvable radical. Consider a $\K[\a^*,h]$-algebra $\A_h $ that is, by definition, the algebra
of $\K^\times$-finite elements in $\Gamma(G/P,\Dcal_h(G/P_0)^{P/P_0})$, where
the action of $P/P_0$ on $\Dcal_h(G/P_0)$ is induced from the action on
$G/P_0$ by right translations. This algebra comes equipped with a
homomorphism $U_h(\g)\rightarrow \A_h$ induced by the quantum comoment
map. So we can consider the quantum Hamiltonian reduction $\Walg^P_{h,\a}:=(\A_h/\A_h\m_\chi)^{M}$.
This is also a graded algebra over $\K[\a^*,h]$.

We will modify the $\K[\a^*,h]$-algebra structure on $\Walg^P_{\a,h}$ as follows. Assume that
$B\subset P$ and let $L$ stand for the Levi  subalgebra of $P$ containing the maximal
torus corresponding to $\h$. We may and will identify $\a$ with the center $\z(\mathfrak{l})$
of $\mathfrak{l}=\operatorname{Lie}(L)$. Consider the map
$\iota:\a\rightarrow \A_h$ defined by $\iota(\xi)=\Phi(\xi)-h \langle\rho,\xi\rangle$,
where $\Phi$ is the initial map $\a\rightarrow \A_h$.
The new $\K[\a^*,h]$-algebra structure on $\Walg^P_{h,\a}$ we need is induced by
$\iota$. The reason why we need this shift will become clear later.


We remark that the identification of $\a$ with the center of $\mathfrak{l}$ gives
rise to the direct sum decomposition $\h=\a\oplus (\h\cap [\mathfrak{l},\mathfrak{l}])$
and hence to the projection $\h\twoheadrightarrow \a$. So we can set
$\Walg_{\a,h}:=\K[\a^*,h]\otimes_{\K[\h^*,h]}\Walg_{h,\h}$.

\begin{Lem}\label{Lem:W_parab_qout}
\begin{enumerate}
\item $\Walg_{h,\a}^P/ h\Walg_{h,\a}^P=\K[\widetilde{\Sl}(e,P)]$.
\item $\Walg_{h,\a}^P$ coincides with the subalgebra of $\K^\times$-finite elements
in $\Gamma(\widetilde{\Sl}(e,P), \widetilde{\Walg}_h\widehat{\red} P)$.
\item There is a natural graded $\K[h]$-algebra homomorphism  $\Walg_{\a,h}\rightarrow \Walg^P_{h,\a}$.
\item This homomorphism is bijective when $P=B$.
\item This homomorphism  is surjective provided $\g$ is of type $A$.
\end{enumerate}
\end{Lem}
\begin{proof}
The $M$-action on $\mu^{-1}(\chi|_{\m})$ is free, \cite{Ginzburg_HC}, Corollary 1.3.8 (here $\mu$ is the moment
map $\Spec(\A_h/(h))\rightarrow \m^*$).
So the algebra $\Walg_{h,\a}^P$ satisfies quantization commutes with reduction condition by Lemma \ref{Lem:quant_comm_red}. But the algebra of global functions on $\XS\red P=(T^*G\red P)\red_\chi M$, by definition, is just $(\A_h/(h))\red_\chi M$. This implies (1). The proof of (2) is now analogous to that of Lemma
\ref{Lem:global_sections}.

Let us prove (3). We need to establish a homomorphism $\K[\a^*,h]\otimes_{U_h(\g)^G}U_h(\g)\twoheadrightarrow \A_h$, then we will apply the reduction by $M$.

We have  the quantum comoment map homomorphism $U_h(\g)\rightarrow \A_h$ together with a homomorphism
$\K[\a^*,h]\rightarrow \A_h$ specified above. Let us show that these two homomorphisms agree on $U_h(\g)^G$
(the latter maps to $\K[\a^*,h]$ via the composition $U_h(\g)^G\cong \K[\h^*,h]^W\hookrightarrow \K[\h^*,h]\twoheadrightarrow
\K[\a^*,h]$). This is a pretty standard fact but we will provide its proof for readers convenience.
First of all, since both homomorphisms $U_h(\g)^G\rightarrow \A_h$ are graded it is enough to prove
that they coincide modulo $h-1$. But the algebra $\A:=\D(G/P_0)^{P/P_0}$ acts on $\K[G/P_0]$
and the subalgebra $S\a\subset \A$ acts faithfully. Now the claim that the natural action of
$U(\g)^G$ on $\K[G/P_0]$ factors through the homomorphism $U(\g)^G\rightarrow S\a$ is just
part of the construction of the Harish-Chandra isomorphism $U(\g)^G\cong S\h^W$.
This completes the proof of (3).

So we have constructed a homomorphism $\K[\a^*,h]\otimes_{U_h(\g)^G}U_h(\g)\rightarrow \A_h$. For
$P=B$ the natural homomorphism $\K[\h^*,h]\otimes_{U_h(\g)^G}U_h(\g)\rightarrow \A_h$ is a bijection.
Indeed, since the right hand side is $\K[h]$-flat, it is enough to show that this homomorphism
is an isomorphism modulo $h$. This  follows from the fact that $\A_h/(h)=\K[T^*G\red B]=\K[\h^*\times_{\g^*\quo G}\g^*]=\K[\h^*,h]\otimes_{U_h(\g)^G}U_h(\g)/(h)$.

Proceed to assertion (5). Again, it is again to prove that the homomorphism is surjective modulo $h$.
The homomorphism becomes $\K[\a^*\times_{\g^*\quo G}\Sl]\rightarrow \K[\XS\red P]$. Again, the homomorphism
is the identity on $\K[\a^*]$ and both algebras are flat (=graded free) over $\K[\a^*]$. So its enough to prove that
the homomorphism is surjective modulo $(\a)$. But modulo $(\a)$ the left algebra is just $\K[\Sl\cap\mathcal{N}]$,
where $\mathcal{N}$ denotes the nilpotent cone in $\g^*$. The right algebra is $\K[\Sl(e,P)]$. The image of
$\Sl(e,P)$ in $S\cap \mathcal{N}$ coincides with $S\cap G\p^{\perp}$.

Let us show that the latter is a normal
Poisson variety. First of all,  $G\p^\perp$ is the closure of a nilpotent orbit and hence is normal, thanks to the results of Kraft and Procesi, \cite{KP}.
The intersection $\Sl\cap G\p^\perp$ is transversal at $\chi$ so $\Sl\cap G\p^\perp$
is normal at $\chi$. To prove the normality at the other points  we notice that
the Kazhdan action contracts $\Sl\cap G\p^\perp$ to $\chi$.
Also the morphism $\Sl(e,P)\rightarrow \Sl\cap G\p^\perp$ is birational because the natural morphism
$T^*(G/P)\rightarrow G\p^\perp$ is birational and $G$-equivariant. So we see that the morphism
$\Sl(e,P)\rightarrow \Sl\cap \mathcal{N}$ gives rise to an isomorphism $\K[\Sl(e,P)]\cong \K[\Sl\cap G\p^\perp]$.
So the morphism $\K[\Sl\cap \Nil]\rightarrow \K[\Sl(e,P)]$ is surjective, as required.
\end{proof}
\begin{Rem}
In general, the homomorphism $\Walg_{h,\a}\rightarrow \Walg_{\a,h}^P$ is not surjective but is surjective
modulo $(h-1)$. Let us sketch a proof. First, we need to show that the homomorphism $S\a\otimes_{U(\g)^G}U(\g)\rightarrow \D(G)\red P$
is surjective. This homomorphism is the identity on $S\a$ so it is enough to show that the induced homomorphism
of the fibers at $\lambda\in \a^*$ is surjective for any $\lambda\in \a^*$. But this follows from results of Borho
and Brylinski, \cite{BB}, Theorem 3.8 and Remark 3.9.
\end{Rem}

\subsection{Main theorems}\label{SUBSECTION_W_main}
First of all let us state a result on the isomorphism of deformations of Kleinian singularities.

Let  $\Gamma, Q, \vv,\ww,\g,\Orb,\delta$ be as in Subsection \ref{SUBSECTION_Klein_classical}.
Pick a Cartan subalgebra $\h\subset \g$, a system $\alpha_1,\ldots,\alpha_r$ of simple roots in $\h$.
Recall that the space $\z=\gl(\vv)^{*\GL(\vv)}$ is identified with $\K^{Q_0}$ by
$\sum_{i\in Q_0}\chi_i\epsilon_i\mapsto [\xi\mapsto\sum_{i\in Q_0}\chi_i \tr(\xi_i)]$.
Here $\epsilon_i, i\in Q_0,$ is the tautological basis in $\K^{Q_0}$.
Consider the subspace $\z_0\subset \z$ of all vectors orthogonal to $\delta$.
Identify $\h^*$ with $\z_0$ by $\lambda\mapsto \sum_{i=0}^r  \lambda_i\epsilon_i$,
where $\lambda_1,\ldots,\lambda_r$ are defined from $\lambda=\sum_{i=1}^r \lambda_i\alpha_i$,
and $\lambda_0=-\sum_{i=1}^r \delta_i \epsilon_i$.

Recall that we have the $\K[\h^*][h]$-algebra $\Walg_{h,\h}$ (with the modified structure map $\K[\h^*][h]\rightarrow
\Walg_{\h,h}$, see Subsection \ref{W_parab}) and also a $\K[\z_0][h]$-algebra
$\W_{\z_0}(\D Q,\vv,\ww)_h$.

\begin{Thm}\label{Thm:Klein_W}
There is a $\K[\h^*][h]$-linear isomorphism $\W_{\z_0}(\D Q,\vv,\ww)_h\xrightarrow{\sim} \Walg_{h,\h}$ of graded associative algebras.
\end{Thm}

\begin{Rem}\label{Rem:Klein}
Usually one considers the $\K[\z_0][h]$-algebra $\W_{\z_0}(\D Q,\vv)_h$ instead of $\W_{\z_0}(\D Q,\vv,\ww)_h$.
However, it is pretty straightforward to see that these two algebras are naturally isomorphic.
\end{Rem}

Now let $G=\SL_N$, $n,r_1,\ldots,r_n,P, \vv,\ww,Q,\Orb$ have the same meaning as in Subsection \ref{SUBSECTION_A_quiver_vs_Slod}. Let $\a$ be constructed from $P$ as in Subsection \ref{W_parab}.
Let us relate the algebras $\Walg_{\a,h}^P$ and $\W(\D Q,\vv,\ww)_h$. First, we need to  identify $\a$
with $\z^*$. We can view $\a$ as the space $\{\diag(x_1,\ldots,x_1,\ldots,x_{n},\ldots,x_{n})\}$
of matrices, where $x_i$ appears $r_i$ times with $\sum_{i=1}^n r_i x_i=0$. Map $\diag(x_1,\ldots, x_n)\in \a$
to $\sum_{i=1}^{n-1}(\sum_{j=1}^i r_j x_j)\epsilon_i$.

We assume that $d_i\geqslant 2 v_i-v_{i+1}-v_{i-1}$ for all $i$ (we set $v_0=v_n=0$).
This is equivalent to $r_i\geqslant r_{i+1}$ for all $i$. Also recall that in Subsection \ref{SUBSECTION_quiver_var}
we assigned elements $\mathrm{d}$ to $\ww$ and $\mathrm{v}$ to $\vv$. Our condition
is equivalent to saying that $\mathrm{d}-\mathrm{v}$ is dominant.

The reason for this assumption
is that it guarantees that quantization commutes with reduction holds for $\W(\D Q,\vv,\ww)_h$.
If we remove this assumption then one can still show that we have an epimorphism
$\W(\D Q,\vv,\ww)_h\twoheadrightarrow \Walg_{h,\a}^P$ that is an isomorphism if and only
if the scheme $\M_0(\D Q,\vv,\ww)$ is reduced.

\begin{Thm}\label{Thm:type_A}
There is a $\K[\a^*][h]$-linear isomorphism $\W(\D Q,\vv,\ww)_h\xrightarrow{\sim}
\Walg_{h,\a}^P$ of graded associative algebras.
\end{Thm}

\begin{Rem}\label{Rem:Thm_A} The previous theorem provides the third realization of
the W-algebra in type $A$. The first one was that of Premet, while the second one was
the Yangian type presentation of $\Walg$ by generators and relations due to  Brundan-Kleshchev,\cite{BrK}.
At the moment, an analogous presentation of the general parabolic W-algebras in type $A$
is not known, but should not be difficult to obtain. We also remark that Mirkovic and
Vybornov related parabolic Slodowy slice and Slodowy varieties to affine Grassmanians, \cite{MV}.
One could speculate that their results is a classical analog of the Brundan-Kleshchev
presentation.
\end{Rem}

\subsection{Reduction of even quantizations}\label{SUBSECTION_even_red}
Let $\X,G$ be as in Subsection \ref{SUBSECTION_class_red} and let $\Dcal$ be an even graded $G$-equivariant quantization of $\X$ with parity antiautomorphism $\sigma$.
Also we suppose that $\K^\times$ acts on $\X$ as in Subsection \ref{SUBSECTION_class_red}.
Let $\Phi:\g\rightarrow \Gamma(\X,\Dcal)$ be the quantum comoment map, we assume that $\Phi(t^2\xi)=t.\Phi(\xi)$.
The goal of this subsection is  to obtain a criterium for $\Dcal\red G$ to be even. Of course,
this depends on the choice of $\Phi$: for two quantum comoment maps $\Phi,\Phi'$ we have $\Phi(\xi)-\Phi'(\xi)=
\langle\alpha,\xi\rangle h$ for some $\alpha\in \g^{*G}$.

Let $\alpha$ be the character of the $G$-action on $\bigwedge^{top}\g$. The main result
of this subsection is the following theorem.

\begin{Thm}\label{Thm:even_quant}
Suppose that $\sigma(\Phi(\xi))-\Phi(\xi)=\langle\alpha,\xi\rangle h$. Then $\Dcal\red G$ is an even
quantization.
\end{Thm}

For example, if the group $G$ has no characters, e.g. is semisimple or unipotent, then the reduction of
an even quantization is always even. Also if $G$ is reductive and the quantum comoment map is {\it symmetrized},
i.e., $\sigma(\Phi(\xi))=\Phi(\xi)$, then $\Dcal\red G$ is an even quantization of $X\red G$, and $\Dcal\red_0 G$
is an even quantization of $\X\red_0 G$. Now if we replace a reductive group $G$ with its parabolic
subgroup $P$, then the quantization $\X\red_{\rho h}P$ is even. This motivates the shift
we made in Subsection \ref{W_parab} (see also the proofs of Theorems \ref{Thm:Klein_W},\ref{Thm:type_A}
in the next subsection).

The proof of Theorem \ref{Thm:even_quant} is organized as follows. First, we prove two auxiliary lemmas.
Then we prove Theorem \ref{Thm:even_quant} in the case when $G$ is reductive, which is the most technical
part of the proof. Next, we deal with the cases of the 1-dimensional unipotent group. And then we complete
the proof of Theorem \ref{Thm:even_quant}.

First of all, let us investigate a relationship between $\Phi$ and the antiautomorphism $\sigma$ participating
in the definition of an even quantization.

\begin{Lem}\label{Lem:even_ham}
$\sigma(\Phi(\xi))-\Phi(\xi)\in \K h$ for all $\xi\in\g$. Moreover, if $\xi\in\g$ vanishes
on $\g^{*G}$, then $\sigma(\Phi(\xi))=\Phi(\xi)$.
\end{Lem}
\begin{proof}
For any local section $f$ or $\Dcal$ we have $$[\sigma(\Phi(\xi)),\sigma(f)]=
-\sigma([\Phi(\xi),f])=-\sigma(h\xi_{\Dcal}f)=h\xi_{\Dcal} \sigma(f)=[\Phi(\xi),\sigma(f)],$$
where $\xi_\Dcal$ stands for the derivation of $\Dcal$ induced by $\xi$.
So $\Phi(\xi)-\sigma(\Phi(\xi))$ lies in the center of $\Dcal$. Since $\X$ is symplectic,
the center coincides with $\K[[h]]$. Since both $\Phi(\xi), \sigma(\Phi(\xi))$ have degree 2 with respect to the $\K^\times$-action,
we see that $\sigma(\Phi(\xi))-\Phi(\xi)\in \K h$.

 The $G$-action on $\K[[h]]$ is trivial. So
$\xi\mapsto \Phi(\xi)-\sigma(\Phi(\xi))$ is both $G$-invariant and $G$-equivariant map.
Hence  for any $\xi$ in the annihilator of $\g^{*G}$ we have $\sigma(\Phi(\xi))=\Phi(\xi)$.
\end{proof}

Recall that $\pi:\X\rightarrow \X/G$ stands for the quotient morphism. By our assumptions,
this morphism is locally trivial in etale topology.

\begin{Lem}\label{Lem:surj_proj}
\begin{enumerate}
\item
The natural morphism $\pi_*(\Dcal)^G\rightarrow \Dcal\red G$ induces an isomorphism
$\pi_*(\Dcal)^G/\pi_*(\Dcal\Phi(\g_0))^G\xrightarrow{\sim} \Dcal\red G$ (were $\g_0$ is the annihilator
of $\g^{*G}$ in $\g$).
\item $\pi_*(\Dcal)$ is a flat (left or right) module over $\pi_*(\Dcal)^G$.
\end{enumerate}
\end{Lem}
\begin{proof}
Let us prove the first assertion.
It is enough to prove that the homomorphism $\pi_*(\Dcal)^G\rightarrow \Dcal\red G$ is surjective.
Similarly to the proof of Proposition 3.4.1 in \cite{Miura}, we can show that $\pi_*(\Dcal)^G/(h)=\pi_*(\Str_X)^G$.
So to prove that  $\pi_*(\Dcal)^G\rightarrow \Dcal\red G$ is surjective it is enough
to verify that the natural morphism $\pi_*(\Str_\X)^G\rightarrow \Str_{\X\red G}$ is surjective.
The surjectivity property is preserved by an etale base change. So we may assume that
$\X\rightarrow \X/G$ is a free principal $G$-bundle, so that $\X=(\X/G)\times G$.
It follows that any $G$-stable ideal in $\Str_{\X}$ is the pullback of an ideal
in $\Str_{\X/G}$. The surjectivity claim follows.

The second assertion follows from the fact that the morphism $\X\rightarrow \X/G$ is flat.
\end{proof}

Let us proceed to proving Theorem \ref{Thm:even_quant} in the case when $G$ is reductive.

\begin{Prop}\label{Prop:even_reduction}
If $G$ is reductive and $\sigma(\Phi(\xi))=\Phi(\xi)$, then $\Dcal\red G$ is even.
\end{Prop}
\begin{proof}
The proof is divided into several steps.

{\it Step 1.}
Set $\I:=\Dcal \Phi(\g_0)$. The antiautomorphism $\sigma$ descends to $\pi_*(\Dcal)^G$.
 Assume for a moment that the following claim holds
\begin{itemize}
\item[(*)]
$\pi_{*}(\I)^G$
is $\sigma$ stable.
\end{itemize}

Then $\sigma$ induces an antiautomorphism of $\Dcal\red G=\pi_{*}(\Dcal)^G/\pi_{*}(\I)^G$
and this antiautomorphism is a parity antiautomorphism.

We remark that (*) is local. So we may assume that $\X/G$ is affine. It follows that $\X$ is affine.
Then $\Gamma(\X,\I)=\Gamma(\X,\Dcal)\Span(\Phi(\xi): \xi\in\g_0)$. Abusing the notation,
we will write $\Dcal,\I$ for $\Gamma(\X,\Dcal),\Gamma(\X,\I)$. Set  $\J:=\I^G$. Then (*) can be rewritten as
\begin{itemize}
\item[(*)] $\sigma(\J)=\J$.
\end{itemize}
In the subsequent steps we will prove  (*) in this form.

{\it Step 2.} On this step we will reduce the proof to the case when $G$ is connected and semisimple.
Set $G':=(G^\circ,G^\circ)$. Assume that (*) holds for $G'$ instead of $G$. Then we have a parity antiautomorphism
$\sigma'$ of $\Dcal\red G'$. Replacing $G$ with $G/G'$ and $\Dcal$ with $\Dcal\red G'$
we may assume that $G^\circ$ is a torus.

We remark that any $\Phi(\xi)$ is $G^\circ$-invariant. Since $G^\circ$ is reductive,
$\I^{G^\circ}$ coincides with the left (=right) ideal in $\Dcal^{G^\circ}$
generated by $\Phi(\xi)$. For $f\in \Dcal^{G^\circ}$ and $\xi\in \g_0$ we have $\sigma(f\Phi(\xi))=
\Phi(\xi)\sigma(f)=\sigma(f)\Phi(\xi)$. So $\Dcal^{G^\circ}\Phi(\g_0)$ is  $\sigma$-stable. Now it is
easy to see that $\I^{G}$ is $\sigma$-stable.

{\it Step 3.} So we assume that $G$ is semisimple and connected. On this step we are going to reduce the proof
of (*) for $\Dcal$ to the proof of an analog of (*) for certain completions of $\Dcal$.

Pick a point $y\in \mu^{-1}(0)/G\subset \X/G$. Let $\m_y$ denote the ideal of $\pi_{G}^{-1}(y)$ in $\Str_\X$.
Denote by $\widetilde{\m}_y$ the inverse image of $\m_y$ in $\Dcal$. Consider the completion
$\Dcal^{\wedge_y}:=\varprojlim_{n\rightarrow \infty} \Dcal/\widetilde{\m}_y^n$. This is a complete and separated topological $\K[[h]]$-algebra equipped with a natural $G$-action. The quotient of $\Dcal^{\wedge_y}$ modulo
$h$ is naturally identified with the completion $\K[X]^{\wedge_y}:=\K[X^{\wedge_{\pi^{-1}(y)}}]$. Also we can define
the completion $(\Dcal^G)^{\wedge_y}:=\varprojlim_{n\rightarrow \infty} \Dcal^G/\m_y^n$. We remark that
$\sigma$ induces an antiautomorphism of $\Dcal^{\wedge_y}$ to be denoted by the same letter.

Consider the natural map $\Dcal\widehat{\otimes}_{\Dcal^G}(\Dcal^G)^{\wedge_y}\rightarrow \Dcal^{\wedge_y}$.
Here in the left hand side $\Dcal$ is equipped with the $h$-adic topology, while $(\Dcal^G)^{\wedge_y}$ has
the topology of a completion.
We claim that this map is bijective. Indeed, analogously to Lemma A2 in \cite{ES_appendix}, we see that
both $\Dcal^{\wedge_y}, (\Dcal^G)^{\wedge_y}$ are flat over $\K[[h]]$ and complete in the $h$-adic
topology. By assertion 2 of Lemma \ref{Lem:surj_proj}, $\Dcal\widehat{\otimes}_{\Dcal^G}(\Dcal^G)^{\wedge_y}$
is also flat over $\K[[h]]$. So it is enough to check that our map is an isomorphism modulo $h$, i.e.,
that a natural map $\K[X]\widehat{\otimes}_{\K[X]^G}(\K[X]^G)^{\wedge_y}\rightarrow \K[X]^{\wedge_y}$ is an isomorphism.
But this is straightforward from the construction.

It follows that $(\Dcal^{\wedge_y})^G=(\Dcal^G)^{\wedge_y}$.
Consider the closure $\J^{\wedge_y}$ of $\J$ in $(\Dcal^{\wedge_y})^G$. The isomorphism from the previous paragraph
implies that
$\J^\wedge_y= [\Dcal^{\wedge_y}\Span_\K(\Phi(\xi))]^G$.  Assume for a moment that
$\sigma(\J^{\wedge_y})=\J^{\wedge_y}$ for all $y$. Let us deduce from this that $\sigma(\J)=\J$.

Consider the functor of $\widetilde{\m}_y^G$-adic completion on the category of finitely generated
left $\Dcal^G$-modules. As in \cite{ES_appendix}, Lemma A2, one can show that this functor is exact.
So $\J^{\wedge_y}=\J^{\wedge_y}+\sigma(\J^{\wedge_y})$ coincides with the completion of $\J+\sigma(\J)$. Set
$\mathcal{N}:=(\sigma(\J)+\J)/\J$. The completion of this module at $y$ vanishes for all
$y\in \mu^{-1}(0)/G$. On the other hand, $\mathcal{N}\subset \Dcal^G/\J$ is supported on
$\mu^{-1}(0)/G$. It follows that $\mathcal{N}=\{0\}$.

{\it Step 4.} Let us consider a special case of $\X$ and $\Dcal$.

Consider the quantization $\Dcal:=\Dcal^{1/2}_{h}(G)$ of $\X:=T^*G$. Equip the algebra $\Dcal$
with a $G$-action induced from the action on $G$ by left translations. Then $\Phi(\xi):=\xi_G$
defines a quantum comoment map.

The algebra $\Dcal^G$ is generated (as a $\K[[h]]$-algebra) by the left invariant vector
fields and is naturally isomorphic to the $h$-adic completion $U_h(\g)^{\wedge_h}$ of $U_h(\g)$.
The quotient $U_h(\g)^{\wedge_h}/\J$ is flat over $\K[h]$ and $(U_h(\g)^{\wedge_h}/\J)/ h (U_h(\g)^{\wedge_h}/\J)=\K[T^*G\red_0 G]=\K$.
So $U_h(\g)^{\wedge_h}/\J=\K[[h]]$. From here it is easy to deduce that $\J=U_h(\g)^{\wedge_h}\g$. In particular, $\J$ is $\sigma$-stable.

We will need a trivial generalization of this result. Namely, the ideal $\J$ remains $\sigma$-stable
if we replace $\Dcal^{1/2}_h(G)$ with the tensor product $\Dcal^{1/2}_h(G)\widehat{\otimes}_{K[[h]]}\W_{h,2m}^{\wedge_h}$,
where $\W_{h,2m}^{\wedge_h}$ is  equipped with a trivial $G$-action.

{\it Step 5.} On this step we are dealing with general $\X,\Dcal$. Our goal is  to describe the structure of the the triple
$\Dcal^{\wedge_y}, \Phi:\g\rightarrow \Dcal^{\wedge_y}, \sigma:\Dcal^{\wedge_y}\rightarrow \Dcal^{\wedge_y}$.
Then we will deduce the equality $\sigma(\J^{\wedge_y})=\J^{\wedge_y}$ from this description.

Set $m:=\frac{1}{2}\dim \X-\dim G$. Consider the quantum algebra $\Dcal_h^{1/2}(G)\otimes_{\K[[h]]}\W_{h,2m}^{\wedge_h}$. Let $\Dcal'$ denote its completion with respect to the ideal
of the base $G\subset T^*G\hookrightarrow T^*G\times \K^{2m}$. This is an algebra equipped
with
\begin{itemize}
\item the product $G$-action, where the action on the Weyl algebra is supposed to be trivial,
\item a quantum comoment map $\Phi':\g\rightarrow \Dcal'$ induced from the quantum comoment
map $\g\rightarrow \Dcal^{1/2}_h(G)$,
\item A parity antiautomorphism $\sigma'$ that preserves the tensor product decomposition,
and coincides with the antiautomorphisms on the factors that were introduced above.
\end{itemize}

\begin{Lem}\label{Lem:equivalence}
There is a $G$-equivariant isomorphism $\iota:\Dcal^\wedge_y\rightarrow \Dcal'$ of $\K[[h]]$-algebras
intertwining the quantum comoment maps and the parity anti-automorphisms.
\end{Lem}
\begin{proof}[Proof of Lemma \ref{Lem:equivalence}]
Applying (a slight modification of) Theorem 3.3.4 from \cite{Miura} (without $\K^\times$-actions) we see that there is a $G$-equivariant isomorphism
$\iota_0:\Dcal^\wedge_y\rightarrow \Dcal'$ intertwining the quantum comoment maps (since $G$ is supposed
to be semisimple the compatibility with quantum comoment maps follows from the $G$-equivariance).
So we only need to prove the following claim:
\begin{itemize}
\item[(**)] Let $\sigma_1,\sigma_2$ be two $G$-equivariant parity anti-automorphisms of $\Dcal'$.
Then there is $f\in \Dcal'^G$ such that $\sigma_1=\exp(\ad f)\sigma_2\exp(-\ad f)$.
\end{itemize}

First of all, we remark that there is $f'\in \Dcal'^G$ such that $\sigma_1=\exp(\ad(f'))\circ \sigma_2$.
Indeed, $\sigma_1\circ \sigma_2^{-1}$ is a $G$-equivariant $\K[[h]]$-linear automorphism of $\Dcal'$
that is the identity modulo $h$. So $\sigma_1\circ\sigma_2^{-1}=\exp( h d)$, where $d$ is a $\K[[h]]$-linear
derivation of $\Dcal'$. But the completion of $T^*G\times \K^{2m}$ along any $G$-orbit has the trivial
first De Rham cohomology because $G$ is semisimple. This easily implies that any $\K[[h]]$-linear
derivation of $\Dcal'$ has the form $\frac{1}{h}\ad(f')$ for some $f'\in \Dcal'$. Since
$d$ is $G$-equivariant, we see that there is a $G$-invariant element $f'\in \Dcal'$ with
$d=\frac{1}{h}\ad(f')$.

So $\sigma_1=\exp(\ad(f'))\circ \sigma_2$. The equality $\sigma_1^2=\operatorname{id}$ implies
$\exp(\ad(f'))\exp(-\ad(\sigma_2(f')))=\operatorname{id}$. So $f'-\sigma_2(f')\in \K[[h]]$. Replacing
$f'$ with $f'- P$ for an appropriate series $P\in \K[[h]]$, we may assume that $f'=\sigma_2(f')$.

For $f_1,f_2\in \Dcal'$ let $f_1\circ f_2$ denote the Campbell-Hausdorff
series of $f_1,f_2$ (the series converges because $[\Dcal',\Dcal']\subset h \Dcal'$).
We have $$\exp(\ad(f))\sigma_2\exp(-\ad(f))=\exp(\ad(f))\exp(\ad(\sigma_2(f)))\sigma_2.$$
So it is enough to show that there exists a $G$-invariant element $f$ with
\begin{equation}\label{eq:cs_equation}f\circ \sigma_2(f)=f'.\end{equation} From the form of the Campbell-Hausdorff series and the inclusion
$[\Dcal',\Dcal']\subset h\Dcal'$  it is easy to deduce that (\ref{eq:cs_equation}) has a unique
solution that is automatically $G$-invariant.
\end{proof}
Now to prove that $\J^\wedge_y$ is $\sigma$-stable, we use the results of Step 4.
\end{proof}



Now let us proceed to the case when we reduce by the one-dimensional unipotent group.
\begin{Prop}\label{Prop:red_unipotent}
Let $G_0=\K$ be the one-dimensional unipotent group, $T:=\K^\times$, and
$G=T\ltimes G_0$, where $t\in T$ acts on $G_0$ by $(t,g_0)\mapsto t^a g_0$ for some $a\in \Z$.
Suppose that $\sigma(\Phi(\xi_0))=\Phi(\xi_0)$ for the unit element $\xi_0\in \g_0$ (this is automatically
true if $a\neq 0$), and $\sigma(\Phi(\eta))-\Phi(\eta)=a_1h$ for the unit element $\eta\in \mathfrak{t}=\K$.
Then
\begin{enumerate}
\item The quantization $\Dcal\red G_0$ is even.
\item Let $\Phi_0:\mathfrak{t}\rightarrow \Gamma(\X\red G_0, \Dcal\red G_0)$ be the quantum comoment
map induced by $\Phi$. Then $\sigma(\Phi_0(\eta))=\Phi_0(\eta)+(a_1-a)h$.
\end{enumerate}
\end{Prop}
\begin{proof}
Let us prove assertion 1. Lemma \ref{Lem:surj_proj} implies that $\Dcal\red G_0=\pi_*(\Dcal)^{G_0}/
\pi_*(\Dcal\Phi(\xi_0))^{G_0}$. But $\Phi(\xi_0)$ is $G_0$-invariant, and $\pi_*(\Dcal)^{G_0}$ is flat over
$\K[\Phi(\xi_0)]$. So we have the equality $\pi_*(\Dcal\Phi(\xi_0))^{G_0}=\pi_*(\Dcal)^{G_0}\Phi(\xi_0)$ and the
right hand side is $\sigma$-stable, compare with Step 3 of the proof of Proposition \ref{Prop:even_reduction}.
This implies assertion 1.

Let us prove assertion 2. Since the group $G_0$ is unipotent, there is an open affine
subset $Y^0\subset \X/G_0$ such that the restriction the quotient morphism $\pi:\X\rightarrow \X/G_0$ is trivial
over $Y^0$.  Replacing $\X$ with $\pi^{-1}(Y^0)$ we may assume that $\X/G_0$ is affine
and $\X=\X/G_0\times G_0$. Again we write $\Dcal$ instead of $\Gamma(\X,\Dcal)$.
Let $x$ denote the coordinate function on $G_0=\K$ so that $\xi_0.x=1$. We remark that the
operator $\xi_0:\K[X]\rightarrow \K[X]$ is surjective. It follows that $\xi_0:\Dcal\rightarrow \Dcal$
is surjective. So we can find a lifting $\tilde{x}$ of $x$ to $\Dcal$ such that $\xi_0.\tilde{x}=1$.


Consider the element $f=\Phi(\eta)-a\tilde{x}\Phi(\xi_0)\in \Dcal$. Then
$[f,\Phi(\xi_0)]=0$ and so $f\in \Dcal^G$. Moreover, the image of $f$ in $\Dcal\red G_0$
coincides with $\Phi_0(\eta)$. Let us compute $\sigma(f)$. We have
\begin{align*}&\sigma(f)=\sigma(\Phi(\eta))-a\sigma(\Phi(\xi_0))\sigma(\tilde{x})=\Phi(\eta)+ a_1h-a \Phi(\xi_0)\sigma(\tilde{x})=\\& =f+a_1h
-a[\Phi(\xi_0),\sigma(\tilde{x})]+a(\tilde{x}-\sigma(\tilde{x}))\Phi(\xi_0)=f+a_1h-ah\xi_0.\sigma(\tilde{x})
+a(\tilde{x}-\sigma(\tilde{x}))\Phi(\xi_0).\end{align*}
But $\xi_0.\sigma(\tilde{x})=\sigma(\xi_0.x)=1$ because $\sigma$ is $G_0$-equivariant. So we see that
$\sigma(f)$ coincides with $f+(a_1-a)h$ modulo $\Dcal^G\Phi(\xi_0)+
h^2\Dcal^G$. It follows that $\sigma(\Phi_0(\eta))$ is congruent $\Phi_0(\eta)+(a_1-a)h$  modulo
$h^2\Dcal^G$.
This implies assertion 2.
\end{proof}

\begin{proof}[Proof of Theorem \ref{Thm:even_quant}]
Similarly to Step 2 of the proof of Proposition \ref{Prop:even_reduction}, we may assume that
$G$ is connected. The character $\alpha$ does not change if we replace $G$ with its solvable
radical. So Proposition \ref{Prop:even_reduction} reduces the proof to the case when $G$ is solvable.

Now let $G_0$ be a one-dimensional normal unipotent subgroup in $G$. Let $\alpha_0$ be the character
of the action of $G$ on $\g_0$ and let $\Phi_0:\g/\g_0\rightarrow \Gamma(X\red G_0, \Dcal\red G_0)$
be the induced moment map. Proposition \ref{Prop:red_unipotent} implies that $\sigma(\Phi_0(\xi))-\Phi_0(\xi)=
\langle\alpha-\alpha_0,\xi\rangle h$ for any $\xi\in\g/\g_0$. But $\alpha-\alpha_0$ is nothing else
but the character of $G/G_0$ on $\bigwedge^{top}(\g/\g_0)$. So by induction we reduce the proof
to the case when $G$ is a torus. Here our claim follows from Proposition \ref{Prop:even_reduction}.
\end{proof}

\subsection{Proofs of the main theorems}\label{SUBSECTION_W_proofs}
\begin{proof}[Proof of Theorem \ref{Thm:Klein_W}]
Consider  the symplectic formal scheme $\widetilde{X}:=R(\D Q,\vv,\ww)^{\theta,ss}\widehat{\red}_{\z_0} \GL(\vv)$
over the formal neighborhood $\z_0^{\wedge_0}$ of 0 in $\z_0$, where $\theta$ is as in (\ref{eq:character}).
We claim that the symplectic schemes $\widetilde{\Sl}(e,B)$ and $\widetilde{X}$ are $\K^\times$-equivariantly symplectomorphic in such a way that the symplectomorphism lifts the isomorphism $\h^*\rightarrow \z_0$ constructed before Theorem \ref{Thm:Klein_W}.

We have  the line bundles $\mathcal{L}_i', i=0,\ldots,r$ on $X=\M^\theta_0(\D Q,\vv,\ww)$.
We claim that $\bigotimes \mathcal{L}_i^{\otimes\delta_i}\cong \Str_X$. This follows, for example, from a
well-known fact that if we identify $\Pic(X)$ with the weight lattice of $\g$, then $\mathcal{L}'_0$ gets identified
with $-\delta$, where $\delta$ is the maximal root. By Proposition \ref{Prop:comm_def_graded}, $\widetilde{X},\widetilde{\Sl}(e,B)$ are  obtained by pulling back the universal deformation
of $X$ by means of certain maps $\z_0\rightarrow H^2_{DR}(X),\h^*\rightarrow H^2_{DR}(X)$. Now Proposition \ref{Prop:line_bundle_transform} together with Proposition \ref{Prop_DH} imply that the maps  $\z_0\rightarrow H^2_{DR}(X), \h^*\rightarrow H^2_{DR}(X)$ are intertwined by the isomorphism $\h^*\rightarrow U$. This shows the claim of the previous paragraph.

Identify $\widetilde{X}$ with $\widetilde{\Sl}(e,B)$. Consider the canonical quantization
$\Dcal$ of this symplectic scheme over $\z_0^{\wedge_0}$. By Corollary \ref{Cor:quant_classif}, any even graded quantization of $\widetilde{X}$ is isomorphic to $\Dcal$. So Theorem \ref{Thm:even_quant} specifies the condition
on a comoment quantum map for the reduced quantization to be canonical. Thanks to Lemmas \ref{Lem:global_sections},
\ref{Lem:W_parab_qout}, the algebras $\W_{\z_0}(\D Q,\vv,\ww)_h, \Walg_{h,\h}$ are the subalgebras
of $\K^\times$-finite global sections of appropriate quantizations of $\widetilde{X}=\widetilde{\Sl}(e,B)$.
Since these quantizations are graded and even by Theorem \ref{Thm:even_quant}, they are isomorphic,
and so we see that the algebras $\W_{\z_0}(\D Q,\vv,\ww)_h$ and $\Walg_{h,\h}$ are isomorphic as $\K[\h^*,h]$-algebras.
\end{proof}

\begin{proof}[Proof of Theorem \ref{Thm:type_A}]
It is the same as the proof of Theorem \ref{Thm:Klein_W}, but one has to replace results from Subsection \ref{SUBSECTION_Klein_classical} with their counterparts from Subsection \ref{SUBSECTION_A_quiver_vs_Slod}.
Perhaps, the only new part is that we need to check that quantization commutes with reduction
for $\W(\D Q,\vv,\ww)_h$. Thanks to Lemma \ref{Lem:quant_comm_red}, we only need to check
that Proposition \ref{Prop:CB} applies in the present situation.

Consider the quiver $Q^{\ww}$ and the dimension vector $\vv^{\ww}$.
Suppose $\vv^{\ww}$ is decomposed into a sum $\vv'+\vv^1+\ldots+\vv^k$, with $\vv^1,\ldots,\vv^k$ being roots
for the Dynkin quiver $Q$. We need to show that $p(\vv)>\p(\vv')+\sum_{i=1}^k p(\vv^i)$.

We have $p(\vv^i)=0$ for $i=1,\ldots,k$ so we only need to check that
 $p(\vv)> p(\vv')$ for any $\vv'\leqslant \vv, \vv'\neq \vv$ and $v'_s=1$.
Let $\mathrm{v}'$ be the element of the root lattice associated to $\vv'$.
We have $p(\vv)=(\mathrm{d},\mathrm{v})-\frac{1}{2}(\mathrm{v},\mathrm{v}), p(\vv')=(\mathrm{d},\mathrm{v}')-
\frac{1}{2}(\mathrm{v}',\mathrm{v}')$. Here $(\cdot,\cdot)$ is the normalized invariant
scalar product, i.e.   $(\mathbf{x},\mathbf{y}):=
\sum_{i=1}^k x_i(2y_i-y_{i-1}-y_{i+1})$. Set $\mathbf{u}:=\vv-\vv'$, and $\mathrm{u}$ be the corresponding vector
in the root lattice.  Then we have $p(\vv)-p(\vv')=(\mathrm{u},\mathrm{d}-\frac{1}{2}(\mathrm{v}+\mathrm{v}'))=
(\mathrm{u}, \mathrm{d}-\mathrm{v}+\frac{1}{2}\mathrm{u})\geqslant \frac{1}{2}(\mathrm{u},\mathrm{u})$.
The last inequality holds because $\mathrm{d}-\mathrm{v}$ is dominant.
%
%
\end{proof}

\section{Symplectic reflection algebras}\label{SECTION_SRA}
\subsection{Definitions}\label{SUBSECTION_SRA_def}
Let $\mathcal{V}$ be a symplectic vector space with symplectic form $\omega$ and $\Gr\subset\Sp(\mathcal{V})$ be a finite group. Let $S$ denote the set of symplectic reflections in $\Gr$, that is
the set of all $g\in \Gr$ such that $\rk(g-\id)=2$. Decompose $S$ into the union
$\bigsqcup_{i=0}^r S_i$ of $\Gr$-conjugacy classes.
Pick  independent variables $c_0,c_1,\ldots,c_r$, one for each conjugacy class in $S$ and
also an independent variable $h$. Let $\param$ denote the vector space
dual to the span of $h,c_0,\ldots,c_r$, with dual basis $\check{h},\check{c}_0,\ldots,\check{c}_r$. By the (universal) symplectic reflection algebra (shortly, SRA) we mean the quotient $\Halg$ of $\K[\param]\otimes T\mathcal{V}\#\Gr$ by the relations
\begin{equation}\label{eq:SRA_relation}
[u,v]=h \omega(u,v)+\sum_{i=0}^r c_i\sum_{s\in S_i}\omega_s(u,v)s,
\end{equation}
with $\omega_s=\pi^*\omega$, where $\pi$ stands for the projection $V\twoheadrightarrow \operatorname{im}(s-\id)$
along $\ker (s-\id)$.

For $\beta\in \param$ let $H_\beta$ denote the specialization of $\Halg$ at $\beta$.

Let us list some properties of the algebra $\Halg$:
\begin{itemize}
\item[(A)] $\Halg/(h,c_0,\ldots,c_r) \Halg= S\mathcal{V}\#\Gr$.
\item[(B)] $\Halg$ is graded: $\K\Gr\subset \Halg$ has degree 0, $V\subset \Halg$ -- degree 1,
while $\param^*$ is of degree 2.
\item[(C)] $\Halg$ is flat over $\K[\param]$. This is a reformulation of results of Etingof
and Ginzburg, \cite{EG}, Theorem 1.3.
\item[(D)] Finally, and, in a sense, most importantly, $\Halg$ is universal with these three properties
under some mild restrictions on $\Gr$. Namely, assume that $\Gr$ is symplectically irreducible,
that is, there is no proper symplectic $\Gr$-submodule of $\mathcal{V}$. Then let $\param'$ be a vector space
and $\Halg'$ be a graded $\K[\param']$-algebra satisfying the analogs of (A),(B),(C). Then there is a
unique linear map $\param'\rightarrow \param$ such that $\Halg'\cong \K[\param']\otimes_{\K[\param]}\Halg$.
\end{itemize}
To prove (D) one argues as follows. The degree $-2$ component of the Hochshield cohomology
group $\operatorname{HH}^2(S\mathcal{V}\#\Gr)$ is identified with $\param^*$. All graded deformations
are unobstructed because the degree $-4$ component of $\operatorname{HH}^3(S\mathcal{V}\#\Gr)$ vanish.
To compute the Hochshield cohomology of $SV\#\Gr$ one argues similarly to
the proof of Theorem 9.1 in \cite{Etingof}.

In fact, below we will be interested mostly in the so called spherical subalgebra of $\Halg$.
Namely, consider the trivial idempotent $e=\frac{1}{|\Gr|}\sum_{\gamma\in \Gr}\gamma\in \K\Gr$
and note that $\K\Gr\subset \Halg$.  Form the {\it spherical  subalgebra} $e\Halg e$ with unit $e$.
This is a flat graded deformation of $(S\mathcal{V})^\Gr$.

We are mostly interested in the special case when $\Gr$ is the wreath-product $\Gamma_n$ of
a Kleinian group $\Gamma\subset \SL_2(\K)$ and of the symmetric group $S_n$, where $n>1$,
and $\mathcal{V}=L^{\oplus n}$, see Subsection \ref{SUBSECTION_resolution}.

Let $\Gamma\setminus\{1\}=\coprod_{i=1}^l S^0_i$ be the decomposition into $\Gamma$-conjugacy
classes. We have $S=S_{sym}\sqcup\coprod_{i=1}^l S_i$, where
\begin{align*}
&S_{sym}:=\{s_{ij}\gamma_{(i)}\gamma_{(j)}^{-1}, 1\leqslant i<j\leqslant n, \gamma\in \Gamma\},\\
&S_i:=\{\gamma_{(j)}, 1\leqslant j\leqslant n, \gamma\in S^0_i\}, i=1,\ldots,l.
\end{align*}
where $\gamma_{(j)}$ means the element $(1,\ldots,1,\gamma,1,\ldots,1)\in \Gamma^n$
with $\gamma$ on the $j$-th place, and $s_{ij}$ stands for the transposition of the $i$-th and
$j$-th elements in $S_n$.

We remark that $\Gamma$ is symplectically irreducible provided $\Gamma\neq \{1\}$.
We can make $S_n$ to act  symplectically irreducibly
if we replace $L^n$ with the double of the reflection representation of $S_n$.
Below we still write $c_1,\ldots,c_l$ for the independent variables corresponding to $S_i, i=1,\ldots,l$,
and we write $k$ for the variable corresponding to $S_{sym}$. Then (\ref{eq:SRA_relation})
becomes the same system of relations as (1.2.2),(1.2.3) in \cite{EGGO}. Of course, for $n=1$
we just do not have the class $S_{sym}$. However, it will be convenient for us to consider the space
$\widetilde{\param}:=\param$ for $n>1$ and $\widetilde{\param}:=\param\oplus \K \check{k}$ for $n=1$
and set $\widetilde{\Halg}:=\K[\widetilde{\param}]\otimes_{\K[\param]}\Halg$. So $\widetilde{\Halg}=\Halg$
for $n>1$ and $\K[k]\otimes \Halg$ for $n=1$.

\subsection{Main result}\label{SUBSECTION_SRA_main}
Our ultimate goal is to reprove results relating $e\Halg e$ to  certain quantum Hamiltonian reductions.
The latter is as follows.

Let $N_i, Q,\delta, \vv,\ww$ be  as in Subsection \ref{SUBSECTION_resolution}.
Set $V:=R(\D Q, \vv, \ww)$.
Further, set $G:=\GL(n\delta),\z:=\g^{*G}$.
Let $\W_h$ be the homogenized Weyl algebra of $V$.
Consider the reduction $\W(\D Q,\vv,\ww)_h:=\W_h\red G$.
This is a graded algebra over $\K[\z][h]$.

Let us state our main result. Recall that we have fixed a $\K^\times$-equivariant Poisson isomorphism $\K^{2n}/\Gamma_n\cong \M_0(\D Q,\vv,\ww)$.
Set $\mathbf{c}:=h+\sum_{i=1}^r c_i \sum_{\gamma\in S_i^0}\gamma\in \param^*\otimes Z(\K\Gamma)$. Further, recall the identification $\z\cong \K^{Q_0}$. It will be convenient
for us to change our usual notation and write $\epsilon_0,\ldots,\epsilon_r$ for the tautological basis of $\z^*$ (and not of $\z$).
Also set $\widehat{\z}^*:=\z^*\oplus \K h$ and let $\check{\epsilon}_0,\ldots, \check{\epsilon}_r,\check{h}$
be the dual basis in $\widehat{\z}$.

\begin{Thm}\label{Thm_SRA_iso}
Suppose $\Gamma\neq \{1\}$.
There is a graded algebra isomorphism
$e\widetilde{\Halg} e\rightarrow \W(\D Q,\vv,\ww)_h$,
that maps $\widetilde{\param}^*\subset e\widetilde{\Halg}e$ to $\widehat{\z}^*\subset \W(\D Q,\vv,\ww)_h$
and induces the fixed isomorphism $\K^{2n}/\Gamma_n\cong \M_0(\D Q,\vv,\ww)$.
The corresponding map $\upsilon:\widetilde{\param}^*\rightarrow \widehat{\z}^*$ is the inverse of the following map
\begin{equation}\label{eq:SRA_map}
\begin{split}
&h\mapsto h,\\
&\epsilon_0\mapsto \tr_{N_0}\mathbf{c}/|\Gamma|-(k+h)/2.\\
&\epsilon_i\mapsto \tr_{N_i}\mathbf{c}/|\Gamma|, i=1,\ldots,r.
\end{split}
\end{equation}
\end{Thm}

This theorem is  similar to the principal result of \cite{EGGO} but there are several differences.
First, (\ref{eq:SRA_map}) looks different from the analogous formula in \cite{EGGO}. This is because
their quantum comoment map differed from ours by a character (our quantum comoment map is, in a sense,
symmetrized but theirs is not). Second, our parameters are independent variables, while
\cite{EGGO} considers numerical values. Finally, the proof in \cite{EGGO} works only when the
quiver $Q$ is bi-partive, which is true for $\Gamma$ of types $D,E$ and $A_{l}$ for even $l$. The cases
$A_l$ for all $l>0$ are covered by the work of Oblomkov, \cite{Oblomkov} and Gordon, \cite{Gordon}.
The case of  $n=1$ follows basically from Holland's paper \cite{Holland}. Holland's results
may be interpreted as follows. The case $\Gamma=\{1\}$ was obtained in \cite{GG}. The proof
in this case can be obtained by a slight modification of our argument, but we are not going to
provide it.

For $n=1$ Theorem \ref{Thm_SRA_iso} together with Remark \ref{Rem:Klein} imply the following result,
which essentially was first proved by Holland. Set $\widehat{\z}_0:=\z_0\oplus \K h$.

\begin{Thm}\label{Thm:iso_Kleinian2}
Let $n=1$. There is a graded algebra isomorphism $e\Halg e\rightarrow\W_{\z_0}(\D Q,\vv)_h$
mapping $\param^*$ to  $\widehat{\z}_0^*$, where
the induced map $\upsilon_0:\param^*\rightarrow \widehat{\z}_0^*$ is the inverse of the following map
\begin{equation}\label{eq:SRA_map2}
\begin{split}
&h\mapsto h,\\
&\epsilon_0\mapsto \tr_{N_0}\mathbf{c}/|\Gamma|-h.\\
&\epsilon_i\mapsto \tr_{N_i}\mathbf{c}/|\Gamma|, i=1,\ldots,r.
\end{split}
\end{equation}
\end{Thm}

Theorem \ref{Thm_SRA_iso} for $n=1$ will be proved in Subsection \ref{SRA_Procesi}.
The case $n>1$ is much
more complicated. Its proof  will be completed in Subsection \ref{SUBSECTION_red_to_Klein}.

Let us explain the scheme of the proof.  To prove the existence of a graded endomorphism $\Upsilon: e\widetilde{\Halg}e\rightarrow \W(\D Q,\vv,\ww)_h$
mapping $\param^*=\widetilde{\param}^*$ to $\widehat{\z}^*$ is relatively easy. There are two ingredients for this
proof: the universality property of SRA's, see (D) in the the previous subsection, and the existence of a bundle
$\Pro$ on a resolution of $\K^{2n}/\Gamma_n$. This bundle allows to construct a deformation of
$S(\K^{2n})\#\Gamma_n$ whose spherical subalgebra is precisely $\W(\D Q,\vv,\ww)_h$. In the case $n=1$
we have enough information about the bundle $\Pro$ to recover the corresponding map $\param^*\rightarrow \widehat{\z}^*$ pretty easily. But this is not the case in general, so the most difficult part of the proof is to show that
the map $\param^*\rightarrow \widehat{\z}^*$ is as needed.

Of course, it is enough to show that the restriction  $\Upsilon|_{\param^*}$ differs from $\upsilon$
by a map induced by an automorphism of $\W(\D Q,\vv,\ww)_h$. We study certain group of automorphisms
of the algebras $\W(\D Q,\vv,\ww)_h$ in Subsection \ref{SUBSECTION_SRA_aut}. First we use Maffei's
construction of isomorphisms of quiver varieties, \cite{Maffei_Weyl}, to show that the group
we are interested in includes the Weyl group $W_{fin}$ of the Dynkin part of $Q$.
Then we check that for $n=1$ the group
under consideration basically coincides with $W_{fin}$. Finally, we produce a certain automorphism
of $\W(\D Q,\vv,\ww)_h$ for $n>1$ that (as we will see later) does not belong to $W_{fin}$.

Then our strategy to prove Theorem \ref{Thm_SRA_iso} for $n>1$ is to reduce it to the case $n=1$,
where the result is already known. This is achieved by considering certain completions
of $e\Halg e, \W(\D Q,\vv,\ww)_h$ and their isomorphism induced by $\Upsilon$. We will consider two
different completions. This will be done in Subsection \ref{SUBSECTION:SRA_completions}.

In Subsection \ref{SUBSECTION_red_to_Klein} we complete the proof of Theorem \ref{Thm_SRA_iso}.
We will show that the isomorphism of completions introduced in Subsection \ref{SUBSECTION:SRA_completions}
give rise to certain endomorphisms of some SRA with $n=1$. From here, thanks to results of
Subsection \ref{SUBSECTION_SRA_aut}, we will deduce
that $\Upsilon|_{\param^*}$ coincides with $\upsilon$ up to an element of $W_{fin}\times \Z/2\Z$
acting on $\widehat{\z}$. Then we will see that any element of this group is actually induced
by an automorphism of $\W(\D Q,\vv,\ww)_h$.

The last subsection of the section has nothing to do with the main theorem. There we use
our techniques to establish a result to be used in a subsequent paper \cite{GL}.

\subsection{An isomorphism via a (weakly) Procesi bundle}\label{SRA_Procesi}
Here we are going to prove that there is a $\K[h]$-linear graded algebra homomorphism $\Upsilon:
e\widetilde{\Halg} e\rightarrow\W(\D Q,\vv,\ww)_h$ mapping   $\widetilde{\param}^*$ to $\widehat{\z}^*$.
Then we will prove Theorem \ref{Thm_SRA_iso} for $n=1$.

Thanks to the universality
property of $\Halg$, the existence of $\Upsilon$ will follow if we produce
a graded flat $\K[\widehat{\z}^*]$-algebra $\widetilde{\Halg}'$ that deforms $\K[\K^{2n}]\#\Gamma$ and such
that $e\widetilde{\Halg}' e\cong \W(\D Q,\vv,\ww)_h$. For the algebra $\widetilde{\Halg}'$ we basically take the endomorphism algebra of a quantization of a  bundle  $\widehat{\Pro}$ on the symplectic variety $\widehat{X}$, see Subsection \ref{SUBSECTION_resolution}.

So let $\Dcal$ stand for the canonical quantization of $\widehat{X}$. By Proposition \ref{Prop:even_reduction},
$\Dcal\cong \W_{h,V^*}\red^\theta G$.  Recall, Lemma \ref{Lem:global_sections}, that the subalgebra
of $\K^\times$-finite elements in $\Dcal$ is naturally identified with $\W(\D Q, \vv,\ww)_h$.

Since $\Ext^i_{\Str_{\widehat{X}}}(\widehat{\Pro},\widehat{\Pro})=0$, one can lift $\widehat{\Pro}$ to a unique projective
$\K^\times\times \Gamma_n$-equivariant right $\Dcal$-module $\widehat{\Pro}_h$. Automatically,  $\End_{\Dcal^{opp}}(\widehat{\Pro}_h)/(h)=\End_{\Str_{\widehat{X}}}(\widehat{\Pro})$ and so
$\End_{\Dcal^{opp}}(\widehat{\Pro}_h)$ is flat over $\K[\widehat{\z}]$.
So we see that $\End_{\Dcal^{opp}}(\widehat{\Pro}_h)/(\widehat{\z})=\End_{\Str_X}(\Pro)=\K[\K^{2n}]\#\Gamma_n$.
Let $\widetilde{\Halg}'$ be the subalgebra of $\K^\times$-finite elements in $\End_{\Dcal^{opp}}(\widehat{\Pro}_h)$.
Since $\widehat{\Pro}^{\Gamma_n}=\Str_{\widehat{X}}$, we see that $\widehat{\Pro}_h^{\Gamma_n}=\Dcal$.
So $\Dcal\cong e [{\mathcal End}_{\Dcal^{opp}}(\widehat{\Pro}_h)]e$, and $\W(\D Q,\vv,\ww)_h\cong e \widetilde{\Halg}' e$.

The graded flat deformation $\widetilde{\Halg}'$ of $\K[\K^{2n}]\#\Gamma_n$ gives rise to a (unique for $n>1$) linear map $\widehat{\z}\rightarrow \widetilde{\param}$ with $\widetilde{\Halg}'\cong \K[\widehat{\z}]\otimes_{\K[\widetilde{\param}]}\Halg$.
Taking the spherical subalgebras we get a homomorphism $\Upsilon: e\widetilde{\Halg} e\rightarrow
\W(\D Q,\vv,\ww)_h$. Our problem now becomes to determine the restriction $\Upsilon|_{\param^*}$.

First of all, let us show that $\Upsilon(h)=h$.
The algebras $e\widetilde{\Halg}'e/(h),e\widetilde{\Halg}e/(h)$
are commutative (for the first one this is evident, and for the second one follows from \cite{EG}).
On the other hand, for $\beta'\in \widehat{\z}$ with $\langle\beta',h\rangle\neq 0$ the specialization
$\W_{\beta'}(\D Q,\vv,\ww)_h$ is not commutative (because the Poisson bracket on $\K[\K^{2n}]^\Gamma$
induced by the bracket on the filtered algebra $\W_{\beta'}(\D Q,\vv,\ww)_h$ is nonzero). Similarly,
for $\beta\in \param$ with $\langle\beta,h\rangle\neq 0$ the algebra $e H_\beta e$ is non-commutative
for the same reason. This implies that $\Upsilon(h)$ is a non-zero multiple of $h$.
Now consider the Poisson brackets on $\K[\M_0(\D Q,\vv,\ww)]$ and on $\K[\K^{2n}]^{\Gamma_n}$.
Let $\rho:\M(\D Q,\vv,\ww)_h\twoheadrightarrow \K[\M_0(\D Q,\vv,\ww)]$ be the canonical projection.
The bracket on $\K[\M_0(\D Q,\vv,\ww)]$ is given by
$\{a,b\}=\rho(\frac{1}{h}[\iota(a),\iota(b)])$, where  $\iota: \K[\M_0(\D Q,\vv,\ww)]\rightarrow
\W(\D Q,\vv,\ww)_h$ is any section of $\rho$. The bracket on $\K[\K^{2n}]^{\Gamma}$ can be defined
using $e\Halg e$ in a similar way. Since our identification of $\K[\K^{2n}]^{\Gamma_n}$
and $\K[\M_0(\D Q,\vv,\ww)]$ was Poisson, we see that $\Upsilon(h)=h$.

In fact, the case $n=1$ is particularly easy because in this case we know much more about
$\widehat{\Pro}$ than in general, see Subsection \ref{SUBSECTION_Klein_classical}. We use the description
of $\widehat{\Pro}$ given there to prove Theorem \ref{Thm_SRA_iso} for $n=1$.

In Subsection \ref{SUBSECTION_Klein_classical} we constructed $\widehat{\Pro}$ explicitly for any generic $\theta$.
Now suppose that $\theta$ is chosen as explained after Proposition \ref{Prop:Nakajima1}.
In fact, one can interpret $\widehat{X},\widehat{\Pro}$ in a slightly  different way. Namely, consider the variety
$Y$ of all $\Gamma$-linear algebra homomorphisms $\widetilde{\Halg}/(h)\rightarrow \operatorname{End}(\K\Gamma)$. This variety comes equipped with a $\GL(\vv)=\GL(\K\Gamma)$-action and with a canonical $\GL(\vv)$-equivariant bundle $\mathcal{C}$ of rank $|\Gamma|$. Now $Y^{\theta,ss}$ consists of all homomorphisms having a "cocyclic covector", i.e., an element $\alpha\in (\K\Gamma)^*$ such that $a^*(\alpha), a\in \widetilde{\Halg}/(h)$ span the whole space $(\K\Gamma)^*$.
Consider the variety $\widehat{X}':=Y^{\theta,ss}/\GL(\K\Gamma)$ and  the bundle  $\widehat{\Pro}'$ on $\widehat{X}'$
obtained from  $\mathcal{C}$ by descend.
It is a standard fact on the McKay correspondence that there is an $\K^\times$-equivariant isomorphism
$\eta:\widehat{X}'\rightarrow \widehat{X}$ such that
\begin{itemize}
\item $\widehat{\Pro}'\cong \eta^*(\widehat{\Pro})$.
\item $\Psi$ lifts the map $\param^*/\K h\rightarrow \z^*_0/ \K h$ induced by (\ref{eq:SRA_map2}).
\end{itemize}

By the definition of $\widehat{X}',\widehat{\Pro}'$, we have a natural homomorphism $\widetilde{\Halg}/(h)\rightarrow\End_{\Str_{\widehat{X}'}}(\widehat{\Pro}')$. This homomorphism is
the identity modulo $\z_0$ and both algebras are flat over $\K[\z_0]$. It follows that this homomorphism
is an isomorphism. So we see that there is an isomorphism $e\Halg e\cong \K[\param]\otimes_{\K[\widehat{\z}_0]}\W_{\z_0}(\D Q,\vv,\ww)_h$ with the isomorphism $\upsilon'_0: \param^*\rightarrow
\widehat{\z}_0^*$ that maps $h$ to $h$ and is congruent to $\upsilon_0$ modulo $\K h$.
We need to show that this isomorphism actually coincides with $\upsilon_0$.

To show this let us consider the set $\param_{sing}$ of all elements $\beta$ in $\param_1:=\{\beta\in \param| \langle \beta,h\rangle=1\}$ such that the algebra $ e H_{\beta} e$ has infinite homological dimension. According
to \cite{CBH}, this set is a finite union of hyperplanes, whose only common intersection
point is the only point  $\beta^0$ such that the corresponding element $\mathbf{c}^0:=1+\sum_{i=1}^r \langle \beta^0,c_i\rangle \sum_{\gamma\in S_i^0}\gamma$ satisfies $\tr_{N_i}(\mathbf{c}^0)=0$.
So to prove $\upsilon_0=\upsilon_0'$
it is enough to show that $\langle \upsilon'^{-1*}(\beta^0), \epsilon_i\rangle=0$ for $i=1,\ldots,r$.

Recall the isomorphism $\W_{\z_0}(\D Q,\vv,\ww)_h\cong \Walg_{h,\h}$,
Theorem \ref{Thm:Klein_W}.
It is known, thanks to the localization theorems from
\cite{Ginzburg_HC} or \cite{DK}, that the algebra $\Walg_{\lambda}$
has finite homological dimension provided $\langle\lambda,\alpha\rangle\neq 0$ for any root
$\alpha$. The number of hyperplanes in $\param_{sing}$ coincides with the number of positive
roots. So $\upsilon'^{*-1}(\beta^0)$ is the only intersection point of the hyperplanes
$\ker\alpha$ and hence vanishes on $\epsilon_i, i=1,\ldots,r$.

We remark that the claim of the previous sentence can be obtained also without using Theorem
\ref{Thm:Klein_W} and the localization. For this one needs to use the results on automorphisms
of $\W_{\widehat{\z}_0}(\D Q,\vv,\ww)$ obtained in the next subsection. We will see that
a unique fixed point of certain group of automorphisms of $\W_{\widehat{\z}_0}(\D Q,\vv,\ww)_h$
vanishes on $\epsilon_i,i=1,\ldots,r$. On the other hand, this group obviously preserves $\upsilon_0'^{*-1}(\param_{sing})$.

\subsection{Automorphisms}\label{SUBSECTION_SRA_aut}
In this subsection we will study the graded automorphisms of the algebras $\W(\D Q,\vv,\ww)_h$ that preserve $\widehat{\z}^*$  and are the identity modulo $\widehat{\z}^*$. For this we will need to recall
a construction appearing in different forms in the work of Nakajima, Lusztig and Maffei. We will follow
\cite{Maffei_Weyl}.

Consider, for a moment, an arbitrary quiver $\widetilde{Q}=(\widetilde{Q}_0,\widetilde{Q}_1)$ and a dimension vector $\widetilde{\vv}$. Let $W$ be the Weyl group
of the quiver $\widetilde{Q}$.  The group $W$ acts on $\widetilde{\z}:=\K^{\widetilde{Q}_0}$.

\begin{Prop}\label{Prop:Weyl_action}
Pick a character $\theta$ of $\GL(\widetilde{\vv})$ such that $\theta \cdot \widetilde{\vv}=0$.
Let $i$ be a loop free vertex of $\widetilde{Q}$, and $s=s_i, i\in \widetilde{Q}_0,$ be
the simple reflection corresponding to $i$. Suppose that $\theta_i\neq 0, s\widetilde{\vv}\geqslant 0$. Then there exists
a $\K^\times$-equivariant isomorphism $S: \M^\theta(\D \widetilde{Q}, \widetilde{\vv})\xrightarrow{\sim} \M^{s\theta}(\D \widetilde{Q}, s\widetilde{\vv})$
lifting $s:\K^{\widetilde{Q}_0}\rightarrow \K^{\widetilde{Q}_0}$.
\end{Prop}


The construction of the isomorphism is due to Maffei, \cite{Maffei_Weyl}. We will recall his construction
in the setting we need.

Without loss of generality we may assume that $i$ is a source in $\widetilde{Q}$ (i.e., $h(a)\neq i$
for all $a\in \widetilde{Q}_1$).
Set $\widetilde{\vv}':=s\widetilde{\vv},\theta':=s\theta$.
Also we may assume  that $\theta_i>0$ (because $\theta_i, \theta'_i$ have different signs).

Pick the spaces $V_i$ of dimension $\widetilde{\vv}_i, i\in \widetilde{Q}_0,$ and set $T:=\bigoplus_{a, t(a)=i}V_i$.
Also for any $i$ pick a vector space $V_i'$ of dimension $\widetilde{\vv}'_i$.
We remark that $\dim T=\widetilde{\vv}_i+\widetilde{\vv}'_i$. Consider the space \begin{align*}\widetilde{V}&:= \bigoplus_{a, t(a)\neq i}\left(\Hom(V_{t(a)},V_{h(a)})\oplus \Hom(V_{h(a)},V_{t(a)})\right)\oplus\\&\oplus \Hom(V_i,T)\oplus \Hom(T,V_i)\oplus \Hom(T,V_i')\oplus \Hom(V_i',T).\end{align*}

We write an element of $\widetilde{V}$ as $((A_a),(B_a), A,B, A',B')$ with $ A_a\in \Hom(V_{t(a)},V_{h(a)}),
B_a\in \Hom(V_{h(a)},V_{t(a)}), A\in \Hom(V_i,T), B\in \Hom(T,V_i), A'\in \Hom(V_i',T), B'\in \Hom(T,V_i')$.

The space $\widetilde{V}$ comes with a natural  action of  a group $\widetilde{G}:=G\times \GL(\widetilde{\vv}_i)\times \GL(\widetilde{\vv}'_i), G:=\prod_{j\neq i}\GL(\widetilde{\vv}_j)$. Moreover, $\widetilde{V}^{\GL(\widetilde{\vv}'_i)}=R(\D \widetilde{Q}, \widetilde{\vv}), \widetilde{V}^{\GL(\widetilde{\vv}_i)}=R(\D \widetilde{Q}, \widetilde{\vv}')$. Let $\pi,\pi'$ denote the natural projections $\widetilde{V}\twoheadrightarrow R(\D \widetilde{Q},\widetilde{\vv}), R(\D \widetilde{Q}, \widetilde{\vv}')$.

Consider the  locally closed subvariety $Z\subset \widetilde{V}$ consisting of all vectors
$x=((A_a),(B_a),$ $A,B,A',B')$ such that
\begin{enumerate}
\item The sequence $0\rightarrow V_i'\xrightarrow{A'} T\xrightarrow{B} V_i\rightarrow 0$ is exact.
\item $\pi(x)\in \Lambda_{\chi}(\D \widetilde{Q},\widetilde{\vv}), \pi'(x)\in \Lambda_{s\chi}(\D \widetilde{Q}, \widetilde{\vv}')$ for some $\chi\in \K^{\widetilde{Q}_0}$.
\item $A'B'=AB-\chi_i \operatorname{id}_{T}$.
\end{enumerate}

In particular, we see that $\pi,\pi'$ induce  $\widetilde{G}$-equivariant projections
$Z\rightarrow \Lambda(\D \widetilde{Q}, \widetilde{\vv}),\Lambda(\D \widetilde{Q}, \widetilde{\vv}')$  denoted by $\rho,\rho'$. Thanks to \cite{Maffei_Weyl}, Lemma 30,
$\rho^{-1}(\Lambda(\D \widetilde{Q},\widetilde{\vv})^{\theta,ss})=\rho'^{-1}(\Lambda(\D \widetilde{Q}, \widetilde{\vv})^{\theta',ss})$ (for applying this lemma  we need to assume that $\theta_i>0$).
Denote these equal subvarieties of $Z$ by $Z^{ss}$. Now thanks to Lemma 33 in \cite{Maffei_Weyl}, the restrictions
\begin{equation}
\begin{split}
&\rho: Z^{ss}\rightarrow \Lambda(\D \widetilde{Q},\widetilde{\vv})^{\theta,ss},\\
&\rho': Z^{ss}\rightarrow \Lambda(\D \widetilde{Q},\widetilde{\vv}')^{\theta',ss}
\end{split}
\end{equation}
are principal $\GL(\widetilde{\mathbf{v}}_i')$- and $\GL(\widetilde{\mathbf{v}}_i)$-bundles,
respectively. This gives  isomorphisms $\M^\theta(\D \widetilde{Q},\widetilde{\vv})\cong Z^{ss}/\widetilde{G}\cong \M^{\theta'}(\D \widetilde{Q},\widetilde{\vv}')$. For $S_i$ we take the resulting isomorphism
$\M^\theta(\D \widetilde{Q},\widetilde{\vv})\xrightarrow{\sim} \M^{\theta'}(\D \widetilde{Q}, \widetilde{\vv}')$.
This isomorphism is $\K^\times$-equivariant by the construction. The claim that it lifts
$s_i:\K^{\widetilde{Q}_0}\rightarrow \K^{\widetilde{Q}_0}$ follows from the condition (2) in the definition of $Z$.
Most likely, $S$ is always a Poisson isomorphism but we will not need this fact in the whole generality.

Now consider the special case when $s\widetilde{\vv}=\widetilde{\vv}$. Suppose that the natural morphisms $\M_0^{\theta}(\D \widetilde{Q},\widetilde{\vv})
\rightarrow \M_0(\D \widetilde{Q},\widetilde{\vv}), M_0^{s\theta}(\D \widetilde{Q},\widetilde{\vv})\rightarrow \M_0(\D \widetilde{Q},\widetilde{\vv})$ are resolutions of singularities.
\begin{Lem}\label{Lem:Maffei_commutativ} The following diagram is commutative.

\begin{picture}(60,30)
\put(2,22){$\M_0^\theta(\D \widetilde{Q},\widetilde{\vv})$}
\put(36,22){$\M_0^{s\theta}(\D \widetilde{Q},\widetilde{\vv})$}
\put(16,2){$\M_0(\D \widetilde{Q},\widetilde{\vv})$}
\put(23,24){\vector(1,0){13}}\put(26,25){\tiny $S$}
\put(9,20){\vector(1,-1){13}}\put(44,20){\vector(-1,-1){13}}
\end{picture}
\end{Lem}
\begin{proof}
The lemma will follow if we check the following claim:
\begin{itemize}
\item Suppose that $x=((A_a),(B_a), A,B,A',B')$ is an element of $Z$ such that $\chi=0$ (see (1)-(3)) above.
Then $f(\pi(x))=f(\pi'(x))$ for any $f\in \K[R(\D \widetilde{Q},\widetilde{\vv})]^{\GL(\widetilde{\vv})}$.
\end{itemize}
Le Bruyn and Procesi found a set of generators for the algebra $\K[R(\D \widetilde{Q},\widetilde{\vv})]^{\GL(\widetilde{\vv})}$ in \cite{LeBruynProcesi}. Let $p:=(a_0,\ldots,a_k)$ be a cyclic (i.e., $t(a_0)=h(a_k)$) path in $\D \widetilde{Q}$. To this path we can assign a polynomial $f_p$ mapping $x\in R(\D \widetilde{Q},\widetilde{\vv})$ to $\tr(x_{a_k}\ldots x_{a_0})$, where $x_{a_k}$ is the component of $x$ corresponding to $a_k$. The polynomials $f_p$ generate $\K[R(\D \widetilde{Q},\widetilde{\vv})]^{\GL(\widetilde{\vv})}$. So it remains to check that $f_p(\pi(x))=f_p(\pi'(x))$ for any path $p$ and any $x\in Z$ with $\chi=0$, equivalently,  with $AB=A'B'$. Since the trace of a product is stable under a cyclic permutation of the factors, we may assume
that $h(a_k)=t(a_0)\neq i$. In this case the products for $\pi(x)$ and $\pi'(x)$ are the same, thanks to
the equality $AB=A'B'$.
\end{proof}

We will apply the construction above to the special case explained in Subsection \ref{SUBSECTION_SRA_main}.
Let $Q,n,\widetilde{\vv},\ww$ be as in the beginning of that subsection. We are interested in the quiver
$\widetilde{Q}:=Q^{\ww}$ and the dimension vector $\widetilde{\vv}:=\vv^{\ww}$.

\begin{Lem}\label{Lem:Poisson}
Preserve the notation of Proposition \ref{Prop:Weyl_action}. Let $i=1,2,\ldots,r$. Then
the morphism $S: \M^\theta(\D Q^{\ww},\vv^{\ww})\rightarrow \M^{s\theta}(\D Q^{\ww},\vv^{\ww})$
is Poisson.
\end{Lem}
\begin{proof}
Since the morphisms $\M^\theta(\D Q^{\ww},\vv^\ww),M^{s\theta}(\D Q^{\ww},\vv^{\ww})\rightarrow \M(\D Q^{\ww},\vv^{\ww})$ are Poisson and birational,
it is enough to prove that the morphism $\M(\D Q^{\ww},\vv^{\ww})\rightarrow \M(\D Q^{\ww},\vv^{\ww})$ induced by
$S$ (and also denoted by $S$) is Poisson. Let $\{\cdot,\cdot\}_\chi$ denote the Poisson bracket on
$\M_\chi(\D Q^{\ww},\vv^{\ww})$.  We need to show that $\{\cdot,\cdot\}_\chi=S^*\{\cdot,\cdot\}_{s\chi}$.

The Poisson algebra $\K[\M_\chi(\D Q^{\ww},\vv^{\ww})]$ is filtered and the associated graded algebra is $\K[\M_0(\D Q^{\ww},\vv^{\ww})]$. Moreover, $\{\cdot,\cdot\}_\chi$ decreases the degree by 2,
and the induced bracket on $\K[\M_0(\D Q^{\ww},\vv^{\ww})]$ coincides with $\{\cdot,\cdot\}_0$.
The automorphism of $\M_0(\D Q^{\ww},\vv^{\ww})$ induced by $S$ is the identity by
Lemma \ref{Lem:Maffei_commutativ}. It follows that $\{\cdot,\cdot\}_\chi- S^*\{\cdot,\cdot\}_{s\chi}$
decreases degrees at least by 3. But there are no brackets on $\K[\M_0(\D Q^{\ww},\vv^{\ww})]$
of degree less than $-2$, see Proposition \ref{Prop:symplect_quotient_sing}. So $\{\cdot,\cdot\}_\chi= S^*\{\cdot,\cdot\}_{s\chi}$.
\end{proof}

Now let us proceed to the quantum situation. Let $\Dcal^\theta$ denote the reduction
$\W_h\red^\theta \GL(\vv)$. Consider the sheaves $\Dcal^\theta, S^*(\Dcal^{s\theta})$ on $\M^\theta(\D Q,\vv,\ww)$. Thanks to Proposition \ref{Prop:even_reduction}, both are canonical quantizations of $\M^\theta(\D Q,\vv,\ww)/\z$ but with respect to different morphisms $M^\theta(\D Q,\vv,\ww)\rightarrow \z$: $\Dcal^\theta$ -- for the original map $\M^\theta(\D Q,\vv)\rightarrow \z$, while $S^*(\Dcal^{s\theta})$ -- for its composition with $s:\z\rightarrow \z$. It follows that there is a $\K^\times$-equivariant isomorphism $\iota:\Dcal^\theta\xrightarrow{\sim} S^*(\Dcal^{s\theta})$
that induces $s$ on $\Str_\z$. Thanks to Proposition \ref{Lem:global_sections}, we get a $\K[h]$-linear
automorphism $S:\W(\D Q,\vv,\ww)_h\rightarrow \W(\D Q,\vv,\ww)_h$ that acts as $s$ on $\z^*$.
Moreover, $S$ is the identity modulo $\widehat{\z}^*$.

Let $W_{fin}$ denote the Weyl group of the Dynkin part of $Q$.
Let $\Afr$ denote the group of automorphisms of $\W(\D Q,\vv,\ww)_h$ that are $\K^\times$-equivariant,
$\K[h]$-linear, preserve $\widehat{\z}^*\subset \W(\D Q,\vv,\ww)_h$ and induce the identity modulo
$\widehat{\z}^*$.  We have a natural homomorphism $\Afr\rightarrow \GL(\z)$.
The assignment $s\mapsto S$ extends to a homomorphism $W_{fin}\rightarrow \Afr$ whose composition with
$\Afr\rightarrow \GL(\z)$ is the identity.   To see this one either applies results of Maffei,
\cite{Maffei_Weyl}, or uses the following lemma.

\begin{Lem}\label{Lem:AutKleinian}
The restriction of the $\Afr$-action to $\z$ defines an embedding $\Afr\hookrightarrow \GL(\z)$.
\end{Lem}
\begin{proof}
Let $a\in \Afr$ be an element in the kernel. Then $a$ acts by the identity on both  $\widehat{\z}^*$ and on $\W(\D Q,\vv,\ww)_h/(\widehat{\z}^*)$. So for any choice of $\beta\in \widehat{\z}^*$ the element $a$ induces a filtration preserving automorphism $a_\beta$ of $\W_\beta(\D Q,\vv,\ww)$ (the specialization of $\W(\D Q,\vv,\ww)_h$ at $\beta$) such that $\gr a_\beta$ is the identity. If all $a_\beta$ are the identity, then so is $a$. Indeed, for any different elements $x,y\in \W(\D Q,\vv,\ww)_h$ there is $\beta\in \widehat{\z}^*$ such that the images of $x,y$ in $\W_\beta(\D Q,\vv,\ww)$ are different. So let us prove that $a_\beta$ is the identity for any $\beta$.

Let $\operatorname{F}_i \W_\beta(\D Q,\vv,\ww)$ denote the natural filtration on $\W_\beta(\D Q,\vv,\ww)$.
Since the associated graded of $a_\beta$ is the identity, we see that the restriction of $a_\beta$ to  $\operatorname{F}_i \W_\beta(\D Q,\vv,\ww)$ is unipotent for any $i$. So $d=\ln(a_\beta)$ is defined and is a derivation of $\W_\beta(\D Q,\vv,\ww)$ that reduces the degrees. Let $m$ be the maximal integer such $d\operatorname{F}_i \W_\beta(\D Q,\vv,\ww)\subset \operatorname{F}_{i-m} \W_\beta(\D Q,\vv,\ww)$ for all $i$. Then $d$ induces a derivation $d_0$ of $\gr \W_\beta(\D Q,\vv,\ww)=(S\K^{2n})^{\Gamma_n}$ of degree $-m$. The derivation $d_0$ is Poisson by the construction. On the other hand, $d_0$ can be extended to a (automatically Poisson) derivation of $S(\K^{2n})$. This is because the morphism $\K^{2n}\rightarrow \K^{2n}/\Gamma_n$ is \'{e}tale in codimension 1. But any Poisson derivation of $S(\K^{2n})$ is Hamiltonian. It follows that $d_0$ is Hamiltonian, i.e, $d_0=\{f,\cdot\}$ for some $f\in S(\K^{2n})^{\Gamma_n}$. Since $d_0$ is of degree $-m$,  one can choose $f$ of degree $2-m$. But $m\geqslant 1$ so the degree of $f$ is less than 2. Since $\Gamma_n$ has no fixed points in $\K^{2n}$, we get a contradiction.
\end{proof}

Now consider the case $n=1$. Consider the subgroup $\widetilde{\Afr}$ of $\GL(\widehat{\z})$ consisting
of all maps $\widehat{\z}^*\rightarrow \widehat{\z}^*$ that send $h$ to $h$, preserve $\z^*_0$, and  coincide with an element of $W_{fin}$ on $\z^*_0$. We claim that the image of $\Afr$ in $\GL(\widehat{\z}^*)$ coincides with $\widetilde{\Afr}$.
This is a consequence of the following more general and more technical statement to be used also
later.

\begin{Prop}\label{Prop:AutKleinian_main}
Let $n=1$.
Let $U_1,U_2$ be finite dimensional vector spaces, 
$\A_i:=\K[U_i]\otimes \W_{\z_0}(\D Q,\vv,\ww)_h$. Let $\varphi:\A_1\rightarrow \A_2$ be a $\K[h]$-linear
endomorphism that maps $U_1^*\oplus\widehat{\z}_0^*$ to $U_2^*\oplus \widehat{\z}_0^*$
and induces the identity map on $S(\K^2)^{\Gamma}$. Then $\varphi$ maps $\z_0^*$ to $\z_0^*$
and induces an element from $W_{fin}$ on $\z_0^*$.
\end{Prop}
\begin{proof}
It is enough to consider the case when $U_1=0$.
Let $(\widehat{\z}_0)_{sing}$ be the set of all $\beta\in \widehat{\z}_0$ with $\langle\beta,h\rangle=1$
and $\langle\beta,\alpha^\vee\rangle=0$ for some $\alpha\in \Delta$, where $\Delta\subset \z^*_0$
is the finite  part of the  root system of $Q$.

Pick $u\in U_2$.
The endomorphism $\varphi$ induces a filtered algebra homomorphism $$\W_{\varphi^*(\beta+u)}(\D Q,\vv,\ww)\rightarrow
\W_{\beta}(\D Q,\vv,\ww)$$
whose associated  graded is the identity. Hence this homomorphism is an isomorphism.
It follows that for any $u\in U_2$
and $\beta\in (\widehat{\z}_0)_{sing}$
we have $\varphi^*(\beta)+\varphi^*(u)=\varphi^*(\beta+u)\in (\widehat{\z}_0)_{sing}$. Since
$(\widehat{\z}_0)_{sing}$  is not stable under translations by a vector, we see that $\varphi^*(u)=0$.
By the same reason, $\varphi^*$ is bijective on $\widehat{\z}_0$.

It remains to show that $\varphi$ preserves $\z_0^*$ and induces an element from $W_{fin}$ on this space.
To show this it is enough to consider the case when $U_2=0$. Here $\varphi$ is an automorphism
of $\W_{\z_0}(\D Q,\vv,\ww)_h$.

Thanks to Lemma \ref{Lem:AutKleinian},
 we can identify $\varphi$ with its image in $\GL(\widehat{\z}^*_0)$.
From the discussion at the end of the previous subsection we see  that $\varphi$ preserves $(\z_0^*)_{sing}$.
From here it is easy to deduce that $\varphi \in \K^\times W_{fin}  \Afr_0$, where $\Afr_0$
stands for the automorphism group of the Dynkin diagram and $\K^\times$ is viewed as
the group of scalar matrices. Now, according to \cite{CBH}, the global dimension
of  $e H_\beta e$ for $\beta\in \param^*_1$ is bigger than 1 if and only if $c$ is annihilated
by some real root of the affine quiver $Q$. The automorphism $\varphi$ preserves the corresponding subset
of $\widehat{\z}_0$. From here it is easy to deduce that $\varphi\in W_{fin} \Afr_0$. So it remains to show
that $\varphi\in \Afr_0$ implies $\varphi=\operatorname{id}$.

Assume the converse, let $\varphi$ be a nontrivial element of $\Afr_0$.
Recall the element $\check{h}\in \param$.
 The element $\upsilon_0^*(\check{h})\in \widehat{\z}_0$ is $\Afr_0$-stable.
Consider the algebra $H_{\check{h}}$. This algebra is just $\W_2\#\Gamma$, where $\W_2$
stands for the usual (not homogenized) Weyl algebra of $\K^2$.
So the spherical subalgebra $e\Halg_{\check{h}} e$ is just $\W_2^{\Gamma}$.
The restriction of $\varphi$ to $e\Halg_{\check{h}} e=\W_2^{\Gamma}$ is the identity
because the associated graded of $\varphi$ is the identity
(compare with the proof of Lemma \ref{Lem:AutKleinian}).
On the other hand, according to \cite{EG}, Theorem 2.16, the completion of $e\Halg e/(h-1)$ at $\check{h}$
is the universal formal deformation of $\W_2^\Gamma$. In particular, we have an
$\Afr_0$-equivariant isomorphism $\operatorname{HH}^2(\W_2^\Gamma)\xrightarrow{\sim} \param_1$
(the r.h.s. is viewed
as a vector space with origin $\check{h}$).
But the action of $\varphi$ on $\operatorname{HH}^2(\W_2^\Gamma)$
is trivial because $\varphi$ restricts to the identity automorphism of $\W_2^\Gamma$. Contradiction.
\end{proof}

Let us proceed to the case of $n>1$. We still have $W_{fin}$ acting on $\W(\D Q,\vv,\ww)_h$
such that $W_{fin}$ acts on $\widehat{\z}$ exactly as in $n=1$ case. We will see below
that $\Afr=W_{fin}\times \Z/2\Z$ and describe the action of $\Z/2\Z$. Now let us describe
a certain element of $\Afr$. We will see later that this element does not lie in $W_{fin}$.
Namely, we have an antiautomorphism $\tau$ of $\Halg$ given by $\tau(v)=v, \tau(g)=g^{-1}$
for $g\in \Gamma_n, \tau(h)=-h, \tau(k)=-k, \tau(c(\gamma))=-c(\gamma^{-1})$ for $\gamma\in \Gamma$.
Then $\tau$ fixes $e$ and so descends to $e\Halg e$. Also we note that $\tau$ is the identity on
$e\Halg e/(\param)$. So we get antiautomorphism $\tau':=\Upsilon\circ \tau\circ \Upsilon^{-1}$
of $\W_h(\D Q,\vv,\ww)$. Set $\nu:=\sigma\circ \tau'$, where $\sigma$ is the parity antiautomorphism
induced from $\W_h(V)$. So $\nu$ is a $\K[h]$-linear automorphism
of $\W_h(\D Q,\vv,\ww)$ that preserves $\widehat{\z}^*$ and is the identity modulo
$\widehat{\z}^*$. Moreover, the action of $\Upsilon^{-1}\circ \nu\circ \Upsilon$ on $\param/\K h$
coincides with that of $\tau$.

\subsection{Completions}\label{SUBSECTION:SRA_completions}
Below $n>1$.

So we have two tasks: to describe the restriction of $\Upsilon$ to $\param^*$ and to describe the group
$\Afr$. The first is our primary goal, but the second is also very important. We will see that we
are able to recover $\Upsilon|_{\param^*}$ only up to an element of $\Afr$.

To approach both these questions we will study isomorphisms of certain completions
induced by $\Upsilon$. The completions will have a form $(e\Halg e)^{\wedge_b}, \W(\D Q,\vv,\ww)_h^{\wedge_b}$
for certain points $b\in \K^{2n}/\Gamma_n$.

A general definition is as follows.
Let $A$ be an algebra such that its center $Z$ is finitely generated, and $A$ is finite over $Z$.
Further, let $\A$ be another algebra equipped with an epimorphism $\A\twoheadrightarrow A$.
Pick a point $b\in \Spec(Z)$ and let $\m_b$ be the ideal in $A$ generated by the maximal
ideal of $b$ in $Z$. Then consider the preimage $\widetilde{\m}_b$ of $\m_b$ in $\A$.
We put $\A^{\wedge_b}:=\varprojlim\A/\widetilde{\m}_b^n$.

In this subsection we will study the structure of the completions we need and in the next one we
will apply these results to accomplish the two goals mentioned above.

Until a further notice we consider the general SRA $\Halg$ corresponding to a space $V$
and a group $\Gr$.
Taking $A=SV\#\Gr, \A:=\Halg$ we get the ideals $\widetilde{\m}_b\subset \Halg,
e\widetilde{\m}_b e\subset e\Halg e$ and the corresponding completions $\Halg^{\wedge_b}, (e\Halg e)^{\wedge_b}$.
The following lemma implies that the completions $(e\Halg e)^{\wedge_b}$ and $e\Halg^{\wedge_b} e$ are naturally
identified.

\begin{Lem}\label{Lem:filtr_comp}
The filtrations $e \widetilde{\m}_b^n e$ and $(e\widetilde{\m}_b e)^n$ on $e\Halg e$ are compatible.
\end{Lem}
\begin{proof}
Let us remark that  $(e \widetilde{\m}_b e)^n \subset e \widetilde{\m}_b^n e$.
On the other hand, as a left $\Halg$-ideal, $\widetilde{\m}_b$ is generated by $\param$ and some
elements $x_1,\ldots,x_m$ of $\Halg$ such that each $x_i$ commutes with $\Gr\subset \Halg$
and $[x_i,x]\subset \param \Halg$ for any $x\in \Halg$. From this description it is easy to deduce that
for any $n$ the ideal $(e\widetilde{\m}_b e)^n$ contains $e\widetilde{\m}_b^N e$ for $N\gg 0$.
\end{proof}

The structure of the algebra $\Halg^{\wedge_b}$ was studied in \cite{sraco}.
Namely, let $\widetilde{b}$ be a point in the preimage of $b$ in $V$ and set $\underline{\Gr}:=\Gr_{\widetilde{b}}$. Then define the algebra $\underline{\Halg}$ as the quotient of $TV[\param]\#\underline{\Gr}$ by the relations
\begin{equation}\label{eq:SRA_relation1}
[u,v]=h \omega(u,v)+\sum_{i=1}^r c_i\sum_{s\in S_i\cap \underline{\Gr}}\omega_s(x,y)s,
\end{equation}
Following \cite{BE},  consider the {\it centralizer algebra} $Z(\Gr,\underline{\Gr}, \underline{\Halg}^{\wedge_0})$
that comes equipped with an embedding $\K\Gr\hookrightarrow Z(\Gr,\underline{\Gr},\underline{\Halg}^{\wedge_0})$.
What we need to know about this algebra is that there is a $\K[\param]\Gr$-linear isomorphism
$\Halg^{\wedge_b}\xrightarrow{\sim} Z(\Gr,\underline{\Gr},\underline{\Halg}^{\wedge_0})$,
\cite{sraco}, see Theorems  1.2.1, 2.7.3 and that $e Z(\Gr,\underline{\Gr},\underline{\Halg}^{\wedge_0}) e$
is naturally identified with $\underline{e}\underline{\Halg}^{\wedge_0}\underline{e}$,
\cite{BE}, Lemma 3.1, where $\underline{e}$ is the trivial idempotent in $\K \underline{\Gr}$. So we get a $\K[\param]$-linear isomorphism $e\Halg^{\wedge_b}e\cong
\underline{e}\underline{\Halg}^{\wedge_0}\underline{e}$.

The algebra $\underline{\Halg}^{\wedge_0}$ can be decomposed into a completed tensor product as follows.
There is a unique $\underline{\Gr}$-stable decomposition $V=V^{\underline{\Gr}}\oplus V^+$.
Consider the Weyl algebra $\W_{V^{\underline{\Gr}},h}$ and the the algebra $\underline{\Halg}^+$ that is the quotient
of $TV^+[\param]\#\underline{\Gr}$ by  relations (\ref{eq:SRA_relation1}). It is clear that $\underline{\Halg}=
\W_{V^{\underline{\Gr}},h}\otimes_{\K[h]}\underline{\Halg}^+$. So $\underline{\Halg}^{\wedge_0}=
\W_{V^{\underline{\Gr}},h}^{\wedge_0}\widehat{\otimes}_{\K[[h]]}\underline{\Halg}^{+\wedge_0}$.
Summarizing we get a $\K[\param]$-linear isomorphism
\begin{equation}\label{eq:iso_SRA1}
(e\Halg e)^{\wedge_b}\cong \W_{V^{\underline{\Gr}},h}^{\wedge_0}\widehat{\otimes}_{\K[[h]]}\underline{e}\underline{\Halg}^{+\wedge_0}\underline{e}.
\end{equation}

Let us proceed to the completions of quantum Hamiltonian reductions.
Let $V$ be a symplectic vector space and $G$ be a reductive group
acting on $V$ by linear symplectomorphisms. Let $\mu:V\rightarrow \g^*$
denote the moment map. Pick a point $b\in V\red_0 G$.
We are interested in the structure of $(\W_h(V)\red G)^{\wedge_b}$.
Let $x\in V$ be a point from a unique closed orbit in the fiber
of $b$. Automatically, $\mu(x)=0$. Set $H:=Gx$ and let $U$ denote the symplectic part of the normal
space to $Gx$ at $x$, i.e., $U:=V/(T_xGx)^{\skewperp}$ (since $\mu(x)=0$, we see that $T_xGx$
is an isotropic subspace in $V$). Then $U$ is a symplectic $H$-module.
Set $\z:=(\g^*)^G, \z':=(\h^*)^H$. We have a natural (restriction) map
$\z\rightarrow \z'$.

\begin{Lem}\label{Lem:Ham_red_completion}
We have a $\K[[\z,h]]$-linear isomorphism $$(\W_h(V)\red G)^{\wedge_b}\cong \K[[\z,h]]\widehat{\otimes}_{\K[[\z',h]]}(\W_h(U)\red H)^{\wedge_0}.$$
\end{Lem}
\begin{proof}
Set $Y=(T^*G\times U)\red_0 H$, where $H$ acts on $T^*G$ from the right. This is a Hamiltonian
$G$-variety  that is naturally isomorphic to the model variety $M_G(H,0,U)$ from \cite{slice}.
Let $y\in Y$ be a point corresponding
to the orbit of $(1,0,0)\in T^*G\times U$. According to the main result of \cite{slice},
the Hamiltonian formal $G$-schemes $V^{\wedge_{Gx}}, Y^{\wedge_{Gy}}$ are isomorphic.
By Proposition \ref{Prop:even_reduction}, the quantization
$(\Dcal_h(G)\otimes_{\K[h]}\W_h(U))\red_0 H$ of $Y$ is even. Also it is graded and therefore canonical.
It follows that the topological algebras
$\W_h(V)^{\wedge_{Gx}}, [(\Dcal_h(G)\otimes_{\K[h]}\W_h(U))\red_0 H]^{\wedge_{Gy}}$
are $G$-equivariantly isomorphic.
Furthermore, similarly to Theorem 3.3.4 in \cite{Miura}, we see that there is a $G$-equivariant
$\K[h]$-linear isomorphism
$\W_h(V)^{\wedge_{Gx}}\cong [(\Dcal_h(G)\otimes_{\K[h]}\W_h(U))\red_0 H]^{\wedge_{Gy}}$
intertwining the quantum comoment maps. It follows that the topological $\K[[\z,h]]$-algebras
$\W_h(V)^{\wedge_{Gx}}\red G, [(\Dcal_h(G)\otimes_{\K[h]}\W_h(U))\red_0 H]^{\wedge_{Gy}}\red G$
are isomorphic. Similarly to the proof of Proposition \ref{Prop:even_reduction},
we see that  $\W_h(V)^{\wedge_{Gx}}\red G\cong (\W_h(V)\red G)^{\wedge_b}$ and
$$[(\Dcal_h(G)\otimes_{\K[h]}\W_h(U))\red_0 H]^{\wedge_{Gy}}\red G\cong
[(\Dcal_h(G)\otimes_{\K[h]}\W_h(U))\red_0 H\red G]^{\wedge_0}.$$ So it remains
to identify $(\Dcal_h(G)\otimes_{\K[h]}\W_h(U))\red_0 H\red G$ with
$\K[\z,h]\otimes_{\K[\z',h]}\W_h(U)\red H$.

First of all, we note that, by the definition of a quantum  Hamiltonian reduction,
$(\Dcal_h(G)\otimes_{\K[h]}\W_h(U))\red_0 H\red G\cong
(\Dcal_h(G)\otimes_{\K[h]}\W_h(U))\red G\red_0 H$. But it is easy to
see that   $(\Dcal_h(G)\otimes_{\K[h]}\W_h(U))\red G$ is naturally identified
with  $\Dcal_h(G)\red G\otimes_{\K[h]}\W_h(U)=(\K[\z][h]\otimes_{\K[h]}\W_h(U))\red_0 H$,
where the quantum comoment map $\h\rightarrow \K[\z][h]\otimes_{\K[h]}\W_h(U)$
 has the form $\xi\mapsto -\rho(\xi)\otimes 1+1\otimes\Phi(\xi)$, $\rho$ being
the natural projection $\h\rightarrow \z^*$ and $\Phi:\h\rightarrow \W_h(U)$
being the quantum comoment map for the action of $H$ on $U$. This implies that
$(\Dcal_h(G)\otimes_{\K[h]}\W_h(U))\red_0 H\red G$ and
$\K[\z,h]\otimes_{\K[\z',h]}(\W_h(U)\red H)$   are $\K[\z,h]$-linearly isomorphic.
\end{proof}

Now let us specify what choices of $b\in \K^{2n}/\Gamma_n\cong \M_0(\D Q,\vv,\ww)$
we need.

The stratification of $\K^{2n}/\Gamma_n$ by the symplectic leaves is the same
as the stabilizer stratification. In particular, there are two symplectic leaves of codimension 2:
the first one, $\mathcal{L}_{\Gamma}$ corresponds to the first copy of $\Gamma$ inside  $\Gamma^n\subset \Gamma_n$,
while the second one, $\mathcal{L}_{sym}$, -- to the subgroup of order 2
generated by the transposition $s_{12}\in S_n\subset \Gamma_n$. In other
words, $\overline{\mathcal{L}}_\Gamma,\overline{\mathcal{L}}_{sym}$ are the images of $\{0,u_2,\ldots,u_n\},\{u_1,u_1,u_3,\ldots,u_n\}$ in $\K^{2n}/\Gamma_n$.

We will need points $b^1\in \mathcal{L}_{\Gamma}, b^2\in \mathcal{L}_{sym}$. The algebra
$\underline{\Halg}^{+1}$ constructed from $b^1$ is just $\K[k]\otimes \Halg(\Gamma)$, where $\Halg(\Gamma)$
is the SRA constructed from $\Gamma$ (over $\K[h,c_1,\ldots,c_r]$). Similarly,
the algebra $\underline{\Halg}^{+2}$ constructed from $b^2$ is $\K[c_1,\ldots,c_r]\otimes \Halg(S_2)$,
where the algebra $\Halg(S_2)$ is constructed from $S_2$ (over $\K[k]$).

Now let us consider the Hamiltonian reduction side. Recall that we reduce the space $V:=R(\D Q^{\ww},\vv^{\ww})$
by the action of $G:=\GL(\vv)$.
Pick an element $x\in R(\D Q,
\vv)$ that is decomposed into the sum $\bigoplus_{i=1}^{n-1} x_i\oplus 0$, where $x_i$ are generic
non-isomorphic elements of $R(\D Q,\delta)$ mapping to 0 under the moment map. Since the representation
of $x$ is semisimple, its $G$-orbit is closed. The stabilizer $H:=G_x$ is naturally
identified with $\GL(\delta)\times \K^{\times (n-1)}$.  The symplectic part $U$ of the slice module $T_x V/ \g x$ is identified with $R(\D Q^{\ww},\delta^{\ww})\oplus \K^{2n-2}\oplus \K^{2n-2}$, where the stabilizer $H$ acts via the projection $H\twoheadrightarrow \GL(\delta)$ on the first summand, trivially on the second one, and via the projection $H\twoheadrightarrow \K^{\times (n-1)}$ on the third one.

It is easy to see that $x$ maps into  $\mathcal{L}_\Gamma$ under the quotient map (in fact, any point of $\mathcal{L}_{\Gamma}$ is obtained in this way but we will not need this fact). So we may assume that
$b^1$ coincides with the image of $x$.

Now let us study the structure of the completion $(\W_h(V)\red G)^{\wedge_b}$ in more detail.
By Lemma \ref{Lem:Ham_red_completion}, $(\W_h(V)\red G)^{\wedge_b}\cong \K[[\z,h]]\widehat{\otimes}_{\K[[\z',h]]}
(\W_h(U)\red H)^{\wedge_0}$. Recall the basis $\check{\epsilon}_i,i=0,\ldots,r$ in $\z$. Let $\check{\epsilon}'_i,i=0,\ldots,r$
be the analogous basis in $\gl(\delta)/[\gl(\delta),\gl(\delta)]\subset \z'$. Also let $\lambda_1,\ldots,\lambda_{n-1}$
be the elements in $\z'$ corresponding to the $n-1$ copies of $\K^\times$. The map $\z\rightarrow \z'$
is given by $\epsilon_i\mapsto \epsilon_i'+\delta_i\sum_{j=1}^{n-1} \lambda_j$ and, in particular, is an embedding.

\begin{Lem}\label{Lem:ham_red_iso1} There is a
$\K[\z]$-linear isomorphism $\K[\z]\otimes_{\K[\z_0]}\W_{\z_0}(\D Q, \delta,\ww)_h\cong \K[\z,h]\otimes_{\K[\z',h]}\W_h(U)\red H$
of graded algebras.
\end{Lem}
\begin{proof}
Let $\z_0'\subset \z'$ be defined analogously to $\z_0\subset \z$. The spaces $\z_0'$ and $\z_0$
are naturally identified. We need to show that $\W_{\z_0}(\D Q,\delta,\ww)_h\cong \W_h(U)\red_{\z_0} H$.
This follows from the observation that $U\red_{\z_0}Z(H)^\circ$ is  naturally identified with
$R(\D Q,\vv)$.
\end{proof}

So we see that $\W(\D Q,\vv,\ww)_h^{\wedge_{b^1}}\cong \W_{2n-2,h}^{\wedge_0}\widehat{\otimes}_{\K[[h]]} (\K[\z]\otimes_{\K[\z_0]}\W_{\z_0}(\D Q,\delta,\ww)^{\wedge_0}_h)$.

Now let us proceed to the second leaf.
Pick an element $x\in R(\D Q,
n\delta)$ that is decomposed into the sum $\bigoplus_{i=1}^{n-2} x_i\oplus x_{n-1}^{\oplus 2}$, where $x_i, i=1,\ldots,n-1,$ are generic
non-isomorphic elements of $R(\D Q,\delta)$ mapping to 0 under the moment map.
The stabilizer $H:=G_x$ is naturally
identified with $\GL(2)\times \K^{\times (n-2)}$. The symplectic part $U$ of the slice module $T_x V/ \g x$ is identified with $$(\End(\K^2)\oplus \End(\K^2)^*\oplus \K^2\oplus \K^{2*})\oplus (\K^{n-2}\oplus \K^{n-2*})
\oplus \K^{2n-2}$$ (where $\K^2$ should be
viewed as the multiplicity space for the representation $x_{n-1}$). The stabilizer $H$ acts via the projection $H\twoheadrightarrow \GL(2)$ on the first summand (denote it by $U_1$), via the projection $H\twoheadrightarrow \K^{\times (n-2)}$ on the second one, and trivially on the  third one.

Again, it is easy to see that $x$ maps into  $ \mathcal{L}_{sym}$ under the quotient map  and hence we may assume that
$b^2$ coincides with the image of $x$. Set $\z_1:=\gl(2)^{*\GL(2)}$. We have a basis $\lambda,\lambda_1,\ldots,\lambda_{n-2}$ in $\z'$, where $\lambda$ corresponds to the determinant of
$\GL(2)$ and $\lambda_1,\ldots,\lambda_{n-2}$ correspond to the $n-2$ copies of $\K^\times$.
A natural map $\z\rightarrow \z'$ is given by $\epsilon_i\mapsto \delta_i(\lambda+\sum_{i=1}^{n-2}\lambda_i)$.
Again, $\K[\z]\otimes_{\K[\z']}\W_h(U)\red H=
\K[\z]\otimes_{\K[\z_1]}\W_{h}(U_1)\red \GL(2)$, where the map $\z\rightarrow \z_1$ is given by
$\chi\mapsto (\delta\cdot \chi)\lambda$.

\subsection{Reduction to the Kleinian case}\label{SUBSECTION_red_to_Klein}
The following proposition constitutes a reduction procedure.
\begin{Prop}\label{Prop:Kleinian_reduction}
There are isomorphisms $\Upsilon_1:\K[k]\otimes \underline{e}_1\underline{\Halg}^{+1}\underline{e}_1
\rightarrow \W(\D Q,\delta,\ww)_h$ and $\Upsilon_2:\K[c_1,\ldots,c_r]\otimes \underline{e}_2\Halg(S_2)\underline{e}_2
\rightarrow \K[\z]\otimes_{\K[\z_1]}\W_{h}(U_1)$ such that $\Upsilon_1|_{\param^*},\Upsilon_2|_{\param^*}=\Upsilon|_{\param^*}$.
\end{Prop}
\begin{proof}
Let us prove the existence of $\Upsilon_1$.

{\it Step 1.}

Pick a point $b_1$ as in the previous subsection. The isomorphism $\Upsilon:e\Halg e\rightarrow \W(\D Q,\vv,\ww)_h$
induces the isomorphism of completions $\Upsilon: e\Halg^{\wedge_{b_1}} e\rightarrow \W(\D Q,\vv,\ww)_h^{\wedge_{b_1}}$. As we have seen in the previous subsection we have isomorphisms
\begin{align}\label{eq:iso21}
e\Halg^{\wedge_{b_1}}e\cong\W_{2n-2,h}^{\wedge_0}\widehat{\otimes}_{\K[[h]]}\K[[k]]\widehat{\otimes}
\underline{e}_1\underline{\Halg}^{+1\wedge_0}\underline{e}_1,\\\label{eq:iso22}
\W(\D Q,\vv,\ww)^{\wedge_{b_1}}_h\cong \W_{2n-2,h}^{\wedge_0}\widehat{\otimes}_{\K[[h]]}\W(\D Q,\delta,\ww)_h^{\wedge_0}.
\end{align}
So we get an isomorphism $\Upsilon^{\wedge_{b_1}}$ of the right hand sides. We are going to show that
there is an isomorphism \begin{equation}\label{eq:iso4}\Upsilon_1^{\wedge}: \K[[k]]\widehat{\otimes}
\underline{e}_1(\underline{\Halg}^{+1})^{\wedge_0}\underline{e}_1\rightarrow
\W(\D Q,\delta,\ww)^{\wedge_0}_h\end{equation}
that coincides with $\Upsilon^{\wedge_{b_1}}$ on $\param^*$.
This follows from the next step.

{\it Step 2.}
Let $Z$ be a commutative algebra, and
let $\A_h$ be a $Z[[h]]$-algebra such that:
\begin{itemize}
\item $\A_h$ is flat over $\K[h]$.
\item $\A_h$ is complete in the $\m$-adic topology, where $\m$ is a maximal ideal of $\A_h$ containing $h$
and such that $\A_h/ \m=\K$.
\item $\A_h/h\A_h$ is commutative and Noetherian.
\end{itemize}
Let $\iota_1,\iota_2:\W_{2n-2,h}^{\wedge_0}\rightarrow \A_h$ be  $\K[[h]]$-linear homomorphisms such that
the  images of the maximal ideal of $\W_{2n-2,h}^{\wedge_0}$ lie in $\m$.
There is a $Z[h]$-linear automorphism $A$ of $\A_h$ such that
$\iota_1=A\circ \iota_2$.

Let us prove this claim. Using an easy induction, we reduce the proof to the case when $n=2$. So we have
elements $u_i,v_i\in \A_h, i=1,2,$ with $[u_i,v_i]=h$. First of all, we notice that
$\A_h$ is decomposed into the completed tensor product  \begin{equation}\label{eq:tens_decomp} \W_h(U)^{\wedge_0}\widehat{\otimes}_{\K[[h]]}\underline{\A}_h,\end{equation} where
 $U\subset \A_h$ with $u_1,v_1\in U$ and $\underline{\A}_h$ is some subalgebra of $\A_h$ such that the maximal ideal in $\underline{\A}_h/ h\underline{\A}_h$ is Poisson, compare with \cite{Miura}, Subsection 7.2.
In particular, for $a,b\in \underline{\A}_h$ we have $[a,b]\in h \m$. From here it is easy to
see that $u_2,v_2\in U+\m^2$ and modulo $\m^2$ the elements $u_2,v_2$
are linearly independent. So applying the automorphism
of $\A_h$ induced by a linear symplectomorphism of $U$, we may assume that $u_1\equiv u_2, v_1\equiv v_2\mod \m^2$.

Now we claim that there is an element $a\in \m^3$ such that $$\exp(\frac{1}{h}\ad(a))u_1=u_2, \exp(\frac{1}{h}\ad(a))v_1=v_2$$ (the series in the left hand sides automatically converges because $a\in \m^3$). First of all, from the fact that $u_1$ is a generator in the Weyl algebra
and the decomposition (\ref{eq:tens_decomp}), we can deduce that
the map $\frac{1}{h}\ad(u_1): \m^{i+1}\rightarrow \m^i$ is surjective. So if $u_2-u_1\in \m^i$, then there is  $a_1\in \m^{i+1}$ such that $u_2-u_1=\frac{1}{h}[a_1,u_1]$. Now $u_2-\exp(\frac{1}{h}a_1)u_1\in \m^{i+1}$. Thanks to the Campbell-Hausdorff
theorem, we can use an induction to prove  that there is $a'\in \m^3$ with $\exp(\frac{1}{h}\ad(a'))u_1=u_2$.
So we may assume that there is $u_1=u_2$.

Now we need to show that there is $a''$ with $[a'',u_1]=0$ and $\exp(\frac{1}{h}\ad(a''))v_1=v_2$.
Let us note that $[u_1,v_2-v_1]=0$.  In other words, $v_2-v_1$ is expressed as a (non-commutative)
power series of $u_1$, the elements of the skew-orthogonal complement to $u_1,v_1$ in $U$,
and the elements of $\underline{\A}_h$. Now we can find $a''$ using the argument of the previous
paragraph.

{\it Step 3.} So we have proved the existence of an isomorphism in (\ref{eq:iso4}).
Now we are going to show that there is an isomorphism \begin{equation}\label{eq:isom3}\Upsilon_1: \K[k]\otimes
\underline{e}_1\underline{\Halg}^{+1}\underline{e}_1\rightarrow \W(\D Q,\delta,\ww)_h\end{equation} that coincides with
$\Upsilon_1^\wedge$ on $\param^*$.
On both algebras in (\ref{eq:iso4}) there are $\K^\times$-actions such that the algebras of
$\K^\times$-finite vectors coincide with the algebras in (\ref{eq:isom3}).

So we only need to prove
the following claim:

Let $(\K[k]\otimes
\underline{e}_1\underline{\Halg}^{+1}\underline{e}_1)^{\wedge_0}$
be equipped with a $\K^\times$-action such that
\begin{itemize}
\item[(i)]
$t. \alpha=t^2\alpha$ for any $\alpha\in \param, t\in \K^\times$
\item[(ii)] and the induced action on $\K[\K^2]^{\Gamma\wedge_0}=(\K[k]\otimes
\underline{e}_1\underline{\Halg}^{+1}\underline{e}_1)^{\wedge_0}/(\param)$
is the standard one.
\end{itemize}
Then the algebra of $\K^\times$-finite vectors in $(\K[k]\otimes
\underline{e}_1\underline{\Halg}^{+1}\underline{e}_1)^{\wedge_0}$ is isomorphic to $\K[k]\otimes
\underline{e}_1\underline{\Halg}^{+1}\underline{e}_1$.
Moreover, the two gradings on the last algebra (the standard one, and one induced by the $\K^\times$-action)
differ by a compatible inner grading.

We say that a grading of  $\K[k]\otimes \underline{e}_1\underline{\Halg}^{1+}\underline{e}_1$ is compatible and inner if it is the grading by eigenvalues of a derivation of the form $\frac{1}{h}\ad(a)$, where
$a$ is an element of degree 2 with respect to the standard grading on this algebra.

Let $E,E_{st}$ be the derivations induced by the $\K^\times$-action under consideration and by
the standard $\K^\times$-action. Then $E-E_{st}$ is a $\K[\param]$-linear derivation of
$\K[k]\otimes \underline{e}_1\underline{\Halg}^{1+\wedge_0}\underline{e}_1$.
Every Poisson derivation of $(\K[\K^2]^{\Gamma})^{\wedge_0}=\K[[k]]\widehat{\otimes} \underline{e}_1(\underline{\Halg}^{1+})^{\wedge_0}\underline{e}_1/(\param)$ is inner.
It follows that $E-E_{st}=\frac{1}{h}\ad (a)$ for some $a\in \K[[k]]\widehat{\otimes} \underline{e}_1\underline{\Halg}^{1+\wedge_0}\underline{e}_1$, compare with Lemma 2.11.2 in \cite{sraco}.
The algebra $\K[k]\otimes \underline{e}_1\underline{\Halg}^{1+}\underline{e}_1$
has no elements of degree 1, so we can assume that $a=\sum_{i \geqslant 2} a_i$, where $a_i$ has degree $i$
with respect to the standard grading.

We claim that $a_2$ is central. Indeed, when $\Gamma$ is of type $D$ or $E$, the algebra $\K[\K^2]^\Gamma$
has no nonzero elements of degree 2.  Now consider the case when $\Gamma=\mathbb{Z}_n$. Choose
an eigenbasis $x,y\in \K^{2*}$ for $\Gamma$. The algebra $\K[\K^2]^\Gamma$ is generated by $x^n, y^n,xy$.
The subspace of elements of degree 2 in $\K[\K^2]^\Gamma$ is one-dimensional and
is spanned by $xy$. Therefore $a_2=cxy$ for some $c\in \K$.  Consider the actions of $E,E_{st}$ on the (three-dimensional) cotangent space to $0$ in $\K^2/\Gamma$. The eigenvalues of these operators are $n,n,2$ (because of (ii)) and for both the image of $xy$ has eigenvalue 2.
On the other hand, the eigenvalues of $E_{st}-\{c xy,\cdot\}$  equal $n+c,n-c,2$
and again $xy$ corresponds to the eigenvalue 2. But these eigenvalues coincides
with those for $E$. So we see that $c=0$ and so  $a_2$ is central.

So $E=E_{st}+\sum_{i\geqslant 3}\frac{1}{h}\ad(a_i)$. Since $E_{st}(\frac{1}{h}a_i)=(i-2)\frac{1}{h}a_i$, it follows  that there is $b\in \K[[k]]\widehat{\otimes} \underline{e}_1\underline{\Halg}^{+1\wedge_0}\underline{e}_1$
with $E= \exp(-\frac{1}{h} \ad b) E_{st}\exp(\frac{1}{h}\ad b)$. This implies the claim of this step
and completes the proof of the existence of $\Upsilon_1$.

The proof that $\Upsilon_2$ exists is completely analogous.
\end{proof}

The existence of $\Upsilon_1$ together with Proposition \ref{Prop:AutKleinian_main} imply
that  $\Upsilon$ maps $h,c_1,\ldots,c_r$ to $\widehat{\z}_0^*$ and there is $w\in W_{fin}$
such that $\Upsilon(c_i)=w\upsilon(c_i)$. Similarly the existence of $\Upsilon_2$
implies that $\Upsilon(k)=\pm \upsilon(k)$.

Define an action of $W_{fin}\times \Z/2\Z$ on $\widehat{\z}$ in the following way:
$w\in W_{fin}$ fixes $\delta$ and acts on $\widehat{\z}_0$ as before,
while a non-trivial element $\varsigma\in \Z/2\Z$ maps $\delta$ to $-\delta$
and is the identity on $\widehat{\z}_0$. From the previous paragraph we see that
$\Upsilon|_{\param^*}$ coincides with $g\upsilon$ for some $g\in W_{fin}\times \Z/2\Z$.

So it remains to check that the group $\Afr$ established in Subsection 6.4 coincides with $W_{fin}\times \Z/2\Z$.

First of all, let us show that $\Afr\subset W_{fin}\times \Z/2\Z$. Any element
$a\in \Afr$ acts on $\W(\D Q,\vv,\ww)^{\wedge_{b_1}}$ preserving $\widehat{\z}^*$
and coinciding with the identity modulo $\widehat{\z}^*$. Following  the proof
of Proposition \ref{Prop:Kleinian_reduction}, we see that there is an automorphism
of $\W(\D Q,\delta,\ww)_h$ that coincides with $a$ on $\widehat{\z}^*$.
Thanks to Proposition \ref{Prop:AutKleinian_main}, we see that $a$ preserves
$\widehat{\z}_0$ and acts on this space as an element of $W_{fin}$.  Similarly, we can
consider, $b_2$ instead of $b_1$ and get that $a$ preserves $\K\delta$ and acts on this
line by $\pm 1$. This proves the inclusion $\Afr\subset W_{fin}\times \Z/2\Z$.

Let us prove now that $\Afr=W_{fin}\times \Z/2\Z$.
Recall the automorphism $\nu\in \Afr$. It is enough to check that $\nu\not\in W_{fin}$.

Transport the action of $\Afr$ to $e\Halg e$ by means of $\Upsilon$, and the action of
$W_{fin}\times \Z/2\Z$ by means of $\upsilon$. Since $\Upsilon|_{\param^*}=g\circ \upsilon$
for some $g\in W_{fin}\times \Z/2\Z$, we see that the image of $\Afr$ in $\GL(\param^*)$
is still contained in $W_{fin}\times \Z/2\Z$. From the  description
of $\Upsilon^{-1}\circ \nu\circ \Upsilon$ given in the end of Subsection \ref{SUBSECTION_SRA_aut}
we see that this automorphism does not lie in
$W_{fin}$. Hence our claim.

This completes the proof of Theorem \ref{Thm_SRA_iso}.

\begin{Rem}\label{Rem:uniqueness}
It seems to be impossible to produce explicit formulas for the isomorphism $\Upsilon$ using our approach.
However, there is a uniqueness property for $\Upsilon$: this is a unique isomorphism satisfying
the conditions of Theorem \ref{Thm_SRA_iso}. This follows from
Lemma \ref{Lem:AutKleinian}.
\end{Rem}

\subsection{$S_l$-invariance property for the Euler element}
Here we consider the SRA $\Halg$ for the group $\Gamma_n$, where  $\Gamma=\Z/l\Z$.
Below $\h=\K^n$ is the reflection representation of $\Gamma_n$.
Recall that the Euler element ${\bf eu}\in \Halg$ is defined by
$${\bf eu}=\sum_{i=1}^n x_i y_i+\dim \h/2+\sum_{s\in S}\frac{c_s}{1-\lambda_s}s,$$
where $x_1,\ldots,x_n$ is a basis of $\h^*$, $y_1,\ldots,y_n$ is the dual basis of $\h$,
$c_s=c_i$ for $c\in S_i, i=1,\ldots,r$, $c_s=k$ for $s\in S_{sym}$, and $\lambda_s$ is the only non-unit eigenvalue of
$s$ in $\h^*$. Our formula looks different from a usual one, see, for instance, \cite{BE},
because our parameters $c_i$'s differ from ones used in the standard presentation of
the rational Cherednik algebras (by the factor of $-2$).

We remark that ${\bf eu}$ is independent of the choice of $x_1,\ldots,x_n$.
The reason to consider ${\bf eu}$ is that $[{\bf eu},x]=x, [{\bf eu},y]=-y$
for all $x\in \h^*, y\in \h$, where we consider $\h,\h^*$ as lagrangian subspaces in $\mathcal{V}=\K^{2n}$.
Also ${\bf eu}$ is $\Gamma_n$-invariant.

Recall that the group $W_{fin}$ acts on $e\widetilde{\Halg} e\cong \W(\D Q,\vv,\ww)_h$
as was explained in Subsection \ref{SUBSECTION_SRA_aut}. In our case $W_{fin}=S_l$.
The main result of this appendix to be used in a subsequent paper is the following proposition.

\begin{Prop}\label{Prop:eu_invariance}
The element ${\bf eu}^{sph}:=e{\bf eu}\in e\widetilde{\Halg}e$ is $S_l$-invariant.
\end{Prop}

One of the corollaries of this proposition is that the $S_l$-action commutes with the $\K^\times$-action on
$e\widetilde{\Halg}e$ induced from the action by algebra automorphisms on $\widetilde{\Halg}$ given by $t.x=tx, t.y=t^{-1}y, t.w=w,
t.h=h, t.c_i=c_i$, where $x\in \h^*, y\in \h, w\in W$. However, this can also  be easily deduced from Lemma \ref{Lem:AutKleinian}:
for each $a\in \Afr$ the automorphism $tat^{-1}$ also lies in $\Afr$ and the restrictions of $a$ and $tat^{-1}$
to $\widehat{\z}^*$ (or to $\widetilde{\param}^*$) coincide.


Our proof is basically in two steps. First we prove the invariance for a similar element  $\tilde{\bf eu}\in\W(\D Q,\vv,\ww)_h$
and then relate $e {\bf eu}$ with $\tilde{\bf eu}$.

Let us define $\tilde{\bf eu}$.
In our case the quiver $Q$ is the affine Dynkin quiver of type $A_{l-1}$. Consider the action of $\K^\times$
on $V$ given by $t.((A_i),(B_i), \Gamma_0,\Delta_0):=((tA_i),(t^{-1}B_i), \Gamma_0,\Delta_0)$. This
action commutes with $G$ and preserves the symplectic form on $V$. So it gives rise to a quantum
comoment map $\K\rightarrow \W_h(V)$ sending $1$ to the {\it symmetrized Euler element} $\tilde{\bf eu}\in \W_h(V)^G$. Abusing the notation, by $\tilde{\bf eu}$ we also denote the image of $\tilde{\bf eu}$ in $\W_h(V)\red G$.
Recall that $S_l$ acts on $\W_h(V)\red G$ by $\K[h]$-linear automorphisms.

\begin{Lem}\label{Prop:invariant}
The element $\tilde{\bf eu}\in \W_h(V)\red G$ is $S_l$-invariant.
\end{Lem}
\begin{proof}
The Hamiltonian $\K^\times$-action on $V$ considered in the previous paragraph
induces Hamiltonian actions on all reductions $\M^\theta(\D Q,\vv,\ww)$.
Inspecting the Maffei construction recalled in Subsection \ref{SUBSECTION_SRA_aut},
we see that the isomorphisms constructed there are equivariant with respect to the Hamiltonian
$\K^\times$-actions. It follows that  $S_l$ acts on $\W_h(V)\red G$ by
$\K^\times$-equivariant automorphisms. Now all $\sigma(\tilde{\bf eu}), \sigma\in S_l,$
are elements of degree 2 defining  quantum comoment maps for the $\K^\times$-action.
So they all differ from one another by elements of $\widehat{\z}^*$.

It is enough to prove that $\sigma(\tilde{\bf eu})-\tilde{\bf eu}\in \K h$ because
$h$ is $S_l$-invariant. This will follow if we check that the image of $\tilde{\bf eu}$
(also denoted by $\tilde{\bf eu}$) in $\K[\Lambda(\vv,\ww)]^{\GL(\vv)}$ is $S_l$-stable.
We may  assume that the quiver affine quiver $Q$ is oriented
counterclockwise and the framing is attached to vertex 0. Then $\tilde{\bf eu}=\sum_{i=0}^{l-1} \tr(A_i B_i)$.
The condition that the point $x=((A_i),(B_i),\Gamma_0,\Delta_0)$ lies in $\Lambda_\chi(\vv,\ww)$
is $B_0A_0-A_{l-1}B_{l-1}+\Gamma_0\Delta_0=\chi_0$ and $B_iA_i-A_{i-1}B_{i-1}=\chi_i$ for $i\neq 0$.
Let us prove that $s (\tilde{\bf eu})=\tilde{\bf eu}$ for the reflection $s$
constructed from the vertex $i\neq 0$.
Recall the variety $Z$ introduced in Subsection \ref{SUBSECTION_SRA_aut} and the projections
$\pi,\pi':Z\rightarrow \Lambda(\vv,\ww)$. We remark that in the definition of $Z$ we required
$i$ to be the source. With the orientation we have chosen this is not so and to fix this we need
to replace the pair $(A_{i-1},B_{i-1})$ with $(-B_{i-1},A_{i-1})$.
Condition (3) in the definition of $Z$ then implies (in the notation
of Subsection \ref{SUBSECTION_SRA_aut}) that $B_i'A_i'=B_iA_i-\chi_i, A_{i-1}'B_{i-1}'=A_{i-1}B_{i-1}+\chi_i$.
Also we remark that for $j\neq i,i-1$ we have $B_j'A_j'=B_jA_j$. So we have $\sum_{i}\tr(B_i'A_i')=\sum_i \tr(B_iA_i)$.
But this precisely means that $S^* (\tilde{\bf eu})=\tilde{\bf eu}$. Since this equality
holds for all generators of $S_l$, we see that $\tilde{\bf eu}$ is $S_l$-stable.
\end{proof}

Identify $\W(\D Q,\vv,\ww)_h$ and $e\widetilde{\Halg}e$ by means of the isomorphism $\Upsilon$ from Theorem \ref{Thm_SRA_iso}.
Again, for any $\sigma\in S_l$ the operators $[\sigma({\bf eu}^{sph}),\cdot],[{\bf eu}^{sph},\cdot]$
are the same because $S_l$ commutes with the $\K^\times$-action introduced above in this subsection. Since ${\bf eu}^{sph}$ is homogeneous of degree 2, this implies $\sigma({\bf eu}^{sph})-{\bf eu}^{sph}\in \widehat{\z}^*$.
From the explicit form of the correspondence between the parameters, see Theorem \ref{Thm_SRA_iso}, we see that $h,k$
are $S_l$-invariant. So it is enough to show that $\tilde{\bf eu}$ lies in the linear span of ${\bf eu}^{sph}$ and $1$
modulo the ideal generated by $h,k$.

The algebra $e\widetilde{\Halg}e/(h,k)$ is just $(\underline{e}\underline{\Halg}^+_0\underline{e}^{\otimes n})^{S_n}$,
where $\underline{\Halg}^+_0$ is the SRA for $\Gamma$ at $h=0$, and $\underline{e}$ is the trivial idempotent in this algebra.
On the other hand, $\W(\D Q,\vv,\ww)_h/(k,h)$ is the algebra of functions on the reduced scheme
$R(\D Q, n\delta)\red G'$. Here $G'$ is the quotient of $G$
by the subgroup $\{x \operatorname{1}, x\in \K^\times\}$. The induced isomorphism
$$(\underline{e}\underline{\Halg}^+_0\underline{e}^{\otimes n})^{S_n}
\cong \K[R(\D Q, n\delta)\red G']$$ can be described as follows.

First, consider the case $n=1$. Here $\Spec(\underline{e}\underline{\Halg}^+_0\underline{e})$ is the moduli
space parameterizing semisimple $\Halg^+_0$-modules that are $\Gamma$-equivariantly isomorphic
to $\K \Gamma$. To a module we assign the pair of operators $x\in \h^*, y\in \h$.
This assignment is known to  define an isomorphism  $\Spec(\underline{e}\underline{\Halg}^+_0\underline{e})\rightarrow
R(\D Q, \delta)\red G'$ coinciding with the isomorphism we need. The element ${\bf eu}^{sph}$ corresponds to
$xy+\frac{1}{2}+\sum_{s\in S}\frac{c_s}{1-\lambda_s}s$, where $x,y$ are basis
elements in $\h^*,\h$ subject to $\langle x,y\rangle=1$ and viewed as operators
on $\K \Gamma$. These operators are subject to the relation $[y,x]=\sum_{s\in S}c_s s$,
where $S$ is just $\Gamma\setminus \{1\}$. In other words, $x,y$ are just
$\ell$-tuples $(x^0,\ldots,x^{\ell-1}),(y^0,\ldots, y^{\ell-1})$ (that represent maps
between isotypic components of $\K\ZZ/\ell \ZZ$) subject to
\begin{equation}\label{eq:relation}
x^iy^i-x^{i-1}y^{i-1}=-\sum_{j=1}^{\ell-1}c_j \eta^{ji},
\end{equation}
where  $\eta$ is the primitive
$\ell$-th root of 1 defining the action of $\Gamma$ on $\h^*$. The element
${\bf eu}$ is central and so acts by the same scalar on all isotypic components.
The scalar equals $\frac{1}{\ell}\sum_{i=0}^{\ell-1}(x^iy^i+\frac{1}{2}+\sum_{j=1}^{\ell-1}\frac{c_j}{1-\eta^{j}}\eta^{ji})$.
Changing the summation order, we see that the last sum is just $\frac{1}{\ell}\sum_{i=0}^{\ell-1}x^iy^i+\frac{1}{2}=
\frac{1}{\ell}\tilde{\bf eu}+\frac{1}{2}$.

Now let us consider the case of arbitrary $n$. Here the isomorphism $e\Halg e/(k,h)\cong \K[R(\D Q,n\delta)\red G']$
is given as follows. Recall that
$$e\Halg e/(h,k)=(\underline{e}\underline{\Halg}^+_0\underline{e}^{\otimes n})^{S_n}=(\K[R(\D Q,\delta)\red G_0']^{\otimes n})^{S_n},$$
where $G_0'$ is the quotient of $\GL(\delta)$ analogous to $G'$. So we need to define a morphism
from $(R(\D Q,\delta)\red G_0')^n/S_n\rightarrow R(\D Q, n\delta)\red G'$. This morphism
just sends the $n$-tuple of representations $(z_1,\ldots,z_n)$ to their direct sum. The morphism
of interest is an isomorphism and the corresponding homomorphism of algebra is the required one. Under
 the isomorphism $(R(\D Q,\delta)\red G_0')^n/S_n\xrightarrow{\sim} R(\D Q, n\delta)\red G'$
 the element $\tilde{\bf eu}$ on the right hand side corresponds to the sum $\sum_{i=1}^n \tilde{\bf eu}_i$, where
 $\tilde{\bf eu}_i, i=1,\ldots,n,$ are analogous trace polynomials for the $n$ copies of
$\K[R(\D Q,\delta)\red G_0']$. But we also have ${\bf eu}=\sum_{i=1}^n {\bf eu}_i$. Now using the previous
paragraph we prove our claim for arbitrary $n$.


\begin{thebibliography}{99}
\bibitem[BE]{BE} R. Bezrukavnikov, P. Etingof. {\it Parabolic induction and restriction
functors for rational Cherednik algebras}.  Selecta Math.,  14(2009), 397-425.
\bibitem[BeK1]{BK1} R. Bezrukavnikov, D. Kaledin.
{\it Fedosov quantization in the algebraic context}.
Moscow Math. J. 4 (2004), 559-592.
\bibitem[BeK2]{BK2} R. Bezrukavnikov, D. Kaledin.
{\it McKay equivalence for symplectic quotient singularities}.
Proc. of the Steklov Inst. of Math. 246 (2004), 13-33.
\bibitem[BB]{BB} W. Borho, J.-L. Brylinski. {\it Differential operators
on homogeneous spaces}. Invent. Math. 69(1982), 437-476.
\bibitem[BrK]{BrK} J. Brundan, A. Kleshchev, \emph{Shifted Yangians and finite $W$-algebras},
Adv. Math. 200(2006), 136-195.
\bibitem[CB1]{CB1} W. Crawley-Boevey. {\it Geometry of the moment map
for  representations of quivers}. Compositio Math. 126(2001), 257-293.
\bibitem[CB2]{CB2} W. Crawley-Boevey. {\it Normality of Marsden-Weinstein
reductions for representations of quivers}.
\bibitem[CBH]{CBH} W. Crawley-Boevey, M. Holland. {\it Noncommutative deformations of Kleinian singularities}.
Duke Math. J. 92(1998), 605-635.
\bibitem[DH]{DH} J. Duistermaat, G. Heckman. {\it On the variation in the cohomology of
the reduced phase space}. Invent. Math. 69(1982), 259-268.
\bibitem[DK]{DK} C. Dodd, K. Kremnizer, \emph{A Localization Theorem for Finite W-algebras},
arXiv:0911.2210.
\bibitem[E]{Etingof} P. Etingof. {\it Calodgero-Moser systems and representation theory}.
Z\"{u}rich Lectures in Advanced Mathematics. EMS, Z\"{u}rich, 2007.
\bibitem[EG]{EG} P. Etingof, V. Ginzburg. {\it Symplectic reflection algebras, Calogero-Moser space,
and deformed Harish-Chandra homomorphism}. Invent. Math. 147 (2002), N2, 243-348.
\bibitem[EGGO]{EGGO} P. Etingof, W.L. Gan, V. Ginzburg, A. Oblomkov. {\it Harish-Chandra homomorphisms
and symplectic reflection algebras for wreath-products}. Publ. Math. IHES, 105(2007), 91-155.
\bibitem[F1]{Fedosov1} B. Fedosov. {\it A simple geometrical construction of
deformation quantization}, J. Diff. Geom. 40(1994), 213-238.
\bibitem[F2]{Fedosov2} B. Fedosov. {\it Deformation quantization and index theory},
in Mathematical Topics 9, Akademie Verlag, 1996.
\bibitem[F3]{Fedosov3} B. Fedosov. {\it Non-abelian reduction in deformation
quantization}. Lett. Math. Phys. 43(1998), 137-154.
\bibitem[GG1]{GG} W.L. Gan, V. Ginzburg, {\it Quantization of Slodowy slices}, IMRN, 5(2002), 243-255.
\bibitem[GG2]{GG_quiver} W.L. Gan, V. Ginzburg, {\it Almost commuting variety, $\Dcal$-modules and Cherednik algebras}.
\bibitem[Gi]{Ginzburg_HC} V. Ginzburg. {\it Harish-Chandra bimodules for quantized Slodowy
slices}.  Represent. Theory, 13(2009), 236-271.
\bibitem[GSV]{GSV} G. Gonzales-Sprinberg, J.-L. Verdier, {\it Construction g\'{e}om\'{e}trique de la correspondance de McKay}. Annales scientifiques de l'\'{E}cole Normale Sup\'{e}rieure, S\'{e}r. 4, 16 no. 3 (1983), p. 409-449.
\bibitem[Go]{Gordon} I. Gordon. {\it A remark on rational Cherednik algebras and differential operators
on the cyclic quiver}. Glasg. Math. J. 48(2006), 145-160.
\bibitem[GL]{GL} I. Gordon, I. Losev. {\it On category $\mathcal{O}$ for cyclotomic rational Cherednik algebras}. In preparation. 
\bibitem[GS]{GS}  V. Guillemin, Sh. Sternberg.  {\it Symplectic techniques in physics}. Cambridge
University Press, 1984.
\bibitem[H]{Holland} M. Holland. {\it Quantization of the Marsden-Weinstein reduction for extended
Dynkin quivers}. Ann. Sci. Ec. Norm. Super. IV Ser. 32(1999), 813-834.
\bibitem[KaVe]{KV} D. Kaledin, M. Verbitsky. {\it Period map for non-compact holomorphically
symplectic manifolds}. GAFA 12(2002), 1265-1295.
\bibitem[KaVa]{KaVa} M. Kapranov, E. Vasserot. {\it Kleinian singularities, derived categories
and Hall algebras}. Math. Ann. 316(2000), 565-576.
\bibitem[Ka]{Kawanaka} N. Kawanaka, {\it Generalized Gelfand-Graev representations and Ennola duality}, In
"Algebraic Groups and Related Topics", Advanced Studies in Pure Mathematics 6(1985), North-Holland, p. 175-206.
\bibitem[KP]{KP} H. Kraft, C. Procesi. {\it Closures of Conjugacy Classes of Matrices are Normal}. Invent.
Math. 53(1979), 227-247.
\bibitem[Kr]{Kronheimer} P. Kronheimer. {\it The construction of ALE spaces as hyper-K\"{a}hler quotients}.
J. Differential Geom. 29(1989), 665-683.
\bibitem[LBP]{LeBruynProcesi} L. LeBruyn, C. Procesi. {\it Semisimple representations
of quivers}. Trans. Amer. Math. Soc. 317(1990), 585-598.
 \bibitem[L1]{slice} I.V. Losev. {\it Symplectic slices for reductive groups}.  Mat. Sbornik 197(2006),
N2, 75-86 (in Russian). English translation in: Sbornik Math. 197(2006), N2, 213-224.
\bibitem[L2]{Wquant}  I.V. Losev. {\it Quantized symplectic actions and
$W$-algebras}. J. Amer. Math. Soc. 23(2010), 35-59.
\bibitem[L3]{HC} I. Losev. {\it Finite dimensional representations of W-algebras}.
Duke Math J. 159(2011), n.1, 99-143.
\bibitem[L4]{Miura} I. Losev. {\it 1-dimensional representations and parabolic induction
for $W$-algebras}. Adv. Math. 226(2011), 6, 4841-4883.
\bibitem[L5]{ES_appendix} I. Losev. Appendix to: P. Etingof, T. Schedler, {\it Poisson traces and D-modules on Poisson varieties},
GAFA 20(2010), n.4,  958-987.
\bibitem[L6]{sraco} I. Losev. {\it Completions of symplectic  reflection algebras}.
Preprint, arXiv:1001.0239.
\bibitem[L7]{ICM} I. Losev. {\it Finite W-algebras}. Proceedings of the International Congress of Mathematicians
Hyderabad, India, 2010, p. 1281-1307.
\bibitem[Ma1]{Maffei} A. Maffei. {\it Quiver varieties of type A.} Comment. Math. Helv. 80(2005), 1-27.
\bibitem[Ma2]{Maffei_Weyl} A. Maffei. {\it A remark on quiver varieties and Weyl groups}. Ann. Scuola Norm. Sup.
Pisa Cl. Sci. (5), 1(2002), 649-686.
\bibitem[MV]{MV} I. Mirkovic, M. Vybornov. {\it Quiver varieties and Beilinson-Drinfeld Grassmannians of type A}.
arXiv:0712.4160.
\bibitem[Mo]{Moeglin} C. Moeglin, {\it Mod\`{e}les de Whittaker et id\'{e}aux primitifs compl\`{e}tement premiers dans les alg\`{e}bres enveloppantes I}, C.R. Acad. Sci. Paris, S\'{e}r. I 303(1986), No. 17,  845-848.
\bibitem[N1]{Nakajima} H. Nakajima. {\it Instantons on ALE spaces, quiver varieties and Kac-Moody algebras}.
Duke Math. J. 76(1994), 365-416.
\bibitem[N2]{Nakajima_book} H. Nakajima. {\it Lectures on Hilbert schemes of points on surfaces}.
Univ. Lect. Series, 18, AMS, 1999.
\bibitem[O]{Oblomkov} A. Oblomkov. {\it Deformed Harish-Chandra homomorphism for the cyclic quiver}. Math Res. Lett.
14(2007), 359-372.
\bibitem[P]{Premet1} A. Premet. {\it Special transverse slices and their enveloping algebras}. Adv. Math. 170(2002),
1-55.
\bibitem[S]{Slodowy} P. Slodowy. {\it Simple singularities and simple algebraic
groups}. Lect. Notes Math., v.815. Springer, Berlin/Heidelberg/New
York, 1980.
\end{thebibliography}
\end{document}